\documentclass[11pt]{amsart}
\usepackage{amssymb,bm,mathrsfs}
\usepackage{mathabx}
\usepackage[all]{xy}
\usepackage{xcolor}
\usepackage{hyperref}
\usepackage{enumitem}
\usepackage[mathscr]{euscript}
\newcommand{\map}[1]{\xrightarrow{#1}}

\newcommand{\iso}{\cong}
\newcommand{\define}{\stackrel{\mathrm{def}}{=}}

\newcommand{\imes}{\ltimes}

\newcommand{\End}{\mathrm{End}}
\newcommand{\Spec}{\mathrm{Spec}}
\newcommand{\Q}{\mathbb Q}
\newcommand{\Z}{\mathbb Z}
\newcommand{\R}{\mathbb R}
\newcommand{\C}{\mathbb C}
\newcommand{\F}{\mathbb F}
\newcommand{\A}{\mathbb A}
\newcommand{\co}{\mathcal O}
\newcommand{\alg}{\mathrm{alg}}

\newcommand{\ord}{\mathrm{ord}}
\newcommand{\SL}{\mathrm{SL}}
\newcommand{\GL}{\mathrm{GL}}
\newcommand{\Sym}{\mathrm{Sym}}

\newcommand{\et}{\mathrm{et}}
\newcommand{\beef}{\sharp}
\newcommand{\leftcirc}{{}^{\circ}}

\begin{document}
\author{Benjamin Howard and Keerthi Madapusi}
\title[Kudla's modularity conjecture on integral models]{Kudla's modularity conjecture on integral models of orthogonal Shimura varieties}
\date{}

\address{Department of Mathematics\\Boston College\\ 140 Commonwealth Ave. \\Chestnut Hill, MA 02467, USA}
\email{howardbe@bc.edu}
\email{madapusi@bc.edu}

\thanks{B.H.~was supported in part by NSF grants  DMS-2101636 and DMS-1801905. K.M.~was supported in part by NSF grants DMS-220804 and DMS-1802169.}

\begin{abstract}
We construct a family of special cycle classes on the regular integral model of an orthogonal Shimura variety, and show that these cycle classes appear as Fourier coefficients of a Siegel modular form.  Passing to the generic fiber of the Shimura variety recovers a result of Bruinier and Raum, originally conjectured by Kudla.
\end{abstract}

\maketitle
\setcounter{tocdepth}{1}

\tableofcontents

\theoremstyle{plain}
\newtheorem{theorem}{Theorem}[subsection]
\newtheorem{bigtheorem}{Theorem}[section]

\newtheorem{proposition}[theorem]{Proposition}
\newtheorem{lemma}[theorem]{Lemma}
\newtheorem{corollary}[theorem]{Corollary}
\newtheorem{bigcorollary}[bigtheorem]{Corollary}

\theoremstyle{definition}
\newtheorem{definition}[theorem]{Definition}
\newtheorem{hypothesis}[theorem]{Hypothesis}
\newtheorem{bighypothesis}[bigtheorem]{Hypothesis}

\theoremstyle{remark}
\newtheorem{remark}[theorem]{Remark}
\newtheorem{example}[theorem]{Example}
\newtheorem{question}[theorem]{Question}

\numberwithin{equation}{section}
\renewcommand{\thebigtheorem}{\Alph{bigtheorem}}
\renewcommand{\thebigcorollary}{\Alph{bigcorollary}}
\renewcommand{\thebighypothesis}{\Alph{bighypothesis}}


\section{Introduction}


Throughout this paper, we denote by  $(V,Q)$  a quadratic space over $\Q$ of signature $(n,2)$  with $n\ge 1$.
Associated to $V$ is a Shimura datum $(G,\mathcal{D})$ with reflex field $\Q$, where  the reductive group $G=\mathrm{GSpin}(V)$ of spinor similitudes sits in  an exact sequence
\[
1 \to \mathbb{G}_m \to G \to \mathrm{SO}(V) \to 1.
\]
Fixing a  $\Z$-lattice $L\subset V$ on which  the quadratic form takes integral values determines a compact open subgroup $K \subset G(\A_f)$, and hence a smooth complex orbifold
\[
M (\C) = G (\Q) \backslash \mathcal{D}  \times G (\A_f) / K.
\]
By the theory of canonical models of Shimura varieties, these are the complex points of a 
smooth Deligne-Mumford stack  $M \to \Spec(\Q)$   of dimension $n$.

The Shimura variety $M$ carries special cycles of all codimensions,  whose arithmetic properties are the subject of a series of conjectures of Kudla \cite{kudla-msri}.  See also \cite{kudla97}, \cite{KR-hilbert}, \cite{KR-siegel}, and \cite{KRYbook}. 
The organizing principle of these conjectures is that the special cycles should behave like coefficients of the theta kernel used to lift automorphic forms from a symplectic group to an orthogonal group.

  In particular, the special cycles should themselves be, in a suitable sense,  the coefficients of a Siegel modular form.  
  This is a now a theorem of Bruinier and Raum, and the goal of this paper is to extend their modularity result to  special cycles on the canonical integral model of $M$.


\subsection{Modularity on the generic fiber}


For any integer $d\ge 1$, let $\Sym_d(\Q)$ be the set of symmetric $d\times d$ matrices with rational coefficients.  

Let $L^\vee \subset V$ be the dual lattice to $L$ under the bilinear form determined by $Q$.
To each  $T \in \Sym_d(\Q)$ and each tuple of cosets $\mu=(\mu_1,\ldots,\mu_d) \in (L^\vee / L)^d$, Kudla associates a special cycle
\[
Z(T,\mu) \to M 
\]
of  pure codimension $\mathrm{rank}(T)$.
The Shimura variety $M$ carries a distinguished line bundle $\omega$, called the \emph{tautological bundle} or the \emph{line bundle of weight one modular forms}, and we follow Kudla in using the intersection pairing in the Chow ring to define the \emph{corrected cycle class}
\begin{equation}\label{intro generic corrected}
C(T,\mu) = \underbrace{c_1(\omega^{-1}) \cdots c_1(\omega^{-1})}_{ d-\mathrm{rank}(T) }  \cdot Z(T,\mu) 
\in \mathrm{CH}^{d}(M)
\end{equation}
in the codimension $d$ Chow group.
Here $c_1(\omega^{-1}) \in \mathrm{CH}^1(M)$ is the first Chern class of $\omega^{-1}$.
These Chow groups, like all Chow groups appearing in this paper, are taken with $\Q$-coefficients.

The metaplectic double cover of $\mathrm{Sp}_{2d}(\Z)$ acts via the Weil representation $\omega_{L,d}$ on the finite dimensional $\C$-vector space  $S_{L,d}$ of functions  $ (L^\vee / L)^d \to \C$.
The dual representation  has a canonical basis $\{ \phi^*_\mu \}_\mu \subset S_{L,d}^*$ indexed by $d$-tuples $\mu$ as above, and so we may form
\[
C(T) = \sum_{ \mu \in (L^\vee / L)^d } C(T,\mu) \otimes \phi^*_\mu \in \mathrm{CH}^{d}(M) \otimes_\Q S_{L,d}^*.
\]

The following conjecture of Kudla was proved by Borcherds \cite{Bor:GKZ}  in the case of codimension $d=1$ 
 (and before that by Gross-Kohnen-Zagier \cite{GKZ} in the very special case where $M$ is a modular curve).
The general case was  proved   by  Bruinier and Raum   \cite{BWR},  using ideas from the thesis of W.~Zhang \cite{zhang-thesis} to reduce to the case $d=1$.

\begin{bigtheorem}[Bruinier-Raum]\label{BigThm:generic_modularity}
The formal generating series 
\[
\sum_{ T\in \Sym_d(\Q) } C(T) \cdot q^T
\]
converges to a Siegel modular form of weight $\frac{n}{2} +1 $ and representation 
\[
\omega_{L,d}^* : \mathrm{Sp}_{2d}(\Z) \to \GL( S_{L,d}^* ) .
\]
Convergence and modularity are understood in the following sense: for any $\Q$-linear functional $\iota : \mathrm{CH}^{d}(M) \to \C$, the formal generating series
\[
\sum_{ T\in \Sym_d(\Q) }  \iota( C(T))  \cdot q^T
\]
with coefficients in $S_{L,d}^*$ is the $q$-expansion of a holomorphic Siegel modular form  of the stated weight and representation.
\end{bigtheorem}

Strictly speaking, the results of Borcherds and Bruinier-Raum apply to the Chow group of the complex fiber $M(\C)$, not the canonical model  over $\Q$. 
The proof for the canonical model is the same, using the fact that all Borcherds products on $M(\C)$ are algebraic and defined over the field of rational numbers  \cite{HMP}.
 In any case, Theorem \ref{BigThm:generic_modularity} in the form stated here is a consequence of our main result,  Theorem \ref{BigThm:integral_modularity} below.


\subsection{Modularity on the integral model}


Throughout the paper we work with a finite set of primes $\Sigma$ containing all primes $p$ for which the lattice $L_p$ is not maximal (Definition \ref{def:max lattice}), and abbreviate 
\[
\Z[\Sigma^{-1}] = \Z[p^{-1}:\, p\in \Sigma].  
\]
In \cite{HMP} one finds the construction of a normal and flat Deligne-Mumford stack 
\[
\mathcal{M} \to \Spec(  \Z[\Sigma^{-1}]  )
\]
with generic fiber $M$.  
Soon we will impose stronger assumptions on $\Sigma$,  to guarantee that $\mathcal{M}$ is regular.

For each  $T \in \Sym_d(\Q)$ and $\mu \in (L^\vee / L)^d$ we define a \emph{naive special cycle}
\[
\mathcal{Z}(T,\mu) \to \mathcal{M} 
\]
whose generic fiber agrees with Kudla's $Z(T,\mu)$.  Our definition of this cycle is via a moduli interpretation.
The integral model carries a \emph{Kuga-Satake abelian scheme} $\mathcal{A} \to \mathcal{M}$ whose pullback to any $\mathcal{M}$-scheme   $S \to \mathcal{M}$ has a distinguished $\Z$-submodule
\[
V(\mathcal{A}_S) \subset \End(\mathcal{A}_S)
\]
of \emph{special endomorphisms}.  
The space of special endomorphisms is endowed with a positive definite quadratic form, and  the $S$-points of  $\mathcal{Z}(T,\mu)$ are  in bijection with $d$-tuples $x=(x_1,\ldots, x_d) \in V(\mathcal{A}_S)_\Q^d$ of special quasi-endomorphisms  with moment matrix $Q(x)=T$,  whose denominators are controlled (in a precise sense) by the tuple $\mu=(\mu_1,\ldots, \mu_d)$.
For example, if $\mu_i=0$ then $x_i \in V(\mathcal{A}_S)$.

We insist on a modular definition of $\mathcal{Z}(T,\mu)$, as opposed to simply taking the Zariski closure of $Z(T,\mu)$ in the integral model, because this is necessary to ensure that the special cycles  behave well under intersections and pullbacks to smaller orthogonal Shimura varieties (as in Theorems \ref{BigThm:intro intersection} and \ref{BigThm:intro pullback} below).

This insistence comes with a high cost: the naive special cycles need not be flat over $\Z[\Sigma^{-1}]$, and need not  be equidimensional.  
Although those irreducible components of $\mathcal{Z}(T,\mu)$  that are flat over $\Z[\Sigma^{-1}]$ have  codimension $\mathrm{rank}(T)$ in $M$, there will often be irreducible components of the wrong codimension supported in nonzero characteristics.

The intuition behind this phenomenon is easy to explain.  At a characteristic $p$ geometric point  $s\to \mathcal{M}$ at which $\mathcal{A}_s$ is supersingular, the rank of the space of special endomorphisms $V(\mathcal{A}_s)$ is as large as it can be (namely, $n+2$).
This is large enough  that if the entries of $T$ are integral and highly divisible, the entire supersingular locus of $\mathcal{M}_{\F_p}$ will be  contained in $\mathcal{Z}(T,\mu)$.
It is known \cite{HP}  that this supersingular locus has dimension roughly $n/2$, and so the naive cycles $\mathcal{Z}(T,\mu)$ tend to have vertical irreducible components of dimension $>n/2$, regardless of the rank  of $T$.
For this reason,  one cannot construct cycle classes on $\mathcal{M}$ simply by imitating the construction \eqref{intro generic corrected}.

\begin{bighypothesis}
For the remainder of the introduction we assume that $\Sigma$ satisfies the hypotheses of Proposition \ref{prop:regularity}, guaranteeing that  $\mathcal{M}$ is regular.  If, for example, the discriminant of $L$ is odd and squarefree then $\Sigma=\emptyset$ satisfies these hypotheses.
\end{bighypothesis}

We will  construct  \emph{corrected} (or perhaps \emph{derived})  cycle classes
\[
\mathcal{C}(T,\mu) \in \mathrm{CH}^d(\mathcal{M}) 
\]
for all  integers   $d\ge 1$,  all $T \in \Sym_d(\Q)$,  and all $\mu \in (L^\vee / L)^d$.
These cycle classes vanish unless $T$ is positive semi-definite.
In \S \ref{ss:naive comparison} we prove the following result, showing that our construction is compatible with the classes already constructed in the generic fiber.

\begin{bigtheorem}\label{BigThm:intro comparison}
Restricting $\mathcal{C}(T,\mu)$ to the generic fiber recovers  \eqref{intro generic corrected}.   
Moreover,  if  the naive cycle  $\mathcal{Z}(T,\mu)$ is equidimensional of codimension $\mathrm{rank}(T)$ in $\mathcal{M}$, then 
\[
\mathcal{C}(T,\mu)  =  \underbrace{c_1(\omega^{-1}) \cdots c_1(\omega^{-1})}_{ d-\mathrm{rank}(T) }  \cdot \mathcal{Z}(T,\mu) 
\in \mathrm{CH}^{d}(\mathcal{M})
\]
for a distinguished line bundle $\omega$ on $\mathcal{M}$.
\end{bigtheorem}

The next two results show that our corrected cycle classes behave well under intersections and pullbacks to smaller Shimura varieties. Analogous  formulas in the generic fiber are proved in \cite{YZZ} and  \cite{kudla21}.

The following is stated in the text as Proposition \ref{prop:intersection formula}.

\begin{bigtheorem}\label{BigThm:intro intersection}
For  all positive integers $d'$ and $d''$,  symmetric matrices  
\[
T' \in \Sym_{d'}(\Q) \quad\mbox{and}\quad T'' \in \Sym_{d''}(\Q),
\]
 and tuples 
$\mu' \in (L^\vee/L)^{d'}$ and  $\mu'' \in (L^\vee / L)^{d''}$, we have the intersection formula
\[
\mathcal{C} (T',\mu') \cdot \mathcal{C} (T'',\mu'') 
= \sum_{ T = \left(\begin{smallmatrix}  T' & * \\ * & T''  \end{smallmatrix} \right)   } \mathcal{C}(T,\mu)
\]
in the codimension $d'+d''$ Chow group of $\mathcal{M}$. 
On the right hand side  $\mu = ( \mu',\mu'')$ is the concatenation of  $\mu'$ and $\mu''$.
\end{bigtheorem}

Now fix a positive definite self-dual quadratic lattice $\Lambda$, so that the orthogonal direct sum 
\[
L^\beef = L \oplus \Lambda
\]
has signature $(n+ \mathrm{rank}(\Lambda) , 2)$. 
This lattice determines its own Shimura datum, its own regular  integral model $\mathcal{M}^\beef$ over  $\Z[\Sigma^{-1}]$, and its own family of corrected special cycle classes
\[
 \mathcal{C}^\beef(T^\beef, \mu^\beef)\in \mathrm{CH}^d( \mathcal{M}^\beef)
\]
indexed by $T^\beef \in \Sym_d(\Q)$ and $\mu^\beef \in (L^{\beef ,\vee} / L^\beef )^d$. 
The isometric embedding $L \to L^\beef$ determines a finite and unramified morphism $f : \mathcal{M} \to \mathcal{M}^\beef$, inducing a pullback 
\[
f^* : \mathrm{CH}^d( \mathcal{M}^\beef) \to \mathrm{CH}^d( \mathcal{M} ).
\]

The following theorem is a special case of Proposition \ref{prop:derived pullback}.

\begin{bigtheorem}\label{BigThm:intro pullback}
There is a decomposition
\[
f^* \mathcal{C}^\beef(T^\beef, \mu^\beef)  =
 \sum_{   \substack{ S, T \in \Sym_d(\Q)  \\   S+T  = T^\beef   }    }
  R_\Lambda(S) \cdot   \mathcal{C} (  T  , \mu )
\]
of classes in $\mathrm{CH}^d( \mathcal{M})$, where $\mu=\mu^\sharp$ viewed as an element of 
\[
(L^\vee/L)^d \iso (L^{\sharp,\vee} / L^\sharp)^d,
\]
and $R_\Lambda(S)$ is the number of tuples $y\in \Lambda^d$ with moment matrix   $Q(y)=S$.
\end{bigtheorem}

Our main result, stated in the text as Theorem \ref{thm:modularity},  is an  extension of the Bruinier-Raum theorem (n\'{e}e Kudla's modularity conjecture) to the integral model $\mathcal{M}$.

\begin{bigtheorem}\label{BigThm:integral_modularity}
The formal generating series 
\[
\sum_{ T\in \Sym_d(\Q) } \mathcal{C}(T) \cdot q^T
\]
with coefficients
\[
\mathcal{C}(T) \define \sum_{ \mu \in (L^\vee / L)^d } \mathcal{C}(T,\mu) \otimes \phi^*_\mu \in \mathrm{CH}^{d}(\mathcal{M}) \otimes_\Q S_{L,d}^* 
\]
converges to a holomorphic Siegel modular form of weight $\frac{n}{2} +1 $ and representation $\omega_{L,d}^*$.
\end{bigtheorem}

Using  the finite-dimensionality of the space of holomorphic Siegel modular forms of a fixed weight and representation, one immediately obtains the following corollary of Theorem \ref{BigThm:integral_modularity}.

\begin{bigcorollary}
As $ T \in \Sym_d(\Q)$ and $\mu\in (L^\vee / L)^d$ vary, the cycle classes $\mathcal{C}(T,\mu)$ span a finite-dimensional subspace of $\mathrm{CH}^{d}(\mathcal{M})$.
\end{bigcorollary}

In~\cite{Madapusi2022}, the second author uses methods from derived algebraic geometry to construct derived cycle classes on essentially every Shimura variety to which Kudla's conjectures apply, and shows that they satisfy certain anticipated properties, giving an alternative proof of Theorems C, D, E above as special cases. However, the lack of Borcherds products in this generality has so far prevented progress towards any version of the main Theorem F in settings beyond the one treated in this paper.


\subsection{Outline of the paper}


In \S \ref{s:integral models} we recall the essentials of the theory of integral models of orthogonal Shimura varieties, and the families of special cycles  that live on them.   Our main reference for this material is \cite{HMP}, although many of the results we cite appeared before that in \cite{mp:spin}, \cite{AGHMP17}, and \cite{AGHMP18}.

The first new results appear in \S \ref{s:good cycles}, in which we investigate some of the finer geometric structure of the special cycles $\mathcal{Z}(T,\mu)$,  under the assumption that $\mathrm{rank}(T)$ is small compared to $n$.  
We remark that  the notion of smallness here depends on the lattice $L$, not just $n$, but  if $L$ is self-dual then  small means $\mathrm{rank}(T) \le (n-4)/3$.
What we show is that in this situation the special cycle $\mathcal{Z}(T,\mu)$ is flat over $\Z[\Sigma^{-1}]$, and equidimensional of the expected codimension $\mathrm{rank}(T)$ in $\mathcal{M}$.   In particular, when $\mathrm{rank}(T)$ is small one can define a corrected cycle class $\mathcal{C}(T,\mu)$ by imitating  \eqref{intro generic corrected}.

The generic fiber $Z(T,\mu)$ is smooth,  however even when $\mathrm{rank}(T)$ is small the special cycle $\mathcal{Z}(T,\mu)$ need not be regular, or even locally integral; it will often have irreducible components that cross in positive characteristic.  However, we can say enough about the geometry of its  irreducible components to prove in \S \ref{s:good cycles} the injectivity of the restriction map to the generic fiber
\[
\mathrm{CH}^1( \mathcal{Z}(T,\mu) ) \to \mathrm{CH}^1( Z(T,\mu) ) .
\]

Now suppose that $d$ is small relative to $n$.
Having shown in \S \ref{s:good cycles} that the special cycles $\mathcal{Z}(T,\mu)$ with $T\in \Sym_d(\Q)$ are well-behaved, we prove in \S \ref{s:low codimension} that the generating series of corrected cycles $\mathcal{C}(T,\mu)$ is modular.
This is done by  fixing  $T \in \Sym_d(\Q)$ and $\mu \in (L^\vee/L)^d$, and considering the family of special cycles 
 $\mathcal{Z}(T',\mu')$ in which $T'\in \Sym_{d+1}(\Q)$ has upper left $d\times d$ block $T$, and the first $d$ components of $\mu'$ are equal to $\mu$.   As $(T',\mu')$ varies, the resulting special cycles can be viewed as divisors 
 \[
 \mathcal{Z}(T,'\mu') \to \mathcal{Z}(T,\mu)
 \]
 on the fixed $\mathcal{Z}(T,\mu)$, and we prove that they  form the coefficients of a Jacobi form of index $T$ valued in $\mathrm{CH}^1(\mathcal{Z}(T,\,u))$. 
 The essential point is that, by the preceding paragraph, it suffices to check this in the generic fiber.
 In the generic fiber it follows, as in  \cite{zhang-thesis}, by  realizing $Z(T,\mu)$ as a union of orthogonal Shimura varieties and applying the modularity results of Borcherds \cite{Bor:GKZ}.
 By the main result of \cite{BWR}, this Jacobi modularity, for every pair $(T,\mu)$, implies that Theorem \ref{BigThm:integral_modularity} holds under our assumption that  $d$ is small relative to $n$.

 To remove the assumption that $d$ is small, we must first overcome the lack of equidimensionality of $\mathcal{Z}(T,\mu)$.
 In  \S \ref{s:derived cycle classes} we define the corrected classes $\mathcal{C}(T,\mu)$ needed even to state Theorem \ref{BigThm:integral_modularity} in full generality.  
 The construction itself is somewhat formal.  It relies on the close relations between Chow groups and $K$-theory proved in \cite{Gillet1987-ny} for schemes, and extended to stacks in \cite{Gillet1984-tk} and \cite{Gillet2009-tw}.

 Theorem \ref{BigThm:intro intersection} follows directly from the definition of $\mathcal{C}(T,\mu)$, but Theorems \ref{BigThm:intro comparison} and \ref{BigThm:intro pullback} seem to lie much deeper.  If the modularity of Theorem \ref{BigThm:integral_modularity} is to hold, the classes 
 $\mathcal{C}(T)$ of that theorem must satisfy the linear invariance property
 \[
 \mathcal{C}(T) = \mathcal{C}({}^t A T A)
 \]
 for any $A \in \GL_d(\Z)$.  While the analogous invariance of the naive special cycles $\mathcal{Z}(T,\mu)$ is obvious, the invariance of the corrected cycle classes encodes subtle information about self-intersections. 
 We prove this invariance in Proposition \ref{prop:linear invariance} by globalizing the arguments used in \cite{How19} to prove the analogous invariance for special cycles on unitary Rapoport-Zink spaces.
 The linear invariance is then used in an essential way in the proofs of 
Theorems \ref{BigThm:intro comparison} and \ref{BigThm:intro pullback}.

Finally, in \S \ref{s:main result} we prove Theorem \ref{BigThm:integral_modularity} in full generality.  
The idea here is simple enough to explain in a few sentences.  To prove the modularity of the generating series
\[
\phi(\tau) = \sum_T \mathcal{C}(T) q^T
\]
with coefficients in $\mathrm{CH}^d(\mathcal{M}) \otimes S_{L,d}^*$, pick an auxiliary positive definite self-dual lattice $\Lambda$.
As in Theorem \ref{BigThm:intro pullback}, we may  form   the quadratic lattice $L^\beef = L \oplus \Lambda$ of signature $(n^\beef,2)$, and the corresponding generating series
\[
\phi^\beef(\tau) = \sum_T \mathcal{C}^\beef (T) q^T
\]
with coefficients in $\mathrm{CH}^d(\mathcal{M}^\beef) \otimes S_{L^\beef,d}^*$.
One can rephrase the pullback formula of Theorem \ref{BigThm:intro pullback} as a factorization
\[
f^* \phi^\beef(\tau)  = \phi(\tau) \cdot \vartheta_{\Lambda,d}(\tau),
\]
where $\vartheta_{\Lambda,d}(\tau)$ is the usual scalar-valued genus $d$ Siegel theta series determined by the lattice $\Lambda$.
If we choose $\Lambda$ to have large rank, then $d$ will be much smaller than $n^\beef$, and so the modularity of the left hand side follows from the results of \S \ref{s:good cycles}.  Combining this with the modularity of $\vartheta_{\Lambda,d}(\tau)$ shows that $\phi(\tau)$ is a meromorphic Siegel modular form with poles supported on the vanishing locus of $\vartheta_{\Lambda,d}(\tau)$.
The lattice $\Lambda$, being an arbitrary and auxiliary choice, can then be varied to show that $\phi(\tau)$ is actually holomorphic.


\subsection{Acknowledgements}


The authors thank Andreas Mihatsch for alerting us to a misstatement in an earlier version of this paper, Steve Kudla for pointing out a subtlety in our use of the embedding trick,  and the anonymous referee for helpful comments and suggestions.


\section{The Shimura variety and its special cycles}
\label{s:integral models}


This section contains little in the way of new results.
Our goal is to recall from \cite{HMP} the  integral model of the Shimura variety associated to a quadratic space $(V,Q)$  over $\Q$ of signature $(n,2)$,  and the special cycles on that integral model.


\subsection{The Shimura variety}
\label{ss:initial data}


As in the introduction, we denote by  $G=\mathrm{GSpin}(V)$ the group of spinor similitudes of $V$. 
 By construction,  $G$ is an algebraic subgroup of the group of units of the Clifford algebra $C(V)$. 
The bilinear form associated to the quadratic form $Q$ is denoted
\begin{equation}\label{bilinear}
 [x,y]=Q(x+y)-Q(x)-Q(y).
 \end{equation}
 If we define a Hermitian symmetric domain 
\[
\mathcal{D} = \{ z\in V_\C : [z,z]=0,\, [z,\overline{z}] <0 \} / \C^\times  \subset \mathbb{P}(V_\C),
\]
then the pair  $(G,\mathcal{D})$ is a Hodge type Shimura datum with reflex field $\Q$.

Any choice of compact open subgroup in $G(\A_f)$ determines a Shimura variety, but we shall only consider subgroups of a particular type.
Fix a $\Z$-lattice $L\subset V$ satisfying $Q(L) \subset \Z$,  and let $L^\vee$ denote the dual lattice relative to the bilinear form \eqref{bilinear}. 
 For every prime $p$,  abbreviate $L_p=L\otimes \Z_p$, and let 
$
C(L_p) \subset C(V_p)
$ 
be the $\Z_p$-subalgebra generated by $L_p \subset C(V_p)$.
The compact open subgroup
\begin{equation}\label{best compact open}
K_p = G(\Q_p) \cap C(L_p)^\times 
\end{equation}
of $G(\Q_p)$  is the largest one that stabilizes the lattice $L_p$ and acts trivially on the discriminant group $L_p^\vee / L_p$.
The compact open subgroup
\begin{equation}\label{compact open}
K   =  \prod_p   K_p \subset G (\A_f)
\end{equation}
 determines a complex orbifold 
\[
M (\C) = G (\Q) \backslash \mathcal{D}  \times G (\A_f) / K,
\]
whose canonical model $M \to \Spec(\Q)$  is a smooth Deligne-Mumford stack of dimension $n$.

\begin{remark}
For a given prime $p$, one can make the compact open subgroup $K_p \subset G(\Q_p)$ as small as one wants by replacing $L$ by $p^kL$ for some $k\gg 0$.  In particular one is free to assume that  $K$ is neat, so that $M$ is a scheme rather than a stack.  
The penalty for doing so appears in the next subsection, when we form the integral model of $M$ over $\Z [ \Sigma^{-1}]$.
This smaller choice of $L$ will not be maximal at $p$, and so $p$ must be included in the finite set of bad primes $\Sigma$ that we invert.
\end{remark}

\begin{remark}
\label{rem:conn comp complex fiber}
As the derived subgroup $\mathrm{Spin}(V)\subset G$ is simply connected, we find by~\cite[(2.7.1)]{deligne:travaux} that the space of connected components of $M(\C)$ is a torsor under 
\[
\pi_0(\A^\times/\Q^\times\nu(K)\R_{>0})\iso \A_f^\times/\Q_{>0}^\times\nu(K),
\]
where $\nu:G\to \mathbb{G}_m$ is the spinor norm. 
It follows that if $\nu(K) = \widehat{\Z}^\times$, then $M$ is geometrically connected. 
This holds in particular if $L$ contains isotropic vectors $\ell,\ell_*\in L$ with $[\ell,\ell_*] = 1$.
\end{remark}


Suppose  $V^\beef$ is a  quadratic space of signature $(n^\beef,2)$, and let $(G^\beef, \mathcal{D}^\beef)$ be the associated Shimura datum.
 A $\Z$-lattice $L^\beef \subset V^\beef$  on which the quadratic form is $\Z$-valued determines a compact open subgroup  $K^\beef \subset G^\beef(\A_f)$  as in \eqref{compact open}, and hence a Shimura variety $M^\beef$ over $\Q$.

An  isometric embedding 
$
L\hookrightarrow L^\beef
$
 determines an injection  of Clifford algebras $C(V) \to C(V^\beef)$, which then induces a  closed immersion  of algebraic groups $G \hookrightarrow G^\beef$ exhibiting $G$ as the pointwise stabilizer of the orthgonal complement of $V \subset V^\beef$. 
 This embedding of groups induces an embedding of Shimura data
\[
(G , \mathcal{D})\to (G^\beef ,\mathcal{D}^\beef), 
\]
 As $K  \subset  K^\beef \cap G(\A_f)$, the theory of canonical models  implies the existence of a  finite and unramified morphism
\begin{equation}\label{shimura embedding}
M \to M^\beef
\end{equation}
of Deligne-Mumford stacks, given on $\C$-points  by
\[
G(\Q) \backslash \mathcal{D}  \times G (\A_f) / K 
 \map{ (  z,g) \mapsto (z,g) } 
 G^\beef (\Q) \backslash \mathcal{D} ^\beef \times G^\beef (\A_f) / K^\beef.
\]

More generally, for  any $g\in G^\beef (\A_f)$ we may replace $L$ by the quadratic lattice  $L_g = V \cap g L^\beef_{\widehat{\Z}}$  throughout the discussion above.
  The compact open subgroup associated to this lattice is  
\[
K_g = gK^\beef g^{-1}\cap G(\A_f),
\]
and the associated Shimura variety $M_g$ admits a  finite unramified morphism
\begin{equation}\label{twisted shimura embedding}
M_g \to M^\beef
\end{equation}
  given on $\C$-points by
\[
G(\Q) \backslash \mathcal{D}  \times G (\A_f) / K_g
\map{ ( z,h) \mapsto  (z,hg) } 
G^\beef(\Q) \backslash \mathcal{D} ^\beef \times G^\beef (\A_f) / K^\beef.
\]


\subsection{Integral models and special cycles}\label{ss:integral model}


We will use   \cite[\S 6]{HMP}  as our primary reference for the theory of  integral models of $M$.
See also  \cite{KMP}, \cite{AGHMP17}, and \cite{AGHMP18}.

\begin{definition}\label{def:max lattice}
For a prime $p$, we say that $L_p$ is \emph{maximal} if  there is no larger $\Z_p$-lattice of $V_p$ on which $Q$ is $\Z_p$-valued. 
 We say that $L_p$ is \emph{hyperspecial} if either
 \begin{itemize}
  \item $L_p$ is self-dual, or
  \item $p=2$, $\dim_\Q(V)$ is odd,  and  $[ L^\vee_2  : L_2 ]$ is not divisible by $4$.
 \end{itemize}
    We call $L$  \emph{maximal} or \emph{hyperspecial} if $L_p$ has this property for every $p$.
    \end{definition}

  \begin{remark}
Note that 
   \[
  L_p \mbox{ self-dual}  \implies  L_p \mbox{ hyperspecial} \implies  L_p \mbox{ maximal},
  \]
  and that
  $ L_p \mbox{ not maximal} \implies p^2 \mbox{ divides } [L^\vee : L].$
 \end{remark}

\begin{remark}
A hyperspecial lattice $L_p$ was called an \emph{almost self-dual} lattice  in \cite[Definition 6.1.1]{HMP}. 
 If $L_p$ is hyperspecial then   \eqref{best compact open} is a hyperspecial subgroup in the usual sense, justifying the terminology.  See  \cite[\S 6.3]{HMP}.
\end{remark}

As in the introduction,  $\Sigma$ will always denote  a finite  set of primes containing all primes $p$ for which  $L_p$ is not maximal.  
The constructions of \cite[\S 6]{HMP} provide us with a  normal and flat Deligne-Mumford stack 
\[
\mathcal{M} \to \Spec(\Z[\Sigma^{-1}])
\]
with generic fiber $M$.
Strictly speaking, \emph{loc.~cit.}~constructs an integral model  over the localization $\Z_{(p)}$ for any $p$ at which $L_p$ is maximal; these can be collated into a model  over $\Z[\Sigma^{-1}]$ as in \cite[\S 9.1]{HMP}.

\begin{proposition}\label{prop:regularity}
Assume that $\Sigma$ satisfies both
\begin{itemize}
  \item $p\in \Sigma$ for all  primes   $p$  such that  $p^2$ divides  $[ L^\vee : L]$, 
  \item if  $L_2$ is not hyperspecial then $2\in \Sigma$.
\end{itemize}
The stack  $\mathcal{M}$ is regular, and  for any $p \not\in \Sigma$ the localization $\mathcal{M}_{\Z_{(p)}}$ is the canonical integral model of $M$ in the sense of \cite[Definition 4.3]{mp:spin}.
\end{proposition}

\begin{proof}
If   $p\not\in \Sigma$ then either $L_p$ is hyperspecial, or $p$ is odd and $\ord_p(  [L^\vee : L]) =1$.
In the former case  $\mathcal{M}_{\Z_{(p)}}$ is the smooth canonical integral model of $M$ over $\Z_{(p)}$  constructed  \cite{KisinJAMS} and  \cite{KMP}.  In the latter case $\mathcal{M}_{\Z_{(p)}}$ is the regular canonical integral model constructed in~\cite[Theorem 7.4]{mp:spin}.  
\end{proof}

\begin{remark}
By Proposition \ref{prop:regularity}, if  $[L^\vee : L]$ is odd and squarefree then $\mathcal{M}$ is regular for any choice of $\Sigma$, including $\Sigma=\emptyset$.
 \end{remark}

 \begin{remark}
 If $p \not\in \Sigma$ is an odd prime  with $\mathrm{ord}_p([L^\vee:L]) = 2$, the model $\mathcal{M}_{\Z_{(p)}}$  is no longer regular.
  It does admit a regular resolution constructed by Pappas-Zachos~\cite{pappas_zachos}, which has a certain canonicity property formulated by Pappas~\cite{pappas_integral}, and now proven by Daniels-Youcis~\cite{daniels_youcis}. It would be interesting to extend the results of this paper to these regular integral models as well.
 \end{remark}

The integral model $\mathcal{M}$ comes with a  \emph{tautological line bundle}
\begin{equation}\label{taut bundle}
\omega \in \mathrm{Pic}( \mathcal{M} ), 
\end{equation}
called the \emph{line bundle of weight one modular forms} in \cite[\S 6.3]{HMP},  whose fiber at a complex point 
 \[
(z,g) \in G (\Q) \backslash \mathcal{D} \times G(\A_f) / K \iso \mathcal{M} (\C)
\]
 is identified with the isotropic line $\C z \subset V_\C$.

\begin{remark}\label{rem:int shimura embedding}
Recall from \eqref{shimura embedding} the morphism $M \to M^\beef$ of canonical models induced by an isometric embedding $L \hookrightarrow L^\beef$.  If $\Sigma$  also contains all primes $p$ for which $L_p^\beef$ is not maximal, then $M^\beef$ has its own flat and normal integral model 
\[
\mathcal{M}^\beef \to \Spec(\Z[\Sigma^{-1}])
\]
 and  the morphism  above  extends uniquely to a finite morphism
\[
\mathcal{M} \to \mathcal{M}^\beef.
\]
The tautological bundle on $\omega^\beef$ on the target pulls back to the tautological bundle $\omega$ on the source.  See \cite[Proposition 6.6.1]{HMP}.
\end{remark}

%
%
%

The integral model $\mathcal{M}$ also comes with a \emph{Kuga-Satake abelian scheme}
$
\mathcal{A} \to \mathcal{M}.
$
For every scheme $S\to \mathcal{M}$ there is a canonical (\emph{e.g.}~functorial in $S$)  subspace
  \[
    V(\mathcal{A}_S)_\Q\subset \End(\mathcal{A}_S)_{\Q}
  \]
  of \emph{special quasi-endomorphisms},   carrying a positive definite quadratic form defined by 
  $
     Q(x) = x\circ x
  $
  as elements of $\Q\subset \End(\mathcal{A}_S)_\Q$. 
  More generally,  for every coset $\mu\in L^\vee/L$ there is a subset
  \[
    V_\mu(\mathcal{A}_S)\subset V(\mathcal{A}_S)_\Q
  \]
  of \emph{special quasi-endomorphisms with denominator $\mu$}.   When $\mu=0$ this agrees with
  \[
 V(\mathcal{A}_S) \define V(\mathcal{A}_S)_\Q\cap \End(\mathcal{A}_S) ,
  \]
  and the subsets indexed by distinct cosets are disjoint.
Again, we refer the reader to \cite[\S 6]{HMP} for details.

\begin{remark}\label{rem:special congruence}
For any  $x\in V_\mu(\mathcal{A}_S)$, we have
\[
Q(x) \equiv Q(\tilde{\mu})\pmod{\Z},
\]
where $\tilde{\mu}\in L^\vee$ is any lift of $\mu$.  See~\cite[Proposition 4.5.4]{AGHMP18}.
\end{remark}

\begin{remark}\label{rem:special identification}
Suppose  $s\in \mathcal{M}(\C)$ is  a complex point corresponding to a pair 
\[
(z,g) \in G (\Q) \backslash \mathcal{D} \times G(\A_f) / K \iso \mathcal{M} (\C).
\]
 Recalling that $z\in V_\C$ is a nonzero isotropic vector,    there is a canonical  identification
\[
V(\mathcal{A}_s)_\Q =   \{x\in V  : \,  [x,z]  = 0\} ,
\]
respecting quadratic forms and satisfying
\[
V_\mu(\mathcal{A}_s) = \{x\in V  :\, x\in g \cdot ( \mu +  L_{\widehat{\Z}} ) \} .
\]
Here we are regarding $\mu + L_{\widehat{\Z}} \subset V_{\A_f}$.
\end{remark}

\begin{definition}
For   $t\in \Q$ and  $\mu \in L^\vee/L$,  the \emph{special divisor}
\begin{equation}\label{special divisors}
\mathcal{Z} (t,\mu) \to \mathcal{M}
\end{equation}
is the  finite, unramified, and relatively representable $\mathcal{M}$-stack
whose functor of points assigns to any scheme $S\to \mathcal{M}$ the set
\[
\mathcal{Z}(t,\mu)(S) = \{ x \in V_\mu(\mathcal{A}_S) : Q(x) = t \}.
\]
\end{definition}

The definition of special divisors can be generalized as follows.

\begin{definition}
Given an integer $d \ge 1$, a matrix $T\in \Sym_d(\Q)$, and a tuple of cosets
$
\mu = (\mu_1 ,\ldots, \mu_d) \in (L^\vee / L)^d,
$
the \emph{special cycle} 
\begin{equation}\label{special cycle}
\mathcal{Z}(T,\mu) \to \mathcal{M}
\end{equation}
is the finite, unramified, and relatively representable $\mathcal{M}$-stack
whose functor of points assigns to a scheme $S \to \mathcal{M}$ the set of all tuples 
\[
x=(x_1,\ldots, x_d) \in V_{\mu_1}(\mathcal{A}_S) \times \cdots \times V_{\mu_d}(\mathcal{A}_S)
\]
whose moment matrix 
\begin{equation}\label{inner products}
Q(x) \define \left( \frac{[x_i,x_j]}{2}  \right) \in \Sym_d(\Q)
\end{equation}
satisfies $Q(x)=T$.
\end{definition}

\begin{remark}\label{rem:pos def T}
As $V(\mathcal{A}_S)_\Q$ is a positive definite quadratic space, the special cycle  $\mathcal{Z}(T,\mu)$ is empty unless $T$ is positive semi-definite.
\end{remark}


\subsection{Special cycles as Shimura varieties} 
\label{ss:special shimura}


Given a special cycle \eqref{special cycle}, we explain how to write its generic fiber  
\[
Z(T,\mu) = \mathcal{Z}(T,\mu)_\Q
\]
as a disjoint union of Shimura varieties.  
We may assume that $T\in \Sym_d(\Q)$ is positive semi-definite, for otherwise $Z(T,\mu)=\emptyset$ by Remark \ref{rem:pos def T}.

Endow the space of column vectors $\Q^{d}$ with the  (possibly degenerate) quadratic form  
$
Q ( w ) = {}^tw T w ,
$
let $\mathrm{rad}(Q) \subset \Q^d$ be its  radical, 
and define a positive definite quadratic space 
\begin{equation}\label{W space}
W  = \Q^d / \mathrm{rad}(Q)
\end{equation}
of dimension $\mathrm{rank}(T)$.
Let $e_1,\ldots, e_d \in W$ be the images of the standard basis vectors in $\Q^d$.
Using the notation \eqref{inner products},  the tuple $e=(e_1,\ldots, e_d)$ has moment matrix $Q(e)=T$.

If $S$ is any scheme, a morphism  $S\to \mathcal{Z}(T,\mu)$ determines a tuple 
\[
x = (x_1,\ldots, x_d) \in V(\mathcal{A}_S)_\Q^d
\]
with $Q(x) = T$, and hence an isometric embedding 
\begin{equation}\label{W special embed}
W\map{e_i \mapsto x_i}  V(\mathcal{A}_S )_\Q.
\end{equation}

\begin{lemma}
\label{lem:W to V embeddings}
If  $Z(T,\mu)$ is non-empty then there exists an isometric embedding $W\hookrightarrow V$.
\end{lemma}
\begin{proof}
Using Remark \ref{rem:special identification}, a complex point $s\in Z(T,\mu)(\C)$ determines an isometric embedding 
\[
W\map{ \eqref{W special embed} }   V(\mathcal{A}_s)_\Q \subset V. \qedhere
\]
\end{proof}

By  Lemma \ref{lem:W to V embeddings} we may assume there exists an isometric embedding 
$W \hookrightarrow V$, which we now fix.
Any two embeddings lie in the same $G(\Q)$-orbit, so the particular choice  is unimportant.
Let $V^\flat \subset V$ be the orthogonal complement of $W$, so that  $V^\flat$ has signature $(n^\flat,2)$ with $n^\flat = n -\mathrm{rank}(T)$, and 
\[
V = V^\flat \oplus W.
\]
Applying the constructions of \S \ref{ss:initial data}  to $V^\flat$,  we obtain a reductive group $G^\flat = \mathrm{GSpin}(V^\flat)$ and   an embedding of Shimura data
\[
(G^\flat,\mathcal{D}^\flat )\hookrightarrow (G,\mathcal{D}).
\]

The vectors  $e_1,\ldots, e_d \in W \subset V$ determine a subset
\[
\Xi( T, \mu ) \define \{ g\in G(\A_f) :   e_i\in  g \cdot  ( \mu_i +  L_{\widehat{\Z}} )   \mbox{ for all } 1\le i \leq d  \} , 
\]
in which we  regard $\mu + L_{\widehat{\Z}} \subset V_{\A_f}$, and each $g \in \Xi( T, \mu )$ 
determines  $\Z$-lattices 
 \begin{equation}\label{prime sublattices}
 L_g^\flat = V^\flat \cap  g L_{\widehat{\Z}} ,\qquad \Lambda_g = W \cap g L_{\widehat{\Z}}.
 \end{equation}
 As the quadratic form on $V^\flat$ is $\Z$-valued on $L^\flat_g$, the constructions of \S \ref{ss:initial data} associate to  it a Shimura datum $(G^\flat, \mathcal{D}^\flat)$, Shimura variety  $M^\flat_g$   over $\Q$, and a finite unramified morphism  $M^\flat_g \to M$ as in \eqref{twisted shimura embedding}.

 \begin{proposition}\label{prop:ZTmu_uniformization}
  The set $\Xi(T,\mu)$ is stable under left multiplication by $G^\flat(\A_f)\subset G(\A_f)$
 and right multiplication by the compact open subgroup  $K \subset G(\A_f)$ of \eqref{compact open}.
The Shimura variety $M_g^\flat$ depends only on the double coset
 $G^\flat(\Q) g K$, and there is an  isomorphism of $M$-stacks
\begin{equation}\label{cycle presentation}
\bigsqcup_{ g \in G^\flat(\Q)\backslash \Xi(T,\mu)/K } M^\flat_g \iso   Z(T, \mu).
\end{equation}
 \end{proposition}

 \begin{proof}
 Only the decomposition \eqref{cycle presentation} is nontrivial. For that  we use Remark \ref{rem:special identification}
  to identify points of $Z(T,\mu)(\C)$ with 
  \[
G(\Q) \big\backslash   \left\{
  (z,x,g) \in \mathcal{D} \times V^d  \times G(\A_f)   : 
\begin{array}{c}   
  Q(x) = T \\ {  [z,x_i]  =0, \ \forall\,  1 \le i \le d } \\
  { x_i \in  g \cdot( \mu_i + L_{\widehat{\Z}}) , \ \forall\,  1 \le i \le d }
 \end{array}
  \right\} \big/ K.
  \]
  
  The key point is that the group $G(\Q)$ acts transitively on the set of tuples $x\in V^d$ satisfying $Q(x)=T$.  Thus any element  of the double quotient above  is represented by a triple of the form $(z,e,g)$, where $e=(e_1,\ldots, e_d) \in W^d \subset V^d$.  
  As the stabilizer of $e$ is precisely $G^\flat(\Q) \subset G(\Q)$,  and the condition $[z,e_i]=0$ for $1\le i \le d$ is equivalent to $z\in \mathcal{D}^\flat \subset \mathcal{D}$, we may rewrite the double quotient above as
   \[
   Z(T,\mu)(\C) \iso 
G^\flat(\Q) \backslash   \mathcal{D}^\flat   \times \Xi(T,\mu)  / K.
  \]

Over the complex fiber, the decomposition \eqref{cycle presentation} follows easily from this. In particular, for every $g$, we have maps $M_{g,\C}\to Z(T,\mu)_{\C}$ of finite unramified stacks over $M_{\C}$. To finish, it is enough to know that these maps descend over $\Q$: This will give the map underlying the isomorphism~\eqref{cycle presentation}, and that it is an isomorphism can be checked over $\C$. The desired descent to $\Q$ is in fact a consequence of the theory of canonical models for Shimura varieties and uses the moduli interpretation of the cycle $Z(T,\mu)$; see the argument in~\cite[Proposition 6.5]{mp:spin}.
 \end{proof}


\subsection{Basic properties of special cycles} 
\label{ss:cycle basics}


First  we explain in what sense the morphisms \eqref{special divisors}, which are not closed immersions, deserve to be called special divisors.

\begin{definition}\label{def:general cartier}
Suppose $D\to X$ is any finite, unramified, and relatively representable morphism  of Deligne-Mumford stacks.
By \cite[\href{https://stacks.math.columbia.edu/tag/04HJ}{Tag 04HG}]{stacks-project} there is an \'etale cover  $U \to X$ by a scheme such that the pullback $D_U\to U$ is a finite disjoint union 
\[
D_U = \bigsqcup_i D_U^i
\]
with each map $D_U^i \to U$ a closed immersion.
If each of these  closed immersions is an effective Cartier divisor on $U$ in the usual sense (the corresponding ideal sheaves are invertible),  then we call $D\to X$ a \emph{generalized Cartier divisor}.  
\end{definition}

\begin{remark}\label{rem:cartier bundle}
Any generalized Cartier divisor $D\to X$ determines an effective Cartier divisor (in the usual sense)  on $X$.  
Indeed, if we choose an \'etale cover $U \to X$ as in Definition \ref{def:general cartier}, then
$\bm{D}_U = \sum_i  D_U^i$ is an effective Cartier divisor on $U$.  The descent data for $D_U$ relative to $U \to  X$ induces descent data for $\bm{D}_U$, which then determines an effective Cartier divisor  $\bm{D} \hookrightarrow X$.
\end{remark}

\begin{proposition}\label{prop:divisors are divisors}
Fix $t\in \Q$ and $\mu \in L^\vee/L$.
\begin{enumerate}
  \item If $t>0$ then $\mathcal{Z}(t,\mu) \to \mathcal{M} $ is a generalized Cartier divisor.
    \item  If $t<0$ then $\mathcal{Z}(t,\mu) =\emptyset$.
  \item
  If  $t=0$ then
\[
\mathcal{Z}(0,\mu) = \begin{cases}
\mathcal{M} & \mbox{if }\mu=0 \\
\emptyset& \mbox{if }\mu\neq 0.
\end{cases}
\]
\end{enumerate}
\end{proposition}

\begin{proof}
The first assertion is~\cite[Proposition 6.5.2]{HMP}, 
while the second and third follow immediately from the definitions (and Remark \ref{rem:pos def T}).

For future reference, we recall the main ingredient of the proof of (1). Suppose that
\[
\xymatrix{
{   S } \ar[r]  \ar[d]   &   { \mathcal{Z} (t,\mu) }  \ar[d] \\
  { \widetilde{S} }   \ar[r]  &  {   \mathcal{M} }
}
\]
is  a commutative diagram of stacks in which  $S \to \widetilde{S}$ is a closed immersion of schemes defined by an ideal sheaf $J \subset \co_{\widetilde{S}}$ with $J^2=0$.   
The top horizontal arrow corresponds to a special quasi-endomorphism
$
x \in V_\mu(\mathcal{A}_S),
$
and we want to know when $x$ lies in the image of the (injective) restriction map
\begin{equation}\label{special deformation}
V_\mu(\mathcal{A}_{\widetilde{S}}) \to V_\mu(\mathcal{A}_S).
\end{equation}
Equivalently, when there is a (necessarily unique) dotted arrow 
\[
\xymatrix{
{   S } \ar[r]  \ar[d]   &   { \mathcal{Z}(t,\mu) }  \ar[d] \\
  { \widetilde{S} }   \ar@{-->}[ur] \ar[r]  &  {   \mathcal{M} }
}
\]
making the diagram commute.
In this situation, \cite[Proposition 6.5.1]{HMP} provides us with 
a canonical section
\begin{equation*}
\mathrm{obst}_x \in H^0 \big( \widetilde{S} , \omega|^{-1}_{\widetilde{S} } \big),
\end{equation*}
called the \emph{obstruction to deforming $x$},  with the property that  $x$ lies in the image of \eqref{special deformation} if and only if $\mathrm{obst}_x=0$.

Using this and Nakayama's lemma, one  shows that at any geometric point $z \to \mathcal{Z}(t,\mu)$,  the kernel of the natural surjection
\[
 \co^\et_{\mathcal{M} , z} \to \co^\et_{\mathcal{Z}(t,\mu) , z} 
\]
 is  a principal ideal, and (1) follows from this.  Again, see \cite[Proposition 6.5.2]{HMP} for  details.
\end{proof}

\begin{proposition} 
\label{prop:ZTmu_dim_bound}
Fix a special cycle \eqref{special cycle}.
Every irreducible component $\mathcal{Z} \subset \mathcal{Z}(T,\mu)$ satisfies
 \[
 \mathrm{dim}(  \mathcal{Z}  ) \ge   \mathrm{dim}(\mathcal{M})-\mathrm{rank}(T).
 \]
If equality holds, then $\mathcal{Z}(T,\mu)$ is a local complete intersection over $\Z[\Sigma^{-1}]$ at every point of $\mathcal{Z}$.
\end{proposition}

\begin{proof}
For any geometric point $z \to \mathcal{Z}(T,\mu)$, the kernel of the natural surjection
\[
 \co^\et_{\mathcal{M},z} \to 
\co^\et_{\mathcal{Z}(T,\mu),z} 
\]
 is generated by $d = \mathrm{rank}(T)$ elements.
This follows from Nakayama's lemma and the deformation theory used in the proof of Proposition~\ref{prop:divisors are divisors}; see also  \cite[Corollary 5.17]{mp:spin}. 
From this, the asserted inequality is immediate. Moreover, it is clear that $\mathcal{Z}(T,\mu)$ is a local complete intersection at $z$ whenever 
\[
\dim(  \co^\et_{\mathcal{Z}(T,\mu),z} ) = \dim( \co^\et_{\mathcal{M},z}  )- d = \dim(\mathcal{M}) - \mathrm{rank}(T).\qedhere
\]

\end{proof}

\begin{proposition}
\label{prop:naive linear invariance}
For any special cycle \eqref{special cycle}  and any  $A\in \mathrm{GL}_d(\Z)$, 
there is an isomorphism of $\mathcal{M}$-stacks 
\[
\mathcal{Z}(T ,\mu ) \iso \mathcal{Z}({}^tA T A , \mu A).
\] 
\end{proposition}

\begin{proof}
Given a scheme $S \to \mathcal{M}$, the isomorphism sends a tuple
\[
(x_1,\ldots, x_d) \in V(\mathcal{A}_S)_\Q^d
\]
to the tuple $(x_1,\ldots, x_d) \cdot A \in V(\mathcal{A}_S)_\Q^d$.
\end{proof}

\begin{proposition}\label{prop:naive intersection}
Given positive integers $d'$ and $d''$,  symmetric matrices 
\[
T' \in \Sym_{d'}(\Q) \quad \mbox{and}\quad  T'' \in \Sym_{d''}(\Q),
\]
and tuples $\mu' \in (L^\vee/L)^{d'}$ and  $\mu'' \in (L^\vee / L)^{d''}$,
there is a canonical isomorphism of $\mathcal{M}$-stacks
\[
\mathcal{Z}(T',\mu') \times_{\mathcal{M}} \mathcal{Z}(T'',\mu'') \iso
 \bigsqcup_{   T = \left(\begin{smallmatrix}  T' & * \\ * & T''  \end{smallmatrix} \right)   }  \mathcal{Z}(T,\mu)
\]
where the disjoint union is over all $T\in \Sym_{d'+d''}(\Q)$ of the indicated form, 
and $\mu = ( \mu',\mu'') \in (L^\vee/L)^{d'+d''}$ is the concatenation of the tuples $\mu'$ and $\mu''$.
\end{proposition}

\begin{proof}
 For any $\mathcal{M}$-scheme $S$, the $S$-valued points on both sides can be identified with the set of tuples
\[
(x',x'')\in\prod_{i=1}^{d'} V_{\mu'_i}(\mathcal{A}_S)\times\prod_{j=1}^{d''} V_{\mu''_{j}}(\mathcal{A}_S)
\]
such that  $Q(x')=T'$ and $Q(x'') = T''$.
\end{proof}

Suppose we have an  isometric embedding  $L\hookrightarrow L^\beef$ as in the discussion leading to \eqref{shimura embedding}.
As in Remark \ref{rem:int shimura embedding}, we assume that  $\Sigma$ contains all primes $p$ for which $L_p^\beef$ is not maximal, so that there is a morphism of integral models 
\[
 \mathcal{M}  \to \mathcal{M} ^\beef 
\]
over $\Z[\Sigma^{-1}]$.  The target of this morphism  has its own  special cycles 
\[
\mathcal{Z}^\beef(T^\beef,\mu^\beef) \to \mathcal{M}^\beef
\]
 indexed by $T^\beef \in \Sym_d(\Q)$ and $\mu^\beef \in (L^{\beef,\vee}/ L^\beef)^d$, 
and we wish to  describe their pullbacks to $\mathcal{M}$.

Denoting by $\Lambda \subset L^\beef$ the set of vectors  orthogonal to $L$,  there are  inclusions of lattices
\[
L \oplus \Lambda \subset L^\beef \subset L^{\beef,\vee} \subset L^\vee \oplus \Lambda^\vee .
\]
Given cosets
\[
\mu \in L^\vee/L ,\qquad \nu \in \Lambda^\vee /\Lambda, \qquad \mu^\beef \in L^{\beef,\vee}/L^\beef,
\]
we write $\mu+ \nu=   \mu^\beef$ to indicate that the natural map
\[
(L^\vee \oplus \Lambda^\vee) / (L  \oplus \Lambda) \to (L^\vee \oplus \Lambda^\vee ) / L^\beef
\]
sends 
\[
\mu + \nu \mapsto   \mu^\beef \in   L^{\beef,\vee} / L^\beef   \subset     ( L^\vee \oplus \Lambda^\vee ) / L^\beef .
\]

\begin{proposition}\label{prop:naive pullback}
There is an isomorphism of $\mathcal{M}$-stacks
\[
 \mathcal{Z}^\beef (T^\beef, \mu^\beef)  \times_{ \mathcal{M}^\beef} \mathcal{M}
\iso 
\bigsqcup_{   \substack{ T \in \Sym_d(\Q)  \\  \mu\in (L^\vee / L)^d } } 
 \bigsqcup_{ \substack{   \nu \in ( \Lambda^\vee/\Lambda)^d \\ \mu+\nu =  \mu^\beef   } } 
 \bigsqcup_{   \substack{ y\in  \nu + \Lambda^d \\  T+ Q(y) = T^\beef   }    }
   \mathcal{Z} (  T  , \mu ),
\]
where $\mu+\nu =  \mu^\beef$ is understood as above, but componentwise (that is to say, $\mu_i +\nu_i = \mu^\beef_i$ for every $1\le i \le d$).
\end{proposition}

\begin{proof}
By  \cite[Proposition 6.6.2]{HMP}, 
for any scheme $S \to \mathcal{M}$ there is a canonical isometric embedding
\[
V(\mathcal{A}_S)\oplus \Lambda \hookrightarrow V(\mathcal{A}^\beef_S) .
\]
Here $\mathcal{A} \to \mathcal{M}$ and $\mathcal{A}^\beef \to \mathcal{M}^\beef$ are the Kuga-Satake abelian schemes, and  the $\oplus$ on the left  is the orthogonal direct sum. 
This embedding determines a $\Q$-linear isometry
\[
V(\mathcal{A}^\beef_S)_\Q \iso 
 V(\mathcal{A}_S)_\Q \oplus \Lambda_\Q  ,
\]
which  restricts to a bijection
\[
V_{\mu^\beef}(\mathcal{A}^\beef_S) \iso
 \bigsqcup_{ \substack{  \mu \in L^\vee /L \\ \nu \in \Lambda^\vee/ \Lambda \\ \mu+\nu =  \mu^\beef } }V_{\mu}(\mathcal{A}_S) \times (\nu+\Lambda)
\]
for every $\mu^\beef\in L^{\beef, \vee} /L^\beef$.  
The proposition follows easily from this and the definition of  special cycles.
\end{proof}

 
\section{Special cycles of low codimension}
\label{s:good cycles}


Keep $L\subset V$ and $\mathcal{M} \to \Spec(\Z[\Sigma^{-1}])$  as in \S \ref{ss:initial data}  and\S \ref{ss:integral model}. 
Given a positive semi-definite $T\in  \Sym_d(\Q)$ and a $\mu \in (L^\vee/L)^d$,
our goal is prove that if $\mathrm{rank}(T)$ is  small  relative to $n=\dim(V)$, 
then the special cycle $\mathcal{Z}(T,\mu)$ is   equidimensional and flat over $\Z[\Sigma^{-1}]$.

  We  also show  that divisor classes on  $\mathcal{Z}(T,\mu)$  are determined by their restriction to the generic fiber.  
In \S \ref{s:low codimension}, this property will allow us to deduce modularity results for cycles on $\mathcal{M}$ from  known modularity results on its generic fiber.


\subsection{Connectedness in low codimension}


Our notion of smallness of $\mathrm{rank}(T)$ is always relative to the fixed lattice $L$, and depends on the following integer $r(L)$ associated to it.

 \begin{definition}\label{defn:rL}
Denote by  $r(L)$   the smallest integer $r\geq 0$ such that $L$ is isometric to a  $\Z$-module direct summand of a self-dual quadratic $\Z$-module $L^\beef$ of signature $(n+r,2)$.  The existence of such an $L^\beef$  follows from Proposition~\ref{prop:r genus appendix}. 
\end{definition}

\begin{remark}\label{rem:r genus}
For any $g\in G(\A_f)$ we have 
\[
r(L) = r(L_g),
\]
 where $L_g = V \cap g L_{\widehat{\Z}}$.
 This is   immediate from the fact that if $L$ embeds isometrically as a $\Z$-module direct summand of $L^\beef$, then $L_g$  embeds isometrically as a $\Z$-module direct summand of   $L^\beef_g = V^\beef \cap g{L^\beef_{\widehat{\Z}}}$.
\end{remark}

 \begin{proposition} \label{prop:ZTmu geom conn}
Suppose $T\in  \Sym_d(\Q)$ and  $\mu \in (L^\vee/L)^d$.  If
\[
 \mathrm{rank}(T) \leq \frac{ n - 2 r(L) -4 }{3} ,
\]
  then every connected component of the  generic fiber
  $
  Z(T,\mu) = \mathcal{Z}(T,\mu)_\Q
  $
   is geometrically connected.
 \end{proposition}

 \begin{proof}
 It suffices to show that each $\Q$-stack $M_g^\flat$ appearing in  \eqref{cycle presentation}  is geometrically connected.  
 Set $n^\flat = n-\mathrm{rank}(T)$, and recall from \eqref{prime sublattices} the quadratic lattice $L_g^\flat \subset V^\flat$ of  signature $(n^\flat,2)$ used to define $M_g^\flat$.
 Using Remark \ref{rem:conn comp complex fiber}, we are  reduced to proving the existence of  isotropic vectors
 $
 \ell ,\ell_* \in L_g^\flat
 $
 with $[\ell ,\ell_*]=1$.

In general, if  $N$ is a  quadratic $\Z$-module with $N_\Q$ non-degenerate, let $\gamma(N)$ the minimal number of elements needed to generate the finite abelian group $N^\vee / N$.  
This quantity only depends on the $\widehat{\Z}$-quadratic space $N_{ \widehat{\Z}}$. 
Moreover, if we realize $N \subset N^\beef$ as a $\Z$-module direct summand of a self-dual quadratic $\Z$-module as in Definition \ref{defn:rL}, there is a canonical surjection
\[
N^\beef\iso N^{\beef,\vee}\to N^\vee
\]
 whose restriction to $N$ is just the inclusion $N\to N^\vee$.  The induced  surjection
$
N^\beef / N \to N^\vee / N
$
shows that
\[
\gamma(N) \le \mathrm{rank}_\Z(N^\beef) - \mathrm{rank}_\Z(N).
\]

As in Remark \ref{rem:r genus}, set $L_g = V \cap g L_{\widehat{\Z}}$ and  abbreviate 
\[
r=r(L)=r(L_g).
\] 
Fix an embedding  $L_g \to L^\beef$ as a $\Z$-module direct summand of a self-dual quadratic lattice of signature $(n+r,2)$.   As the submodule $L^\flat_g \subset L_g$ of \eqref{prime sublattices} is a  $\Z$-module direct summand,  the paragraph above  implies
\[
\gamma(L_g^\flat ) \le  \mathrm{rank}_\Z(L^\beef) - \mathrm{rank}_\Z(L_g^\flat) 
= \mathrm{rank}(T) +r .
\]
This implies the first inequality  in 
\[
2 \cdot  \gamma(L_g^\flat) +6  \le 2 \cdot \mathrm{rank}(T) + 2r + 6  \le n - \mathrm{rank}(T) + 2  =      \mathrm{rank}_\Z (L_g^\flat)
\]
(the second is by the hypotheses of the proposition),
and so Proposition \ref{prop:hyperbolic embeddings} implies the existence of the desired isotropic vectors $\ell,\ell_* \in L_g^\flat$.
\end{proof}

 
\subsection{Geometric properties in low codimension: the self-dual case}
\label{ss:geometric_props}
 

In this subsection, we assume that $L$ is self-dual.  In particular, $L$ is hyperspecial (Definition \ref{def:max lattice}), and the integral model $\mathcal{M}$ is a smooth   $\Z[\Sigma^{-1}]$-stack by the proof of Proposition \ref{prop:regularity}.

Let $\Lambda$ be a positive definite quadratic $\Z$-module. 
 Set
\[
\mathsf{L}(\Lambda) = \{ \Z\mbox{-lattices }\Lambda'\subset \Lambda_\Q : \;\Lambda\subset \Lambda'\subset(\Lambda')^\vee\subset\Lambda^\vee\},
\]
and for each $\Lambda'\in \mathsf{L}(\Lambda)$, write $\mathcal{Z}(\Lambda')$ for the finite unramified stack over $\mathcal{M}$ with functor of points
\[
\mathcal{Z}(\Lambda')(S) = \{\mbox{isometric embeddings }\iota:\Lambda'\hookrightarrow V(\mathcal{A}_S)\}
\]
for any scheme $S \to \mathcal{M}$.

\begin{remark}\label{rem:secret cycle}
The  above stacks are actually  special cycles under a different name.
In what follows we fix a basis $e_1,\ldots, e_d \in \Lambda$,  and let $T=Q(e)\in \Sym_d(\Q)$ be the moment matrix of $e=(e_1,\ldots, e_d)$.  There is a canonical isometry
$\Lambda_\Q  \iso W$,
where the right hand side is the quadratic space \eqref{W space} determined by $T$, and a canonical isomorphism of $\mathcal{M}$-stacks
\[
\mathcal{Z}(\Lambda) \iso \mathcal{Z}(T,0)
\]
where $0=(0,\ldots,0) \in (L^\vee/L)^d$.
Indeed, an $S$-valued point of the left hand side is an isometric embedding $\iota: \Lambda \to V(\mathcal{A}_S)$, and the  tuple
$
x= (\iota(e_1) , \ldots, \iota(e_d)) \in V(\mathcal{A}_S),
$
 defines an $S$-point of  the right hand side.
\end{remark}

For each $\Lambda'\in \mathsf{L}(\Lambda)$,  the natural map
\[
\mathcal{Z}(\Lambda') \xrightarrow{\iota\mapsto \iota\vert_{\Lambda}}\mathcal{Z}(\Lambda) 
\]
 is  a closed immersion. 
Henceforth we  regard $\mathcal{Z}(\Lambda')$ as a closed substack of $\mathcal{Z}(\Lambda)$, so that 
$
\mathcal{Z}(\Lambda')\subset \mathcal{Z}(\Lambda'')
$
whenever $\Lambda''\subset \Lambda'$ is an inclusion of lattices in $\mathsf{L}(\Lambda)$. 
The  open substack of $\mathcal{Z}(\Lambda')$ defined by 
\[
\leftcirc\mathcal{Z}(\Lambda') = \mathcal{Z}(\Lambda')\smallsetminus \bigcup_{\Lambda'\subsetneq \Lambda''}\mathcal{Z}(\Lambda'')
\]
is then a locally closed substack of $\mathcal{Z}(\Lambda)$.

By construction, we have the equality of sets
\begin{equation}\label{geometric Lambda decomp}
 \mathcal{Z}(\Lambda)(k) = \bigsqcup_{\Lambda'\in \mathsf{L}(\Lambda)}\leftcirc\mathcal{Z}(\Lambda')(k)
\end{equation}
 for any algebraically closed field $k$ with $\mathrm{char}(k)\not\in \Sigma$.
 In fact, given a point $s\in \mathcal{Z}(\Lambda)(k)$ corresponding to an isometric embedding $\iota: \Lambda\to V(\mathcal{A}_s)$, we have $s\in \leftcirc\mathcal{Z}(\Lambda')(k)$ if and only if
\[
\Lambda' = V(\mathcal{A}_s)\cap \iota(\Lambda)_\Q .
\]
This last equality says simply that $\Lambda' \subset \Lambda_\Q$ is the largest lattice such that $\iota$ extends to $\iota : \Lambda' \to V(\mathcal{A}_s)$.

\begin{remark}
In the notation of  \S \ref{ss:DMchow}, the decomposition \eqref{geometric Lambda decomp} amounts to saying that 
the topological space $| \mathcal{Z}(\Lambda)|$ is the disjoint union of its locally closed subsets $|\leftcirc\mathcal{Z}(\Lambda')|$.
\end{remark}

In the generic fiber we have the following strengthening of \eqref{geometric Lambda decomp}.  

\begin{lemma}
\label{lem:union_lambda_prime}
For every $\Lambda' \in \mathsf{L}(\Lambda)$ the  morphism
\[
\leftcirc \mathcal{Z}(\Lambda')_\Q \to \mathcal{Z}(\Lambda)_\Q
\]
is an open and closed immersion, and there is an isomorphism  of $\Q$-stacks
\[
 \mathcal{Z}(\Lambda)_\Q \iso 
\bigsqcup_{\Lambda'\in \mathsf{L}(\Lambda)}\leftcirc\mathcal{Z}(\Lambda')_\Q
\]
 inducing the bijection  \eqref{geometric Lambda decomp}  on geometric points of characteristic $0$.
\end{lemma}

\begin{proof}
Proposition~\ref{prop:ZTmu_uniformization} and Remark \ref{rem:secret cycle} give us a decomposition 
\[
\mathcal{Z}(\Lambda)_\Q \iso Z(T,0) \iso 
 \bigsqcup_{g\in G^\flat(\Q)\backslash \Xi(T,0)/K}M^\flat_g ,
\]
which depends on a choice\footnote{If no such embedding exists then $\mathcal{Z}(\Lambda)_\Q  =\emptyset$ by Lemma~\ref{lem:W to V embeddings},  and there is nothing to prove.}  of  isometric embedding 
$
\Lambda_\Q  \hookrightarrow V.
$
Under this bijection, the locally closed substack 
$\leftcirc \mathcal{Z}(\Lambda')_\Q$ is identified with the disjoint union of those $M^\flat_g$ for which $g$ satisfies  $\Lambda' = \Lambda_\Q\cap g L_{\widehat{\Z}}$.  Here the intersection is taken inside $V_{\widehat{\Z}}$.
  The lemma follows immediately.
\end{proof}

The key geometric result is the following.

\begin{proposition}
\label{prop:zLam_geom_props}
Fix a $\Lambda' \in   \mathsf{L}(\Lambda)$, and assume
$
\mathrm{rank}_\Z(\Lambda)\leq (n-4)/2.
$
\begin{enumerate}
  \item 
  The  $\Z[\Sigma^{-1}]$-stack  $\leftcirc\mathcal{Z}(\Lambda')$ is  normal and flat, and equidimensional of dimension
  \[
n-\mathrm{rank}_\Z(\Lambda) + 1 = \dim(\mathcal{M}) - \mathrm{rank}_\Z(\Lambda).
\]
    \item 
  For any prime $p \not\in \Sigma$,  the special fiber $\leftcirc\mathcal{Z}(\Lambda')_{\F_p}$ is geometrically normal and equidimensional of dimension $n-\mathrm{rank}_\Z(\Lambda)$.

  \item For  any prime $p\not\in \Sigma$, the natural maps
   \[
       \pi_0\big(\leftcirc\mathcal{Z}(\Lambda')_{\F_p^\alg}\big)  \rightarrow
       \pi_0\big(\leftcirc\mathcal{Z}(\Lambda')_{ \Z^\alg_{(p)} }\big)
          \leftarrow \pi_0\big(\leftcirc\mathcal{Z}(\Lambda')_{ \Q^\alg }\big)
   \]
   are bijections, where $\Z^\alg_{(p)}$ is the integral closure of $\Z_{(p)}$ in $\Q^\alg$.
\end{enumerate}
\end{proposition}

\begin{proof}
We will use results from~\cite[ \S 7.1]{HMP} to which the reader is encouraged to refer for details. The key point is that, under our hypotheses, there exists an open substack (see Proposition 7.1.2 of \emph{loc. cit.})
\[
\mathcal{Z}^{\mathrm{pr}}(\Lambda') \subset \leftcirc\mathcal{Z}(\Lambda')
\]
with the following properties:
\begin{enumerate}
  \item It has the same generic fiber as $\leftcirc\mathcal{Z}(\Lambda')$.
  \item For any prime $p \not\in \Sigma$, the special fiber $\mathcal{Z}^{\mathrm{pr}}(\Lambda')_{\F_p}$ is smooth outside of a codimension $2$ substack.
\end{enumerate}

Moreover, Lemma 7.1.5 of \emph{loc. cit.} shows that the complement of $\mathcal{Z}^{\mathrm{pr}}(\Lambda')_{\F_p}$ in $\leftcirc\mathcal{Z}(\Lambda')_{\F_p}$ has codimension at least $2$. 
The statement there assumes that $\Lambda$ is maximal, but this is only used to ensure that $\Lambda'$ maps to a direct summand of $V(\mathcal{A}_s)$ for every geometric point $s \to \leftcirc\mathcal{Z}(\Lambda')_{\F_p}$. For us, this holds essentially by definition of $\leftcirc\mathcal{Z}(\Lambda')_{\F_p}$; see the comments after \eqref{geometric Lambda decomp}. 

Combining the above with the argument of Proposition 7.1.6 of \emph{loc. cit.} proves assertions (1) and (2).  Assertion (3) follows from~\cite[Theorem B]{Madapusi2025}.
\end{proof}

Proposition \ref{prop:zLam_geom_props} has two consequences, which are  of fundamental importance to our arguments. The first   requires the following technical lemma of commutative algebra.

\begin{lemma}
\label{lem:cm_flat}
Suppose  $R$ is a Cohen-Macaulay local ring over $\Z[\Sigma^{-1}]$. The following are equivalent:
\begin{enumerate}
  \item $R$ is flat over $\Z[\Sigma^{-1}]$.
  \item For every minimal prime $P\subset R$, $R/P$ is flat over $\Z[\Sigma^{-1}]$.
\end{enumerate}
\end{lemma} 
\begin{proof}
This is easily deduced from the following two facts.
 First, a $\Z[\Sigma^{-1}]$-algebra $S$ is flat if and only if every prime $p\notin\Sigma$ is a non-zero divisor in $S$. Second, since $R$ is Cohen-Macaulay, its zero-divisors are precisely those contained in some minimal prime of $R$.
\end{proof}

\begin{proposition}
\label{prop:Z_T_lci}
If  
$
\mathrm{rank}_\Z(\Lambda)\leq (n-4)/2,
$
then  $\mathcal{Z}(\Lambda)$ is a flat, reduced, local complete intersection over $\Z[\Sigma^{-1}]$, and is equidimensional of dimension $\dim ( \mathcal{M})  - \mathrm{rank}_\Z(\Lambda)$.
\end{proposition}

\begin{proof}
Fix an irreducible component 
\[
\mathcal{Z}\subset \mathcal{Z}(\Lambda)=\mathcal{Z}(T,0)
\]
 and a geometric generic point $s \to \mathcal{Z}$.\footnote{In Appendix A.1, we explain the notion of a `generic point' $\xi\to M$ of an irreducible component $Z$ of a Deligne-Mumford stack $M$, where $\xi$ is a \emph{punctual stack}. Such a punctual stack admits a finite \'etale cover $\Spec~L\to \xi$ by the spectrum of a field $L$, and a geometric generic point is one obtained by taking a separably closed extension of such a field $L$.} 
 
 By \eqref{geometric Lambda decomp}, there is a unique $\Lambda'\in \mathsf{L}(\Lambda)$ such that $s\to \leftcirc\mathcal{Z}(\Lambda')$. 
By claim (1)  of Proposition \ref{prop:zLam_geom_props},  the stack 
$\leftcirc\mathcal{Z}(\Lambda')$ is equidimensional of dimension $ \dim(\mathcal{M}) - \mathrm{rank}_\Z(\Lambda),$
and so the same is true of its  Zariski closure 
 $
 \overline{ \leftcirc\mathcal{Z}(\Lambda') } \subset \mathcal{Z}(\Lambda).
 $
The inclusion 
 $
 \mathcal{Z} \subset  \overline{ \leftcirc\mathcal{Z}(\Lambda') }
 $
 therefore implies 
 \[
 \dim(\mathcal{Z}) \le  \dim(\mathcal{M}) - \mathrm{rank}_\Z(\Lambda). 
 \]
As the other inequality follows from Proposition~\ref{prop:ZTmu_dim_bound}, we have proved both that $\mathcal{Z}$ has the expected dimension, and that it is equal to an irreducible component of 
$ \overline{ \leftcirc\mathcal{Z}(\Lambda') }$. 
 This latter stack is flat over $\Z[\Sigma^{-1}]$ by  claim (1) of Proposition \ref{prop:zLam_geom_props}, and hence so is $\mathcal{Z}$.

It now follows from  Proposition~\ref{prop:ZTmu_dim_bound}  that $\mathcal{Z}(\Lambda)$ is a local complete intersection over $\Z[\Sigma^{-1}]$.   In particular   $\mathcal{Z}(\Lambda)$ is Cohen-Macaulay, and hence flat by   Lemma \ref{lem:cm_flat} and  the flatness of its irreducible components proved above. To now see that it is reduced, it is enough to know that it is generically smooth, which follows from the complex uniformization in Proposition~\ref{prop:ZTmu_uniformization}.
\end{proof}

\begin{proposition}
\label{prop:CH_1_injective}
If 
$
 \mathrm{rank}_\Z(\Lambda)\leq (n-4)/3,
$
 then  restriction 
 \begin{align*}
\mathrm{CH}^1(\mathcal{Z}(\Lambda)) \to \mathrm{CH}^1( \mathcal{Z}(\Lambda)_\Q)
\end{align*}
to the generic fiber  is injective.
\end{proposition}

\begin{proof}
This amounts  to proving the triviality of the subspace 
\begin{equation}\label{vertical subspace}
\mathrm{CH}^1_{\mathrm{vert}}(\mathcal{Z}(\Lambda)) \subset \mathrm{CH}^1(\mathcal{Z}(\Lambda))
\end{equation}
spanned by the irreducible components of $\mathcal{Z}(\Lambda)_{\F_p}$ as $p\notin \Sigma$ varies.  For this we will use the following lemma, which provides a parametrization of those  components.

\begin{lemma}\label{lem:component param}
For a Deligne-Mumford stack $\mathcal{X}$, denote by $\pi_{\mathrm{irr}}(\mathcal{X})$ its set of irreducible components. For any prime $p\not\in\Sigma$ there are canonical bijections 
\begin{equation}\label{component param}
\xymatrix{
 { \bigsqcup\limits_{\Lambda'\in \mathsf{L}(\Lambda)} \pi_{\mathrm{irr}} \big(\leftcirc \mathcal{Z}(\Lambda')  \big)     } \ar[r]     \ar[d]  
 & 
{ \bigsqcup\limits_{\Lambda'\in \mathsf{L}(\Lambda)} \pi_{\mathrm{irr}} (\leftcirc \mathcal{Z}(\Lambda')_{\F_p}  )   }  \ar[d]    \\  
{   \pi_{\mathrm{irr}}  ( \mathcal{Z}(\Lambda)  )   }  
&
{  \pi_{\mathrm{irr}} ( \mathcal{Z}(\Lambda)_{\F_p}  )   }  
}
\end{equation}
characterized as follows:
\begin{enumerate}
\item
The horizontal arrow takes an irreducible component   $\leftcirc\mathcal{Z} \subset \leftcirc \mathcal{Z}(\Lambda')$ to its reduction  $\leftcirc\mathcal{Z}_{\F_p}\subset \leftcirc \mathcal{Z}(\Lambda')_{\F_p}$.
\item
The vertical arrow on the left takes an irreducible component   of the locally closed substack $\leftcirc \mathcal{Z}(\Lambda') \subset \mathcal{Z}(\Lambda)$ to its Zariski closure.
\item
The vertical arrow on the right takes an irreducible component of  the locally closed substack $\leftcirc \mathcal{Z}(\Lambda')_{\F_p} \subset \mathcal{Z}(\Lambda)_{\F_p}$ to its Zariski closure.
\end{enumerate}
Moreover, given distinct irreducible components
\[
\leftcirc\mathcal{Z}_1,\leftcirc \mathcal{Z}_2 \subset \leftcirc \mathcal{Z}(\Lambda') ,
\]
the intersection of $\leftcirc\mathcal{Z}_1$ with the Zariski closure of $\leftcirc \mathcal{Z}_2$ in $\mathcal{Z}(\Lambda)$ is empty.
\end{lemma}

\begin{proof}
We  first show that for any $\Lambda'\in \mathsf{L}(\Lambda)$,  all arrows in
\[
\xymatrix{
 {   \pi_0\big(\leftcirc \mathcal{Z}(\Lambda')_{\F^\alg_p} \big)   }  \ar[r]   \ar[d]_a  & {  \pi_0 \big(\leftcirc\mathcal{Z}(\Lambda')_{\Z^\alg_{(p)}} \big)  }  \ar[d]^b & {  \pi_0\big(\leftcirc \mathcal{Z}(\Lambda')_{\Q^\alg} \big) }  \ar[l] \ar[d]^c  \\
 {   \pi_0\big(\leftcirc \mathcal{Z}(\Lambda')_{\F_p} \big)   }  \ar[r]_d    & {  \pi_0 \big(\leftcirc\mathcal{Z}(\Lambda')_{\Z_{(p)}} \big)  }  & {  \pi_0\big(\leftcirc \mathcal{Z}(\Lambda')_{\Q} \big) }  \ar[l]^e
}
\]
are bijective.  Claim (3) of Proposition \ref{prop:zLam_geom_props} shows that both horizontal arrows in the top row are bijective.  
Proposition \ref{prop:ZTmu geom conn} and our hypothesis on $\mathrm{rank}(\Lambda)=\mathrm{rank}(T)$ guarantee that every connected component of $\mathcal{Z}(\Lambda)_\Q = Z(T,0)$ is geometrically connected 
(note that $r(L)=0$ by our assumption that  $L$ is self-dual), and so the same is true of $\leftcirc \mathcal{Z}(\Lambda')_\Q$ by Lemma \ref {lem:union_lambda_prime}.
This shows that the arrow labeled $c$ is  bijective.
The morphism 
\[
 \leftcirc\mathcal{Z}(\Lambda')_{\Z_{(p)}} \to \Spec(\Z_{(p)})
 \]
  is flat with reduced special fiber by claims (1) and (2) of Proposition \ref{prop:zLam_geom_props}, and so \cite[\href{https://stacks.math.columbia.edu/tag/055J}{Tag 055J}]{stacks-project}
implies that   the arrow labeled $e$ is injective.  The arrow labeled $b$ is surjective  because it is induced by a surjective morphism of stacks.
 It  follows that all arrows in the square on the right are bijections, as is the composition  $d\circ a$.
This implies the injectivity of $a$, and   surjectivity follows by the same reasoning as for $b$.
The arrow labeled $d$ is bijective because, at this point, we know the bijectivity of all the other arrows.

Now we turn to the diagram \eqref{component param}.
By Proposition \ref{prop:zLam_geom_props} each $\leftcirc \mathcal{Z}(\Lambda')$ is normal and flat over $\Z[\Sigma^{-1}]$, and so there are canonical identifications
\[
\pi_{\mathrm{irr}} \big(  \leftcirc \mathcal{Z}(\Lambda')\big)  
= \pi_{\mathrm{irr}}  \big(  \leftcirc \mathcal{Z}(\Lambda')_{\Z_{(p)}}\big) 
= \pi_0  \big(  \leftcirc \mathcal{Z}(\Lambda')_{\Z_{(p)}}\big).
\]
 Similarly, the normality of   $\leftcirc \mathcal{Z}(\Lambda')_{\F_p}$ implies 
 \[
\pi_{\mathrm{irr}} \big(  \leftcirc \mathcal{Z}(\Lambda')_{\F_p}\big)  
= \pi_0  \big(  \leftcirc \mathcal{Z}(\Lambda')_{\F_p }\big).
\]
 Combining this with the paragraph above yield the top horizontal bijection in \eqref{component param}.
The vertical bijections in \eqref{component param} are formal consequences of \eqref{geometric Lambda decomp} and our dimension calculations; see the proof of Proposition \ref{prop:Z_T_lci}.

The final claim follows from the normality of $\leftcirc \mathcal{Z}(\Lambda')$.
 Suppose $s \to \leftcirc \mathcal{Z}_1$ is a geometric point also contained in the Zariski closure of $\leftcirc \mathcal{Z}_2$ in $\mathcal{Z}(\Lambda)$.  This is the same as the Zariski closure of $\leftcirc \mathcal{Z}_2$ in $\mathcal{Z}(\Lambda')$, and hence any open subset of $\mathcal{Z}(\Lambda')$ containing $s$ must intersect $\leftcirc \mathcal{Z}_2$.
 One such open subset is  $\leftcirc \mathcal{Z}_1$ itself, and so $\leftcirc \mathcal{Z}_1 \cap \leftcirc \mathcal{Z}_2 \neq \emptyset$.  This is impossible, as these are distinct \emph{connected} components of $\leftcirc\mathcal{Z}(\Lambda')$.
\end{proof}

Lemma \ref{lem:component param} determines a canonical bijection
\begin{equation}\label{strong component bijection}
   \pi_{\mathrm{irr}} ( \mathcal{Z}(\Lambda)   )    \to 
 \pi_{\mathrm{irr}} ( \mathcal{Z}(\Lambda)_{\F_p}  )   ,
\end{equation}
but this does \emph{not} send an irreducible component $\mathcal{Z} \subset  \mathcal{Z}(\Lambda)$ to its reduction $\mathcal{Z}_{\F_p}\subset  \mathcal{Z}(\Lambda)_{\F_p}  $.
Indeed,  this reduction need not be irreducible (or reduced), and an irreducible component of $\mathcal{Z}(\Lambda)_{\F_p}$ may be contained in $\mathcal{Z}_{\F_p}$ for more than one $\mathcal{Z}$.
Instead,  the bijection sends $\mathcal{Z}$ to  a distinguished irreducible component of $\mathcal{Z}_{\F_p}$.

\begin{remark}
Although we will not need to do so, one can show  that this distinguished component can be characterized in the following way.
If we pull back the Kuga-Satake abelian scheme to a generic geometric point $\eta \to \mathcal{Z}$ then    there is a tautological isometric embedding $\Lambda \subset V(\mathcal{A}_\eta)$, and a largest
 $\Lambda' \subset \Lambda_\Q$  for which this  extends to  
$
\Lambda' \subset V(\mathcal{A}_\eta) .
$
It follows that if  $s \to \mathcal{Z}_{\F_p}$ is a geometric generic point of an irreducible component, then also  $\Lambda' \subset V(\mathcal{A}_s) $.
The distinguished irreducible component is the unique one for which this last inclusion cannot be extended to any larger lattice in $\Lambda_\Q$.  
\end{remark}
 
 We now return to the proof of Proposition \ref{prop:CH_1_injective}. 
 Fix a prime $p\not\in \Sigma$,  a $\Lambda' \in \mathsf{L}(\Lambda)$,  and an irreducible component $\leftcirc\mathcal{Z} \subset \leftcirc\mathcal{Z}(\Lambda')$.
 Using  Lemma \ref{lem:component param}, we see that the Zariski closure of $\leftcirc\mathcal{Z}$ in $\mathcal{Z}(\Lambda)$ is an irreducible component
 \begin{equation}\label{integral component}
 I(\Lambda', \leftcirc\mathcal{Z}) \in \pi_{\mathrm{irr}} (  \mathcal{Z}(\Lambda) ) ,
 \end{equation}
while  the Zariski closure of  $\leftcirc\mathcal{Z}_{\F_p}$ in $\mathcal{Z}(\Lambda)_{\F_p}$ is an irreducible component
\[
 I_p(\Lambda', \leftcirc\mathcal{Z}) \in \pi_{\mathrm{irr}} (  \mathcal{Z}(\Lambda)_{\F_p} ) 
\]
 contained in \eqref{integral component}.
 Theses two components correspond under the bijection \eqref{strong component bijection},  and all  irreducible components of $\mathcal{Z}(\Lambda)$  and $\mathcal{Z}(\Lambda)_{\F_p}$  are of this form.

To prove the triviality of the subspace \eqref{vertical subspace}, we must therefore prove the triviality of all cycle classes
\begin{equation}\label{vertical chow}
I_p (\Lambda', \leftcirc\mathcal{Z}) \in \mathrm{CH}^1( \mathcal{Z}(\Lambda)) .
\end{equation}
This will be by induction on the size of $\mathsf{L}(\Lambda)$.

The base case is when $\mathsf{L}(\Lambda) = \{\Lambda\}$, which happens exactly when $\Lambda$ is maximal. 
In this case  $\leftcirc\mathcal{Z}(\Lambda) = \mathcal{Z}(\Lambda)$, and so 
every  irreducible component $\leftcirc \mathcal{Z}\subset  \leftcirc \mathcal{Z}(\Lambda)$ is already Zariski closed in $\mathcal{Z}(\Lambda)$.  It follows that
\[
 I_p(\Lambda', \leftcirc\mathcal{Z}) =  \leftcirc\mathcal{Z}_{\F_p} =  I(\Lambda', \leftcirc\mathcal{Z})_{\F_p},
\]
which is trivial in $\mathrm{CH}^1(\mathcal{Z}(\Lambda))$. 
 Indeed, it is the Weil divisor of the rational function  on $\mathcal{Z}(\Lambda)$ that is  $p$ on  $I(\Lambda', \leftcirc\mathcal{Z})$ and $1$ on all other irreducible components.

We now turn to the inductive step. 
For any $\Lambda' \in \mathsf{L}(\Lambda)$,  the inclusion $\mathcal{Z}(\Lambda') \subset \mathcal{Z}(\Lambda)$  induces a pushforward (Proposition \ref{prop:chow pushforward})
\[
\mathrm{CH}^1(\mathcal{Z}(\Lambda')) \to \mathrm{CH}^1(\mathcal{Z}(\Lambda)).
\]
For an irreducible component $\leftcirc\mathcal{Z} \subset \leftcirc\mathcal{Z}(\Lambda')$ we have 
\[
I_p(\Lambda', \leftcirc\mathcal{Z}) \in
\pi_{\mathrm{irr}} (  \mathcal{Z}(\Lambda')_{\F_p} ) \subset \pi_{\mathrm{irr}} (  \mathcal{Z}(\Lambda)_{\F_p} )
\]
by construction, and  \eqref{vertical chow} is the pushforward of the corresponding class
\[
I_p(\Lambda', \leftcirc\mathcal{Z}) \in   \mathrm{CH}^1(\mathcal{Z}(\Lambda')).
\]
If $\Lambda \subsetneq \Lambda'$ then  this last  class is  trivial  by the induction hypothesis, and hence so is  \eqref{vertical chow}.

It now suffices to show that every  $I_p(\Lambda,\leftcirc\mathcal{Z})$ is rationally equivalent to $0$ on $\mathcal{Z}(\Lambda)$.
Consider the corresponding irreducible component $I(\Lambda,\leftcirc\mathcal{Z})$ of $\mathcal{Z}(\Lambda)$.    By the parametrization of the irreducible components of $\mathcal{Z}(\Lambda)_{\F_p}$,  there is  an equality
\begin{equation}\label{reduction decomp}
 I(\Lambda,\leftcirc\mathcal{Z})_{\F_p} 
= \sum_{  
\substack{   \Lambda' \in \mathsf{L}(\Lambda) \\   \leftcirc\mathcal{Z}'  \in \pi_{\mathrm{irr}}( \leftcirc \mathcal{Z}(\Lambda'))
} }
 m(\Lambda',\leftcirc \mathcal{Z}') \cdot 
  I_p(\Lambda' , \leftcirc\mathcal{Z}')  \in \mathscr{Z}^1(\mathcal{Z}(\Lambda))
\end{equation}
for some multiplicities $ m(\Lambda',\leftcirc \mathcal{Z}') \in \Z$. More precisely, $I(\Lambda,\leftcirc\mathcal{Z})_{\F_p}$ is an effective Cartier divisor on $\mathcal{Z}(\Lambda)$, and the multiplicities are given by the length of its \'etale local rings at each of its generic points, each of which of course is the generic point of an irreducible component in $\pi_{\mathrm{irr}}(\mathcal{Z}(\Lambda)_{\F_p})$. Note in particular that this means that the multiplicity $m(\Lambda,\leftcirc\mathcal{Z})$ is \emph{non-zero}.

First we consider those terms on the right hand side for which $\Lambda'=\Lambda$.
\begin{lemma}
For any $ \leftcirc\mathcal{Z}'  \in \pi_{\mathrm{irr}}( \leftcirc \mathcal{Z}(\Lambda)) $ we have
\[
 m(\Lambda,\leftcirc \mathcal{Z}')  \neq 0 \iff \mathcal{Z}'=\mathcal{Z}.
\]
\end{lemma}

\begin{proof}
 By construction we have
\[
  I_p(\Lambda, \leftcirc\mathcal{Z}) \subset   I(\Lambda, \leftcirc\mathcal{Z})_{\F_p} ,
\]
and so
 $
  m(\Lambda ,\leftcirc \mathcal{Z}) \neq 0.
 $
Conversely, if  $ m(\Lambda ,\leftcirc \mathcal{Z}') \neq 0$ then
 \[
 I_p (\Lambda ,\leftcirc \mathcal{Z}') \subset I (\Lambda ,\leftcirc \mathcal{Z})_{\F_p},
 \]
 and hence 
 $
 \leftcirc \mathcal{Z}'_{\F_p} \subset I (\Lambda ,\leftcirc \mathcal{Z})_{\F_p}. 
 $
 In particular,  $\leftcirc \mathcal{Z}'$ intersects the closure of $\leftcirc \mathcal{Z}$ in $\mathcal{Z}(\Lambda)$, and so 
$ \leftcirc \mathcal{Z}' =\leftcirc \mathcal{Z}$ by the final claim of Lemma \ref{lem:component param}.
 \end{proof}

If we take the image of \eqref{reduction decomp} in $\mathrm{CH}^1(\mathcal{Z}(\Lambda))$,
the left hand side vanishes because it the Weil divisor of the rational function on $\mathcal{Z}(\Lambda)$ that is $p$ on $I(\Lambda , \leftcirc\mathcal{Z})$ and $1$ on all other irreducible components.  
 On the right hand side we have proven the vanishing of every term with $\Lambda \subsetneq \Lambda'$, and of every term with $\Lambda'=\Lambda$ and $\leftcirc\mathcal{Z}' \neq \leftcirc \mathcal{Z}$.  Thus
\[
0 =
 m(\Lambda,\leftcirc \mathcal{Z}) \cdot 
  I_p(\Lambda , \leftcirc\mathcal{Z})  
\]
 in the Chow group.  As our Chow groups have rational coefficients, it follows that $ I_p(\Lambda , \leftcirc\mathcal{Z})  =0$,  completing  the proof of Proposition \ref{prop:CH_1_injective}.
 \end{proof}

 
\subsection{Geometric properties in low codimension: the general case}
 

We now return to the consideration of the general case where $L$ is not necessarily self-dual.

Let $r=r(L)$ and  $L \hookrightarrow L^\beef$ be as in Definition~\ref{defn:rL}.
Thus $L^\beef$ is a self-dual quadratic $\Z$-module of  signature $(n+r,2)$, containing $L$ as a $\Z$-module direct summand.   As in Remark \ref{rem:int shimura embedding},  there is an induced finite morphism  
\[
\mathcal{M}\to \mathcal{M}^\beef 
\]
 of normal integral models over $\Z[\Sigma^{-1}]$.
 The target comes with its own Kuga-Satake abelian scheme $\mathcal{A}^\beef \to \mathcal{M}^\beef$,  and its own family of special cycles
 \[
 \mathcal{Z}^\beef( T^\beef,\mu^\beef )\to \mathcal{M}^\beef.
 \]
 Of course we must have $\mu^\beef =0$,  by the self-duality of $L^\beef$, so we abbreviate
 \[
  \mathcal{Z}^\beef( T^\beef)= \mathcal{Z}^\beef( T^\beef, 0).
 \]

At last, we arrive at the main result of \S \ref{s:good cycles}. 

\begin{proposition}
\label{prop:low key}
Fix $T \in \Sym_d(\Q)$ and $\mu \in (L^\vee/L)^d$, and suppose 
\[
\mathrm{rank}(T) \leq \frac{n-2r-4}{3}.
\]
The special cycle $\mathcal{Z}(T,\mu)$ is a flat, reduced, local complete intersection over $\Z[\Sigma^{-1}]$, and is equidimensional of codimension $\mathrm{rank}(T)$ in $\mathcal{M}$.  
Moreover,  restriction to the generic fiber
\[
\mathrm{CH}^1(\mathcal{Z}(T,\mu) ) \to  \mathrm{CH}^1( Z(T,\mu)  )
\]
is injective.
\end{proposition}

\begin{proof}
If we set $n^\beef = n + r$, then every $T^\beef\in \mathsf{S}$ in Lemma~\ref{lem:ZTmu open closed}   satisfies
\[
\mathrm{rank}(T^\beef) = \mathrm{rank}(T) +r  \le   \frac{n^\beef -4}{3}  .
\]

Let $W^\sharp$ be the positive definite quadratic space of rank $\mathrm{rank}(T^\beef)$ associated to $T^\beef$ as in \eqref{W space}, and let $\Lambda^\beef \subset W^\beef$ be the $\Z$-lattice spanned by the  distinguished generators
$
e_1,\ldots, e_{n+r} \in W^\beef.
$
As in \S \ref{ss:geometric_props},  there is an associated Deligne-Mumford stack
\[
\mathcal{Z}^\beef (\Lambda^\beef) \to \mathcal{M}^\beef 
\]
parametrizing isometric embeddings of $\Lambda^\beef$ into $V(\mathcal{A}^\beef)$.

As the tuple $e^\beef =(e_1,\ldots, e_{n+r})$ has moment matrix $T^\beef = Q(e^\beef)$ by construction, Remark \ref{rem:secret cycle} provides us with a canonical isomorphism of $\mathcal{M}^\beef$-stacks
$
\mathcal{Z}^\beef (\Lambda^\beef) \iso \mathcal{Z}^\beef(T^\beef) .
$

We now need the following lemma.
\begin{lemma}
\label{lem:ZTmu open closed}
There exists a finite subset 
\[
\mathsf{S} \subset \Sym_{r +d}(\Q)
\]
 of  positive semi-definite matrices  of rank $r+\mathrm{rank}(T)$ such that there is an open and closed immersion of $\mathcal{M}^\beef$-stacks
\[
\mathcal{Z}(T,\mu)\hookrightarrow \bigsqcup_{T^\beef\in \mathsf{S}}\mathcal{Z}^\beef(T^\beef) .
\]
\end{lemma}
\begin{proof}
Let  $\Lambda \subset L^\beef$ be the orthogonal to $L \subset L^\beef$. By the self-duality of $L^\beef$ there are canonical bijections
 \[
L^\vee/L \leftarrow L^\beef/(L\oplus \Lambda)\rightarrow  \Lambda^\vee/\Lambda.
\]
It follows that  there is a (unique)  $\nu\in (\Lambda^\vee/\Lambda)^d$ such that $\mu+\nu = 0$ as elements of $ L^\beef/(L\oplus \Lambda)$.   
If we fix a lift $ \tilde{\nu}\in (\Lambda^\vee)^d$ and set 
\[
T_1  =  T+Q(\tilde{\nu}) \in \Sym_d(\Q),
\]
then Proposition~\ref{prop:naive pullback} implies that there is  an open and closed immersion
\begin{equation}\label{first immersion}
  \mathcal{Z}(T,\mu) \hookrightarrow \mathcal{Z}^\beef( T_1  )\times_{\mathcal{M}^\beef}\mathcal{M}.
\end{equation}

%
%
%
%

As  in the proof of Proposition~\ref{prop:naive pullback}, for any scheme $S \to \mathcal{M}$ 
there is  a canonical isometric embedding
\[
V(\mathcal{A}_S)\oplus \Lambda \hookrightarrow V(\mathcal{A}^\beef_S)
\]
for any scheme $S \to \mathcal{M}$,  extending $\Q$-linearly to an isomorphism
\[
V(\mathcal{A}_S)_\Q \oplus \Lambda_\Q \iso V(\mathcal{A}^\beef_S)_\Q.
\]
In particular, we have a canonical embedding $\Lambda \to V(\mathcal{A}^\beef_{\mathcal{M}})$. 
A choice of basis $y_1,\ldots, y_r \in \Lambda$ therefore determines a morphism
 \begin{equation}\label{second immersion}
\mathcal{M}\to \mathcal{Z}^\beef(T_2)
 \end{equation}
 of $\mathcal{M}^\beef$-stacks,  where $T_2=Q(y)$ is the moment matrix of the tuple
\[
y=(y_1,\ldots, y_r) \in \Lambda^r \subset V(\mathcal{A}^\beef_\mathcal{M})^r.
\]
This map is in fact an open and closed immersion. Since it is known to be finite, it is enough to know that it is an open immersion. For this, note that both source and target are normal Deligne-Mumford stacks, flat over $\Z_{(p)}$: for $\mathcal{Z}^\beef(T_2)$, this follows from our numerical hypotheses and Proposition \ref{prop:Z_T_lci}. By~\cite[Lemma 5.1.12]{Madapusi2025}, it now suffices to check that the map is generically an open immersion. This can be checked using the argument in~\cite[Lemma 7.1]{mp:spin}.

Combining \eqref{first immersion} and \eqref{second immersion} yields an open and closed immersion
\[
\mathcal{Z}(T,\mu) \to 
\mathcal{Z}^\beef( T_1 )\times_{\mathcal{M}^\beef}\mathcal{Z}^\beef(T_2)\iso
 \bigsqcup_{T^\beef = \begin{pmatrix}
T_1 &*\\
*& T_2
\end{pmatrix}}\mathcal{Z}^\beef(T^\beef),
\]
where we have used the product formula from Proposition~\ref{prop:naive intersection}. 
Explicitly, for any scheme $S\to \mathcal{M}$, a point of $\mathcal{Z}(T,\mu)(S)$ is given by a tuple $x\in \prod_{i=1}^d V_{\mu_i}(\mathcal{A}_S)$ satisfying $Q(x) = T$. The immersion sends $x$ to the tuple
\[
x^\sharp = (x+\tilde{\nu},y)
\in  V(\mathcal{A}^\beef_{S})^d \times V(\mathcal{A}^\beef_S)^{r} 
= V(\mathcal{A}^\beef_S)^{d+r}
\]
in the factor indexed by $T^\sharp = Q(x^\sharp)$.

It  remains to show that $\mathcal{Z}(T,\mu)$ only meets those $\mathcal{Z}^\beef(T^\beef)$ with $T^\beef$ positive semi-definite and
\[
\mathrm{rank}(T^\beef) = \mathrm{rank}(T) + r  = \mathrm{rank}(T) + \mathrm{rank}_\Z(\Lambda).
\] 
This  follows from Remark \ref{rem:pos def T} and the observation (noting that every component of $\tilde{\nu}\in \Lambda_\Q^d$ is a $\Q$-linear combination of $y_1,\ldots, y_r\in \Lambda$)  that the components of   $x^\beef$   and the components of $(x,y)$
generate the same subspace of $V(\mathcal{A}^\beef_S)_\Q = V(\mathcal{A}_S)_\Q\oplus\Lambda_\Q$.
\end{proof}

Using Lemma \ref{lem:ZTmu open closed}, the desired properties of $\mathcal{Z}(T,\mu)$ follow immediately from the corresponding properties of $\mathcal{Z}^\beef (\Lambda^\beef)$ proved in Proposition  \ref{prop:Z_T_lci} and   Proposition  \ref{prop:CH_1_injective}.
\end{proof}

 
\section{Modularity in low codimension}
\label{s:low codimension}
 

Keep the quadratic lattice $L\subset V$ and the integral model  $\mathcal{M}$ over  $\Z[\Sigma^{-1}]$ as in  \S \ref{ss:initial data} and \S \ref{ss:integral model}.
We consider the family of special cycles $ \mathcal{Z}(T',\mu')$
 on $\mathcal{M}$   indexed by  those $T' \in \Sym_{d+1}(\Q)$ whose  upper left $d\times d$ block is a fixed $T\in \Sym_d(\Q)$.
Roughly speaking, our goal is show that if $d$ is small  then these special cycles  are the Fourier coefficients of a Jacobi  form of index $T$, valued in the codimension  $d+1$ Chow group of $\mathcal{M}$.  

Such a result was proved in the generic fiber (without restriction on $d$) in the thesis of W.~Zhang, by reducing it to a modularity result  of Borcherds for generating series of divisors.
It is the crucial Proposition \ref{prop:low key} that will  allow  us to deduce the analogous  result on the integral model from the results of Borcherds in the generic fiber.

 
\subsection{Jacobi forms}
\label{ss:jacobi}
 

  We recall just enough of the theory of Jacobi forms to fix our conventions, as these differ slightly  from \cite{zhang-thesis}  and \cite{BWR}.

Fix an integer  $g\ge 1$.  The Siegel modular group $\Gamma_g=\mathrm{Sp}_{2g}(\Z)$ acts on the   Siegel half-space $\mathcal{H}_g \subset \Sym_g(\C)$ via the usual formula
\[
\gamma \cdot \tau = (A\tau+B)(C\tau+D)^{-1},
\]
where we have written
\[
\gamma= \begin{pmatrix} A&B\\C&D  \end{pmatrix}  \in \Gamma_g.
\]

Let $\Gamma_1=\SL_2(\Z)$ act on the space of matrices $M_{2,g-1}(\Z)$ by left multiplication. Following~\cite{BWR}, we regard the \emph{Jacobi group} 
\[
J_g \define \Gamma_1 \imes M_{2,g-1}(\Z) 
\]
as a sub\emph{set} of $\Gamma_g$ using the injective function (\emph{not} a group homomorphism)
\begin{equation}\label{jacobi embedding}
\left( \begin{pmatrix} a& b\\c&d \end{pmatrix} , \begin{pmatrix} {}^tx \\ {}^ty \end{pmatrix} \right)
\mapsto 
\left(
\begin{array}{c c | c c }
1_{g-1} & - y & 0_{g-1} & x \\
0 & a & a\, {}^tx+b\, {}^ty & b \\ \hline
0_{g-1} & 0 & 1_{g-1} & 0 \\
0 & c & c \, {}^t  x + d\, {}^ty & d
\end{array}\right)
\end{equation}
for column vectors $x,y\in \Z^{g-1}$. Note  that the restriction of this injective map to the subgroup $\Gamma_1\subset J_g$ is actually a group homomorphism.

\begin{remark}
As in \cite[\S 4]{WR15}, there is an \emph{extended Jacobi group} $J_g^{ \mathrm{ext} }$, which can be realized both as a subgroup of  $\Gamma_g$, and as a central extension 
\[
1 \to \Z \to J_g^{ \mathrm{ext} } \to J_g \to 1
\]
with the property that the surjection to  $J_g$ admits a set-theoretic section whose image generates $J_g^{\mathrm{ext}}$.  The use of the function \eqref{jacobi embedding} is a convenient way of hiding the presence of the larger extended Jacobi group, as the subset $J_g \subset \Gamma_g$ generates a subgroup isomorphic to   $J_g^{ \mathrm{ext} }$.  
\end{remark}

The metaplectic double cover of the Siegel modular group is denoted 
$
\widetilde{\Gamma}_g.
$
Its elements are pairs
\begin{equation}\label{meta element}
\tilde{\gamma} = ( \gamma , j_\gamma  ) \in \widetilde{\Gamma}_g,
\end{equation}
consisting of a  $\gamma \in \Gamma_g$  and a holomorphic function  $j_\gamma(\tau)$  on $\mathcal{H}_g$ whose square  is $\det(C\tau+D)$.
As  in \cite[(5)]{BWR}, the \emph{metaplectic Jacobi group}
\[
\widetilde{J}_g \define \widetilde{\Gamma}_1 \imes M_{2,g-1}(\Z)  
\]
can be identified with a subset of  $\widetilde{\Gamma}_g$, using an injection lifting 
\eqref{jacobi embedding}.  Moreover, the restriction of this embedding to $\widetilde{\Gamma}_1$ is also a group homomorphism. In this way, we can view $\widetilde{\Gamma}_1$ as a \emph{subgroup} of $\widetilde{\Gamma}_g$.

In the following definition, taken from \cite[\S 2.2]{BWR},  we write elements of the Siegel half-space as
\[
\tau  = \begin{pmatrix} \tau'' & z \\ {}^t z & \tau'  \end{pmatrix} \in \mathcal{H}_g
\quad \mbox{with} \quad \tau' \in \mathcal{H}_1, \quad \tau'' \in \mathcal{H}_{g-1}, 
\]
 and $z\in \C^{g-1}$ a column vector.

\begin{definition}\label{def:jacobi form}
Suppose $ \rho : \widetilde{\Gamma}_g \to \GL(V)$ is a finite dimensional  representation with finite kernel.
A  holomorphic function 
\[
\phi_T : \mathcal{H}_1 \times \C^{g-1} \to V
\]
 is a \emph{Jacobi form of half-integral weight $k$,  index $T\in \Sym_{g-1}(\Q)$, and representation $ \rho$} if the function 
\begin{equation}\label{jacobi augment}
\Phi_T(\tau) \define \phi_T(\tau',z)\cdot e^{2\pi i \mathrm{Tr}( T\cdot  \tau'' ) }
\end{equation}
on $\mathcal{H}_g$  satisfies the transformation law
\[
\Phi_T(\gamma \cdot \tau) =  j_\gamma (\tau)^{2k}  \rho( \tilde{ \gamma} ) \cdot \Phi_T(\tau) 
\]  
for all elements \eqref{meta element} in the subset $\widetilde{J}_g \subset \widetilde{\Gamma}_g$, and if  for all $\alpha,\beta \in \Q^{g-1}$ the function
$
\phi_T( \tau' , \tau'   \alpha +\beta)
$
of  $\tau' \in \mathcal{H}_1$ is holomorphic at $\infty$.
\end{definition}

Any Jacobi form of representation $\rho$  has a  Fourier expansion
 \[
 \phi_T( \tau' , z ) = \sum_{ \substack{ m \in \Q  \\ \alpha \in \Q^{g-1} } } 
 c(m,\alpha) \cdot q^m \xi_1^{\alpha_1} \cdots \xi_{g-1}^{\alpha_{g-1}  },
 \]
 where we have set $q^m = e^{2\pi i m \tau'}$ and $\xi_i^{\alpha_i} = e^{2\pi i \alpha_i z_i}$.

\begin{remark}
Suppose  $ \rho : \widetilde{\Gamma}_g \to \GL(V)$  is a finite dimensional representation with finite kernel.
Given a  holomorphic Siegel modular form 
  $\phi : \mathcal{H}_g \to V$ of half-integer weight $k$ and representation $\rho$, there is  a Fourier-Jacobi expansion
\[
\phi(\tau ) = \sum_{ T \in \Sym_{g-1}  (\Q)}  \phi_T( \tau', z ) \cdot e^{2\pi i \mathrm{Tr}( T\cdot  \tau'' ) }
\]
in which each coefficient $\phi_T$ is a Jacobi form of the weight $k$, index $T$, and representation $\rho$.
\end{remark}

In practice, the representation $\rho$ will always be a form of the Weil representation.
 Let $N$ be a free $\Z$-module of finite rank endowed with a (nondegenerate) quadratic form $Q$.  Denote by 
\[
S_{N,g} =  \C [ ( N^\vee / N)^g ]  \iso \C[ N^\vee/N]^{\otimes g}
\]
the finite dimensional  vector space of $\C$-valued functions on $(N^\vee/N)^g$, and by $S_{N,g}^*$ its $\C$-linear dual.
For any $\mu \in (N^\vee / N)^g$ we denote by 
$
\phi_\mu \in S_{N, g}
$
the characteristic function of $\mu$.  As $\mu$ varies these form a basis of $S_{N,g}$, and we denote by 
$
\phi^*_\mu \in S^*_{N, g}
$
the dual basis vectors.

Denote by 
\begin{equation}\label{weil}
\omega_{N,g} : \widetilde{\Gamma}_g \to \GL( S_{N,g} )
\quad \mbox{and}\quad \omega^*_{N,g} : \widetilde{\Gamma}_g \to \GL( S^*_{N,g} )
\end{equation}
the Weil representation and its contragredient.  
To resolve some confusion in the literature, we now pin down the precise normalization of \eqref{weil}.
 Suppose $N$ has signature $(p,q)$, and  $\tilde{\gamma}=(\gamma, j_\gamma)$ is as in \eqref{meta element}.
\begin{enumerate}
\item
If 
$
\gamma = \left(\begin{smallmatrix}
A & 0 \\ 
0 & D
\end{smallmatrix} \right)
$
then, noting that $j_\gamma$ is a constant function (in fact a square root of $\det(A)=\det(D)=\pm 1$), 
\[
\omega_{N,g}(  \tilde{\gamma} ) \cdot \phi_\mu =  j_\gamma^{p-q} \cdot  \phi_{\mu \cdot A^{-1}} .
\]
\item
 If $ \gamma = \left(\begin{smallmatrix}
I & B \\ 
0 & I
\end{smallmatrix} \right)$  and $j_\gamma =1$ then
\[
\omega_{N,g}(  \tilde{\gamma} ) \cdot \phi_\mu =   e^{-2\pi i \mathrm{Tr}( Q( \tilde{\mu} ) B)}  \cdot  \phi_{\mu} ,
\]
where $Q( \tilde{\mu})\in \Sym_g(\Q)$ is the moment matrix of any lift $\tilde{\mu} \in (N^\vee)^g$.
\item
 If $ \gamma = \left(\begin{smallmatrix}
0 & -I \\ 
I & 0
\end{smallmatrix} \right)$  and the square root $j_\gamma(\tau) = \sqrt{\det(\tau)}$ is the standard branch determined by   
$
j_\gamma( i I_g ) = e^{ g \pi i /4} ,
$
  then
\[
\omega_{N,g}(  \tilde{\gamma}) \cdot \phi_\mu =   \frac{e^{   \pi i g(p-q)/4}}{ [ N^\vee : N]^{g/2} }   
\sum_{ \nu \in (N^\vee/N)^g } e^{2\pi i [ \mu,\nu] }  \cdot  \phi_{\nu} .
\]
\end{enumerate}
These relations determine $\omega_{N,g}$ uniquely.  The normalization is such that if  $N$  positive definite of rank $n$, the theta series
\[
\vartheta_{N,g}(\tau) = \sum_{ \mu \in (N^\vee/N)^g } 
\left(  \sum_{ x\in  \mu+ N^g}  e^{ 2\pi i  \mathrm{Tr}(\tau Q(x)) }  \right)
\phi^*_\mu 
\]
 is a Siegel modular form of weight $n/2$ and  representation $\omega^*_{N,g}$.
 
 \begin{remark}
 Our Weil representation does not agree with the Weil representation  $\rho_{N,g}$  of \cite[Definition 2.2]{zhang-thesis}. 
Instead, the isomorphism 
 \[
 S_{N,g} \map{\phi_\mu \mapsto \phi_\mu^*  }  S_{N,g}^*
 \]
  identifies  $\rho_{N,g}$  with $\omega_{N,g}^*$.   Alternatively, our $\omega_{N,g}$ agrees with Zhang's $\rho_{-N,g}$, where $-N$ has the same underlying $\Z$-module as $N$, but is endowed with the signature $(q,p)$ quadratic form $-Q$.
 \end{remark}
  
 \begin{remark}
There is a recurring error in \cite{zhang-thesis}, originating in the proof of \cite[Theorem 2.9]{zhang-thesis}.
 That proof claims that  a certain generating series of Borcherds  is a modular form of representation  $ \rho_{L,1}^*$, where   $L$ is a quadratic lattice of signature $(n,2)$.   In fact, this modular form has representation $ \omega_{L,1}^*$;  see the proof of Proposition \ref{prop:generic divisor generating}.
 The point is that Borcherds works with a lattice $L$ of signature $(2,n)$, and to  reformulate the theorem for signature $(n,2)$ one must replace the quadratic form by its negative.
 This does not change the space $S_{L,1}$, but it does change  the Weil representation (see the previous remark).
  \end{remark}

\begin{remark}
When $g=1$ we omit it from the notation, so that 
\[
\widetilde{\Gamma} = \widetilde{\Gamma}_1
\quad \mbox{and}\quad
\mathcal{H}=\mathcal{H}_1,
\]
 and the Weil representation and its contragredient are 
\[
\omega_{N} : \widetilde{\Gamma} \to \GL( S_{N} )
\quad \mbox{and}\quad
\omega^*_{N} : \widetilde{\Gamma} \to \GL( S^*_{N} ).
\]
\end{remark}

\subsection{Statement of modularity in low codimension}
 

 Throughout the remainder of \S  \ref{s:low codimension} we impose the following two hypotheses.

\begin{hypothesis}\label{hyp:regular primes}
The integral model $\mathcal{M}$ over $\Z[\Sigma^{-1}]$ is regular.  See Proposition \ref{prop:regularity} for  conditions on  $\Sigma$ that guarantee this.
\end{hypothesis}

\begin{hypothesis}\label{hyp:low codimension}
Recalling the integer  $r(L) \ge 0$  of  Definition \ref{defn:rL}, we  assume  that $d$ is a positive integer satisfying 
\[
d+1  \leq \frac{n-2r(L)-4}{3}  .
\]
\end{hypothesis}

The first hypothesis is needed\footnote{If we knew that  there was a well-defined intersection product on the rational Chow group of a normal (but not necessarily regular)  stack, then Hypothesis \ref{hyp:regular primes} would be unnecessary throughout \S \ref{s:low codimension}.
   In fact, it would be sufficient to know that there is a well-defined intersection product between \emph{Cartier divisors} and arbitrary cycle classes on a normal stack.} to make sense of  \eqref{eqn:def_corrected_class} below, which requires   a well-defined intersection product on the rational Chow groups of $\mathcal{M}$.  This is  available to us if $\mathcal{M}$ is regular, as explained in \S \ref{ss:chow}.  The second hypothesis is imposed so that we may make use of  Proposition \ref{prop:low key}.

Suppose $T' \in \Sym_{d+1}(\Q)$ and $\mu' \in (L^\vee/L)^{d+1}$.
Hypothesis \ref{hyp:low codimension} and Proposition \ref{prop:low key} imply that the special cycle
\[
\mathcal{Z}(T',\mu') \to \mathcal{M}
\]
is flat over $\Z[\Sigma^{-1}]$, and equidimensional of codimension $\mathrm{rank}(T')$ in $\mathcal{M}$.
By Definition~\ref{defn:naive class}, there is an associated  \emph{naive cycle class}
\begin{equation*}
[\mathcal{Z}(T',\mu')] \in \mathrm{CH}^{\mathrm{rank}(T')}( \mathcal{M} ).
\end{equation*}
Define the \emph{corrected cycle class} 
 \begin{equation}\label{eqn:def_corrected_class}
 \mathcal{C}(T',\mu')    \define  \underbrace{c_1(\omega^{-1}) \cdots c_1(\omega^{-1})}_{d+1-\mathrm{rank}(T') } \cdot  [\mathcal{Z}(T',\mu')] \in \mathrm{CH}^{d+1} (\mathcal{M}) ,
 \end{equation}
 where $c_1(\omega^{-1})$ is the image of the inverse tautological line bundle \eqref{taut bundle} under the  first Chern class map of Definition \ref{def:first chern class}.
 Abbreviate
\[
\mathcal{C}(T') = \sum_{\mu'\in (L^\vee/L)^{d+1}}\mathcal{C}(T',\mu')  \otimes \phi^*_{\mu'}\in 
\mathrm{CH}^{d+1}(\mathcal{M})\otimes S^*_{L,{d+1}}.
\]
The remainder of  \S \ref{s:low codimension} is devoted to the proof of the following result.

\begin{proposition}\label{prop:jacobi generating}
For any fixed  $T\in \Sym_d(\Q)$, the formal generating series 
\[
 \sum_{    \substack{  m\in \Q  \\   \alpha\in    \Q^d     }   } 
\mathcal{C}  \left(\begin{matrix}  T &  \frac{\alpha}{2} \\  \frac{{}^t \alpha}{2}   & m   \end{matrix}\right)   
 \cdot   q^m \xi_1^{\alpha_1}\cdots \xi_d^{\alpha_d} 
\]
with coefficients in $\mathrm{CH}^{d+1}( \mathcal{M} ) \otimes S^*_{L,d+1}$   is a Jacobi form  of index $T$,  weight $1+\frac{n}{2}$,  and representation 
\[
\omega^*_{L,d+1}: \widetilde{\Gamma}_{d+1} \to \GL(S_{L,d+1}^*) .
\]
Here  Jacobi modularity is understood,  as in Theorem \ref{BigThm:generic_modularity},  after applying any $\Q$-linear functional  $\mathrm{CH}^{d+1}( \mathcal{M} ) \to \C$.
\end{proposition}

 Proposition \ref{prop:jacobi generating}, when combined  with the main result of \cite{BWR}, is already enough to show   that 
\[
 \sum_{ T' \in \Sym_{d+1} (\Q) } \mathcal{C}(T') \cdot q^{T'} 
\]
 is the $q$-expansion of a Siegel modular form of representation $\omega^*_{L,d+1} $.  See   Theorem \ref{thm:modularity} and its proof for details.


\subsection{Auxiliary cycle classes}
\label{ss:aux}


In this subsection we work with a fixed positive semi-definite $T\in \Sym_d(\Q)$ and 
$
\mu=(\mu_1,\ldots, \mu_d) \in (L^\vee/L)^d .
$

Given a $T'\in \Sym_{d+1}(\Q)$ with   upper left $d\times d$ block $T$,  and a  $ \mu_{d+1} \in  L^\vee/L $, there is a morphism 
\begin{equation}\label{drop one endo}
\mathcal{Z}  (  T' ,  \mu' ) \to \mathcal{Z} (T, \mu )
\end{equation}
of special cycles on $\mathcal{M}$, where $\mu'=( \mu_1,\ldots,\mu_{d+1})$.
This morphism sends an $S$-valued point 
\[
(x_1,\ldots, x_{d+1}) \in V_{\mu_1}(\mathcal{A}_S) \times \cdots \times  V_{\mu_{d+1}}(\mathcal{A}_S) 
\]
 of the source to its truncation
\[
(x_1,\ldots, x_d) \in V_{\mu_1}(\mathcal{A}_S) \times \cdots \times  V_{\mu_d}(\mathcal{A}_S) .
\]
The morphism \eqref{drop one endo} is finite and unramified, by the same argument as in the proof of \cite[Proposition 2.7.2]{AGHMP17}.

Recall from \eqref{W space} that $T$ determines a positive definite quadratic space $W$ of dimension $\mathrm{rank}(T)$, together with distinguished generators  $e_1,\ldots, e_d \in W$. 
As in \eqref{W special embed}, a functorial point  $S\to \mathcal{Z}(T,\mu)$ determines   an isometric embedding 
\begin{equation}\label{W special embed 2}
W\map{e_i \mapsto x_i}  V(\mathcal{A}_S )_\Q.
\end{equation}
This allows us to define a family of finite unramified stacks
\begin{equation}\label{first aux cycle}
\mathcal{Y} (m,\mu_{d+1},w)\to \mathcal{Z}(T ,\mu )
\end{equation}
indexed by  $m\in \Q$,   $\mu_{d+1}\in L^\vee/L$,   and $w\in W$, with functor of points
\[
\mathcal{Y}(m,\mu_{d+1},w) (S) 
= \left\{x_{d+1}\in V_{\mu_{d+1}}(\mathcal{A}_S):  \begin{array}{c} Q(x_{d+1}-w) = m \\ 
{ [x_{d+1},e_i] = [w,e_i] \ \forall  1\le i \le d }
  \end{array} \right\}.
\]
Note that  the conditions $[ x_{d+1} , e_i] = [w,e_i]$ are equivalent to $w$ being the orthogonal projection of $x_{d+1}$ to $W \subset V(\mathcal{A}_S)_\Q$.

The following Proposition shows that the new stacks \eqref{first aux cycle} are not really new at all.
They are  special cycles we already know, but indexed in a different way.

\begin{proposition}\label{prop:divisor is cycle}
There is an isomorphism of $\mathcal{Z}(T,\mu)$-stacks
\begin{equation}\label{divisor is cycle}
  \mathcal{Y}(m,\mu_{d+1},w) \iso \mathcal{Z} (T'  , \mu' ),
\end{equation}
where  $\mu'=(\mu_1,\ldots, \mu_{d+1})$, and 
\[
T' = \begin{pmatrix}
T&\frac{\alpha}{2}\\
\frac{{}^t \alpha}{2} & m + Q(w)
\end{pmatrix} ,
\]
 for the column vector $\alpha\in \Q^d$  with components $\alpha_i = [w,e_i]$.
Moreover, 
\[
m\ne 0 \iff \mathrm{rank}(T') = \mathrm{rank}(T)+1,
\]
 and when these conditions hold both \eqref{drop one endo} and  \eqref{first aux cycle} are generalized Cartier divisors  (Definition \ref{def:general cartier}).
\end{proposition}

\begin{proof}
 Let $W'$ be the quadratic space of dimension $\mathrm{rank}(T')$ determined by $T'$, exactly as the space $W$ of  \eqref{W space} was determined by $T$. 
As we do not assume that $T'$ is positive semi-definite, the quadratic space $W'$ need not be positive definite, but it is  nondegenerate (by construction).
   There are  distinguished vectors $e_1,\ldots, e_{d+1}$ that span $W'$, the vectors $e_1,\ldots, e_d$ span a positive definite subspace $W\subset W'$ isometric to \eqref{W space},  and the tuples
\[
e =(e_1,\ldots, e_d) \in W^d
\quad\mbox{and}\quad
e' =(e_1,\ldots, e_{d+1}) \in (W')^{d+1}
\]
satisfy $Q(e)=T$ and  $Q(e') = T'$.  Using the relation between $T'$ and $w$,  one checks  first that  $w$ is the orthogonal projection of $e_{d+1}$ to $W$, and then that 
\[
0 = [ e_{d+1}-w,w] = Q(e_{d+1}) -Q( e_{d+1}-w) - Q(w) .
\]
In particular
\begin{equation}\label{proj length}
Q( e_{d+1}-w) =m.
\end{equation}

Now we construct the isomorphism \eqref{divisor is cycle}.  
Fix an $\mathcal{M}$-scheme $S$ and a tuple 
\[
x=(x_1,\ldots, x_d) \in V_{\mu_1}(\mathcal{A}_S) \times \cdots \times V_{\mu_d}(\mathcal{A}_S).
\]
with moment matrix $Q(x)=T$.   This determines  a  point   $x\in \mathcal{Z}(T,\mu)(S)$, and an isometric embedding $W \to V(\mathcal{A}_S)_\Q$ by \eqref{W special embed 2}.

A lift of $x$ to $\mathcal{Z}(T',\mu')(S)$ determines a special quasi-endomorphism
\[
x_{d+1} \in V_{\mu_{d+1}} (\mathcal{A}_S),
\]
which then determines an extension of $W \to V(\mathcal{A}_S)_\Q$ to 
\[
W' \map{e_i \mapsto x_i} V(\mathcal{A}_S)_\Q.
\]
The calculation \eqref{proj length} shows that  $Q(x_{d+1}-w)=m$, and combining this with
$
[x_{d+1} , e_i] = [x_{d+1},x_i] = \alpha_i 
$
shows that $x_{d+1}$ defines a lift of $x$ to  $\mathcal{Y}(m,\mu_{d+1},w)(S)$.

Conversely,  any lift of $x$ to $\mathcal{Y}(m,\mu_{d+1},w)(S)$ corresponds to an $x_{d+1} \in V_{\mu_{d+1}} (\mathcal{A}_S)$ with the property that $Q(x_{d+1}-w)=m$, and the orthogonal projection of $x_{d+1}$ to 
$W \subset V(\mathcal{A}_S)_\Q$ is $w$.   An elementary linear algebra argument shows that 
$x'=(x_1,\ldots, x_{d+1})$ has moment matrix $T'$, so determines a lift of $x$ to $\mathcal{Z}(T',\mu' )(S)$.
This  establishes the isomorphism \eqref{divisor is cycle}.

If  $m=0$ then  \eqref{proj length} implies that  $e_{d+1}-w$ is an isotropic vector orthogonal to $W$, which is therefore contained in the radical of the quadratic form on $W'$.
  As this radical is trivial, $e_{d+1}=w\in W$.
It follows that  $W'=W$, and  $\mathrm{rank}(T') = \mathrm{rank}(T)$.

Now suppose   $m\neq 0$.  It follows from   \eqref{proj length} that  $e_{d+1}\neq w$,  hence $e_{d+1} \neq W$ and 
\[
\mathrm{rank}(T') = \dim(W') = \dim(W)+1=\mathrm{rank}(T)+1.
\]
It remains to show that \eqref{first aux cycle} is a generalized Cartier divisor. 
 In light of the isomorphism \eqref{divisor is cycle},  it suffices to show  the same for \eqref{drop one endo}.

By Remark \ref{rem:pos def T} we may assume that $T'$ is  positive semi-definite, for otherwise $\mathcal{Z}(T',\mu') = \emptyset$ and the claim is vacuous.
This assumption  implies that $W'$ is a positive definite quadratic space, and so \eqref{proj length} implies $m>0$.  In particular  $m+Q(w) >0$. 
By Proposition \ref{prop:divisors are divisors}, the right vertical arrow in 
\[
\xymatrix{
 { \mathcal{Z}(T,\mu) \times_\mathcal{M} \mathcal{Z}( m+Q(w), \mu_{d+1}  ) }  \ar[r]  \ar[d] &   {  \mathcal{Z}( m+Q(w) , \mu_{d+1}  )  }  \ar[d]   \\
   {  \mathcal{Z}(T,\mu)   } \ar[r]  & {   \mathcal{M}  } 
}
\]
is a generalized Cartier divisor. Proposition \ref{prop:naive intersection} allows us to realize $\mathcal{Z}(T',\mu')$ as an open and closed substack of the upper left corner, and so there exists an \'etale cover $U\to \mathcal{Z}(T,\mu)$ such that
$
  \mathcal{Z}(T',\mu')_U \to U
$
is a disjoint union of closed immersions $Z_i\to U$, each of which is locally defined by the vanishing of a section of $\co_U$. 
To see that $Z_i$ is an effective Cartier divisor on $U$, we must show that this section is not a zero divisor.  
This relies on the following lemma of commutative algebra.

\begin{lemma}
\label{lem:nzd cm}
Suppose that $S$ is a Cohen-Macaulay local Noetherian ring with maximal ideal $\mathfrak{m}$; then an element $a\in \mathfrak{m}$ is a non-zero divisor if and only if $\dim S/(a) = \dim S - 1$.
\end{lemma}
\begin{proof}
By Krull's Hauptidealsatz~\cite[\href{https://stacks.math.columbia.edu/tag/00KV}{Tag 00KV}]{stacks-project}, we have 
\[
\dim ( S/(a)  ) \ge \dim (S) - 1,
\]
 with equality holding exactly when $a$ is not contained in any minimal prime of $S$. On the other hand, saying that $a$ is a non-zero divisor is equivalent to saying that $a$ is not contained in any associated prime of $S$. Since $S$ is Cohen-Macaulay, its associated primes are precisely its minimal ones, and so the lemma follows.
\end{proof}

 Recall that Hypothesis~\eqref{hyp:low codimension} guarantees  that $\mathcal{Z}(T',\mu')$ has dimension 
\[
\dim( \mathcal{M} ) - \mathrm{rank}(T') = \dim (\mathcal{M} )- \mathrm{rank}(T) - 1 = \dim( \mathcal{Z}(T,\mu) ) - 1.
\]
As $\mathcal{Z}(T,\mu)$ is Cohen-Macaulay by Proposition \ref{prop:low key}, the desired conclusion now follows from Lemma~\ref{lem:nzd cm}.
\end{proof}

If $m\neq 0$,  Proposition \ref{prop:divisor is cycle} allows us to define, using Remark \ref{rem:cartier bundle}, 
\begin{equation}\label{first aux bundle}
[  \mathcal{Y}(m,\mu_{d+1},w) ]  \in \mathrm{CH}^1 ( \mathcal{Z}(T,\mu)  ) 
\end{equation}
as the  cycle class associated to the generalized Cartier divisor \eqref{first aux cycle}.
 More precisely, it is the first Chern class (Definition \ref{def:first chern class}) of the associated line bundle.

We next extend the definition to  $m=0$.  In this case the condition $Q(x_{d+1}-w) =0$ imposed in the definition of the domain of \eqref{first aux cycle}  can be satisfied by at most the  vector $x_{d+1}=w$, and so 
\[
\mathcal{Y}(0,\mu_{d+1},w) (S)
= \begin{cases}
\mathcal{Z}(T,\mu)(S)   & \mbox{if\ }  w \in V_{ \mu_{d+1} } ( \mathcal{A}_S ) \\
\emptyset & \mbox{otherwise.}
\end{cases}
\]
This  implies that the morphism \eqref{first aux cycle} is a closed immersion.  
 On the other hand,  Proposition~\ref{prop:divisor is cycle} tells us that there is a distinguished choice of $(T',\mu')$ for which $\mathrm{rank}(T) = \mathrm{rank}(T')$ and 
 \[
 \mathcal{Y}(0,\mu_{d+1},w) \iso \mathcal{Z}(T',\mu').
 \]
We deduce that for this choice of $(T',\mu')$ the morphism   \eqref{drop one endo} is  a closed immersion between stacks of the same dimension.
Now form the first Chern class 
\[
c_1\left(\omega^{-1}|_{\mathcal{Z}(T',\mu')}\right)\in \mathrm{CH}^1(\mathcal{Z}(T',\mu')),
\]
where $\omega \in \mathrm{Pic}(\mathcal{M})$ is the tautological bundle, and define
 \begin{equation}\label{constant aux class}
[  \mathcal{Y}(0,\mu_{d+1},w) ]   \in \mathrm{CH}^1(\mathcal{Z}(T,\mu))
\end{equation}
to be its push-forward via ~\eqref{drop one endo}.

Let  $Y (m,\mu_{d+1},w)$ be the generic fiber of \eqref{first aux cycle}, and let 
\begin{equation}\label{generic aux}
[ Y (m,\mu_{d+1},w) ]  \in \mathrm{CH}^1 ( Z(T ,\mu ) )
\end{equation}
be the restriction of \eqref{first aux bundle} and \eqref{constant aux class} to the generic fiber of $\mathcal{Z}(T,\mu)$.
The  following proposition is our version of   \cite[Proposition 2.6]{zhang-thesis}.
One should regard it as a corollary of a theorem of  Borcherds \cite{Bor:GKZ}.

\begin{proposition}\label{prop:generic divisor generating}
The formal generating series 
\[
\sum_{   \substack{    m \in\Q  \\  w \in  W   \\  \mu_{d+1} \in L^\vee/L    }  }
  [ Y ( m , \mu_{d+1} , w ) ]  \otimes \phi^*_ {\mu_{d+1}} 
\cdot 
 q^{ m+ Q(w) }   \xi_1^{   [ w , e_1]  } \cdots \xi_{d}^{   [ w , e_d]  } 
\]
converges\footnote{Convergence is understood in the sense of Theorem \ref{BigThm:generic_modularity}.
That is to say,  after applying any $\C$-linear functional $\mathrm{CH}^1(Z(T,\mu)) \to \C$.} to  a holomorphic function
\[
\phi_T(\tau', z) :  \mathcal{H} \times \C^d \to \mathrm{CH}^1(Z(T,\mu)) \otimes S_L^*. 
\]
The corresponding  function \eqref{jacobi augment} on $\mathcal{H}_{d+1}$ satisfies 
\[
\Phi_T(\gamma \cdot \tau) =  j_\gamma (\tau)^{2+n}  \omega_{L}^*( \tilde{ \gamma} ) \cdot \Phi_T(\tau) 
\]  
for all elements  $( \gamma , j_\gamma) \in \widetilde{\Gamma} \subset \widetilde{\Gamma}_{d+1}$  as in \eqref{meta element}.
\end{proposition}

\begin{proof}
The claim is vacuously true if $Z(T,\mu)$ is empty.
Hence, as in \S \ref{ss:special shimura}, we may fix an orthogonal decomposition
$
V = V^\flat \oplus W
$
and use this to express
\begin{equation}\label{special shimura app}
   Z(T, \mu) =
\bigsqcup_{ g \in G^\flat(\Q)\backslash \Xi(T,\mu)/K } M^\flat_g .
\end{equation}
as a disjoint union of smaller Shimura varieties.
As in \eqref{prime sublattices}, each $g\in \Xi(T,\mu)$ determines lattices
$L_g^\flat \subset V^\flat$ and  $\Lambda_g \subset W$.

The Shimura variety  $M^\flat_g$ has its own tautological line bundle $\omega^\flat_g$, its own Kuga-Satake abelian scheme $A^\flat_g \to M^\flat_g$, and its own family of special cycles
 \[
 Z^\flat_g(m,\nu) \to M^\flat_g
 \]
 indexed by  $m\in \Q$ and $\nu \in  L_g^{\flat,\vee} / L^\flat_g$. 
 When $m\neq 0$ there is an associated class
 \[
 [ Z^\flat_g(m,\nu)] \in \mathrm{CH}^1(M^\flat_g) 
 \]
by Remark \ref{rem:cartier bundle} and Proposition \ref{prop:divisors are divisors}.  
When $m=0$ we define 
\[
 [ Z^\flat_g(0,\nu)] = \begin{cases}
  c_1( \omega^{\flat,-1}_g) & \mbox{if }\nu =0 \\
 0 & \mbox{otherwise}.
 \end{cases}
\]

\begin{remark}
At a prime $p\not\in \Sigma$ the lattice $L_g^\flat$ may be far from maximal.
Fortunately, we have no need for any integral model of $M_g^\flat$ over $\Z[\Sigma^{-1}]$.
All constructions and proofs from \S \ref{ss:integral model} and \S \ref{ss:cycle basics} can be carried out (usually with less effort) directly  on the canonical model  over $\Q$. 
\end{remark}

 \begin{lemma}\label{lem:aux divisor compare}
 For every $g\in \Xi(T,\mu)$  there is a canonical isomorphism of $M_g^\flat$-stacks
 \begin{equation}\label{divisor comparison 1}
Y(m,\mu_{d+1},w) \times_{Z(T,\mu)}  M^\flat_g \iso 
 \bigsqcup_{\substack{  \nu \in L_g^{\flat,\vee} / L^\flat_g \\ \nu+w \in g  \cdot( \mu_{d+1} +L_{\widehat{\Z}}) } } 
 Z^\flat_g(m,\nu) .
\end{equation}
Here we regard $\mu_{d+1} +L_{\widehat{\Z}} \subset V_{\A_f}$.  Moreover, 
 \begin{equation}\label{divisor comparison}
[Y(m,\mu_{d+1},w)] \vert_{M^\flat_g} =
 \sum_{\substack{  \nu \in L_g^{\flat,\vee} / L^\flat_g \\ \nu+w \in g  \cdot( \mu_{d+1} +L_{\widehat{\Z}}) } } 
[  Z^\flat_g(m,\nu) ] \in \mathrm{CH}^1(M^\flat_g)
\end{equation}
where the left hand side is the projection of \eqref{generic aux} to the $g$-summand in 
\[
 \mathrm{CH}^1 ( Z(T ,\mu ) ) \iso \bigoplus_{g \in G^\flat(\Q)\backslash \Xi(T,\mu)/K }\mathrm{CH}^1(M_g^\flat).
\] 
 \end{lemma}

 \begin{proof}
Applying the  proof of Proposition \ref{prop:naive pullback} to the morphism $M^\flat_g \to M$ of Remark \ref{twisted shimura embedding}, 
we see that that  for any $M^\flat_g$-scheme $S$ there is a canonical isometry
\[
V(A_S)_\Q = V(A^\flat_{g,S})_\Q \oplus W, 
\]
which restricts to a bijection
 \begin{align}\label{eqn:mu to nu decomp}
   V_{\mu_{d+1}}( A_S) =
    \bigsqcup_{  \substack{ {\nu\in L_g^{\flat,\vee}/L_g^\flat }  \\ 
    { \lambda\in \Lambda_g^\vee/\Lambda_g} \\ 
    { \nu+\lambda \in  g \cdot ( \mu_{d+1} + L_{\widehat{\Z}} )  } }} 
    V_\nu(A^\flat_{g,S})\times (\lambda + \Lambda_g).
 \end{align}
 This isometric embedding $W \to V(A_S)_\Q$ agrees with that of \eqref{W special embed 2}.
 
A $S$-point  of  the left hand side of \eqref{divisor comparison 1}
is an $S$-point of $M_g^\flat \subset Z(T,\mu)$, together with a special quasi-endomorphism 
 \[
 x_{d+1}\in V_{\mu_{d+1}}(A_S)
 \]
 whose orthogonal projection to $W$ is $w$, and such that  $Q(x_{d+1}-w) = m$. 
   Using~\eqref{eqn:mu to nu decomp}, we find that there is a unique pair of cosets
 \[
 \nu\in L_g^{\flat,\vee}/L_g^\flat \quad \mbox{and}\quad  \lambda\in \Lambda_g^\vee/\Lambda_g
 \]
such that $\nu +\lambda \in g\cdot (\mu_{d_1}+L_{\widehat{\Z}})$, and such that 
\[
 x_{d+1}-w \in  V_\nu(A^\flat_{g,S})      \quad \mbox{and}\quad    w\in \lambda + \Lambda_g.
\]
In particular, $x_{d+1}-w$ determines an $S$-point of $Z_g^\flat(m,\nu)$.

This construction establishes the  isomorphism \eqref{divisor comparison 1}, from which  \eqref{divisor comparison} follows directly. We note that when $m=0$
 both sides of  \eqref{divisor comparison} are equal to
 \[
 \begin{cases}
 c_1(\omega^{-1}_{M^\flat_g}) & \mbox{if }w \in g \cdot (\mu_{d+1} +  L_{\widehat{\Z}} ) \\
 0 & \mbox{otherwise.} 
 \end{cases} \qedhere
 \]
  \end{proof}

Dualizing the tautological map $S_{\Lambda_g} \otimes S^*_{\Lambda_g} \to \C$ yields a homomorphism
\begin{equation}\label{dual contraction}
\C \to S^*_{\Lambda_g} \otimes S_{\Lambda_g}
\end{equation}
sending $1\mapsto \sum_{ w\in \Lambda_g^\vee / \Lambda_g } \phi^*_w \otimes \phi_w$.
On the other hand, if we abbreviate
\[
L_g = V \cap g L_{\widehat{\Z}},
\]
we may use the inclusions
\begin{equation*}
L^\flat_g \oplus \Lambda_g \subset L_g \subset L_g^\vee \subset L_g^{\flat,\vee} \oplus \Lambda_g^\vee
\end{equation*}
to define a homomorphism
$
 S^*_{L^\flat_g}   \otimes  S^*_{\Lambda_g}   \to S_{L_g}^*
$
by 
\[
\phi^*_\nu \otimes \phi^*_w \mapsto \begin{cases}
\phi^*_{ \nu + w } & \mbox{if }  \nu+w \in L_g^\vee\\
0 & \mbox{otherwise}
\end{cases}
\]
for all $\nu \in L_g^{\flat,\vee}/L^\flat_g$ and $w\in \Lambda_g^\vee/\Lambda_g$.
This defines the second arrow in the $\widetilde{\Gamma}_1$-equivariant composition
\begin{equation}\label{tensor zigzag}
S^*_{L^\flat_g}  \map{  \eqref{dual contraction}} S^*_{L^\flat_g} \otimes S^*_{\Lambda_g} \otimes S_{\Lambda_g}
 \to S^*_{L_g}    \otimes  S_{\Lambda_g} \to S^*_L    \otimes  S_{\Lambda_g}
\end{equation}
The third arrow is defined using  the isomorphism $S^*_{L_g}  \iso S^*_L $ induced by $g^{-1} : L^\vee_g/L_g \to L^\vee/L$.


 It is a theorem of Borcherds \cite{Bor:GKZ} that the formal generating series
\[
 \sum_{ \substack{    m \ge 0 \\ \nu \in  L_g^{\flat,\vee} / L^\flat_g  } }
 [ Z^\flat_g(m,\nu) ] \otimes \phi^*_\nu \cdot q^m
\]
 with coefficients in   $ \mathrm{CH}^1( M^\flat_g) \otimes S^*_{ L^\flat_g}$ is a modular form on $\mathcal{H}$ of weight $1+\frac{n^\flat}{2}$ and representation 
 \[
 \omega^*_{L^\flat_g} : \widetilde{\Gamma} \to \GL(  S^*_{L^\flat_g} ) .
 \]
 More precisely, as Borcherds works only in the complex fiber, one should use   \cite[Theorem B]{HMP} for the corresponding modularity statement on the canonical model.
 
Applying  \eqref{tensor zigzag}   to this last generating series coefficient-by-coefficient  yields the formal generating series 
 
  \begin{equation}\label{third divisor modularity} 
  \sum_{ \substack{    m \ge 0 \\ \nu \in  L_g^{\flat,\vee} / L^\flat_g   } }
 \sum_{   \substack{    \mu_{d+1} \in L^\vee/L \\  w \in  \Lambda^\vee_g/\Lambda_g  \\ \nu+w \in  g\cdot( \mu_{d+1} + L_{\widehat{\Z}})  }  }
[ Z^\flat_g ( m , \nu ) ]  \otimes   \phi^*_{\mu_{d+1}} \otimes \phi_ w \cdot  q^m 
\end{equation}
with coefficients  in $\mathrm{CH}^1(M^\flat_g) \otimes S^*_L \otimes S_{ \Lambda_g }$, which is therefore a modular form on $\mathcal{H}$ of weight $1+\frac{n^\flat}{2}$ and representation 
 \[
 \omega^*_{L} \otimes \omega_{ \Lambda_g  }  : \widetilde{\Gamma} \to \GL( S^*_{L} \otimes S_{ \Lambda_g } ) .
 \]

  

Consider the  theta function 
$
 \mathcal{H} \times \C^d \to S_{\Lambda_g}^*
$
defined by
\begin{equation}\label{generalized theta}
\vartheta_w (\tau',z) =  \sum_{   w  \in  \Lambda_g^\vee    }   
 \phi^*_w \cdot 
 q^{Q(w) }   \xi_1^{   [ w , e_1]  } \cdots \xi_d^{   [ w , e_d]  } .
\end{equation}
If, as in Definition \ref{def:jacobi form}, we define a function on $\mathcal{H}_{d+1}$ by 
\[
\Theta_w( \tau) = \vartheta_w (\tau',z)\cdot e^{2\pi i \mathrm{Tr}( T\cdot  \tau'' ) }
\]
 then \cite[Lemma 2.8]{zhang-thesis} implies the equality
 \[
 \Theta_w(  \gamma \cdot \tau) = j_\gamma(\tau)^{n-n^\flat}  \omega^*_{\Lambda_g}( \tilde{\gamma}) \cdot \Theta_w (\tau)
 \]
 for all \eqref{meta element} in the subgroup $\widetilde{\Gamma} \subset \widetilde{\Gamma}_{d+1}$.   

Now use the tautological pairing  
$
S_{\Lambda_g} \otimes S_{\Lambda_g}^* \to \C
$
 to multiply  \eqref{third divisor modularity} with  \eqref{generalized theta}.
 Lemma \ref{lem:aux divisor compare} implies that the resulting generating series is 
\[
 \sum_{   \substack{    m \in\Q  \\  w \in  W   \\  \mu_{d+1} \in L^\vee/L    }  }
[ Y ( m , \mu_{d+1} , w ) ]\rvert_{M^\flat_g}  \otimes \phi^*_ {\mu_{d+1}} 
\cdot 
 q^{ m+ Q(w) }   \xi_1^{   [ w , e_1]  } \cdots \xi_d^{   [ w , e_d]  } ,
\]
which  therefore satisfies the transformation  law stated in Proposition \ref{prop:generic divisor generating}.    
Varying $g$ and using  \eqref{special shimura app} completes the proof of that proposition.
\end{proof}

\begin{corollary}\label{cor:integral divisor generating}
The generating series 
\[
\sum_{   \substack{    m \in\Q  \\  w \in  W   \\  \mu_{d+1} \in L^\vee/L    }  }
[ \mathcal{Y} ( m , \mu_{d+1} , w ) ]  \otimes \phi^*_ {\mu_{d+1}} 
\cdot 
 q^{ m+ Q(w) }   \xi_1^{   [ w , e_1]  } \cdots \xi_{d}^{   [ w , e_d]  } 
\]
 defines a holomorphic function 
\[
\mathcal{H} \times \C^d \to  \mathrm{CH}^1(\mathcal{Z}(T,\mu)) \otimes S_L^*
\]
satisfying the same transformation law under $\widetilde{\Gamma} \subset \widetilde{\Gamma}_{d+1}$ as the generating series of 
Proposition \ref{prop:generic divisor generating}.
\end{corollary}

\begin{proof}
Proposition \ref{prop:low key} (which applies, thanks to Hypothesis  \ref{hyp:low codimension}) implies the injectivity  of the restriction map  
\[
\mathrm{CH}^1(\mathcal{Z}(T,\mu)) \to \mathrm{CH}^1(Z(T,\mu)),
\]
and so the claim is  immediate from Proposition \ref{prop:generic divisor generating}. 
\end{proof}


\subsection{Proof of Proposition \ref{prop:jacobi generating}}


Proposition \ref{prop:jacobi generating} is vacuously true if $T\in \Sym_d(\Q)$  is not positive semi-definite.  Indeed, in this case $\mathcal{Z}(T,\mu)= \emptyset$ by Remark \ref{rem:pos def T}, and it follows from \eqref{eqn:def_corrected_class} and \eqref{drop one endo}  that the generating series of Proposition \ref{prop:jacobi generating} vanishes coefficient-by-coefficient.
Thus we may  assume, as in \S \ref{ss:aux}  that $T\in \Sym_d(\Q)$ is  positive semi-definite, and let  $e_1,\ldots, e_d  \in W$ be as in  \eqref{W space}.

By Hypothesis \ref{hyp:low codimension} and Proposition \ref{prop:low key}, 
for any $\mu=(\mu_1,\ldots, \mu_d) \in (L^\vee/L)^d$   the canonical finite unramified map
\[
f^{(T,\mu)}  :  \mathcal{Z}(T,\mu) \to \mathcal{M} 
\]
has equidimensional image of codimension $\mathrm{rank}(T)$, and so induces (Proposition \ref{prop:chow pushforward}) a pushforward
\[
f^{(T,\mu)}_*: 
 \mathrm{CH}^1  (\mathcal{Z}(T ,\mu))  \to \mathrm{CH}^{ \mathrm{rank}(T) + 1 }  (\mathcal{M}).
\]
The following lemma relates the images of the cycle classes \eqref{first aux bundle}  and \ref{constant aux class} under this map to the coefficients  appearing in Proposition \ref{prop:jacobi generating}.

\begin{lemma}\label{lem:first relation}
For any  $m\in \Q$,   $w \in W$, and $\mu_{d+1} \in L^\vee/L$ we have the equality
\[
\underbrace{ c_1(\omega^{-1}) \cdots c_1(\omega^{-1}) }_{ d  - \mathrm{rank}(T ) } \cdot 
f^{(T,\mu)}_*  [ \mathcal{Y}(m,\mu_{d+1},w) ]  
=   \mathcal{C}( T' , \mu') 
\]
in $ \mathrm{CH}^{d+1}(\mathcal{M})$.  On the right,   $T' \in \Sym_{d+1}(\Q)$ and $\mu' \in ( L^\vee/ L )^{d+1}$ have the same meaning   as in  Proposition~\ref{prop:divisor is cycle}.
 \end{lemma}

 \begin{proof}
 First suppose  $m\neq 0$, so that   $\mathrm{rank}(T') = \mathrm{rank}(T)+1$ by Proposition~\ref{prop:divisor is cycle}.
It  follows directly from the definitions and the commutative diagram
\[
\xymatrix{
{ \mathcal{Y}(m,\mu_{d+1},w)   }  \ar[dr]_{\eqref{first aux cycle}} \ar@{=}[rr]^{\eqref{divisor is cycle}} & & { \mathcal{Z}  ( T' ,  \mu' )  }  \ar[dl]^{\eqref{drop one endo} }\\ 
&  { \mathcal{Z}  ( T  ,  \mu )  }  \ar[d]_{f^{(T,\mu)}}   \\
& {\mathcal{M}} 
}
\]
 that
\[
f^{(T,\mu)}_* [ \mathcal{Y}(m,\mu_{d+1},w)]  
= [\mathcal{Z}(T',\mu') ]  
\]
in the codimension $\mathrm{rank}(T)+1$ Chow group of $\mathcal{M}$, and 
intersecting both sides with $d-\mathrm{rank}(T)$ copies of  $c_1(\omega^{-1})$ proves the claim.

Suppose now that $m=0$, so that  $\mathrm{rank}(T') = \mathrm{rank}(T)$ by Proposition~\ref{prop:divisor is cycle}.
In this case 
\[
f^{(T,\mu)}_*  [ \mathcal{Y}(0,\mu_{d+1},w) ] 
=
f^{(T',\mu')}_*    \left( \omega^{-1}|_{\mathcal{Z}(T' , \mu' )}  \right)
 = c_1(\omega^{-1})\cdot [ \mathcal{Z}(T' ,\mu' )] 
\]
holds in the codimension $\mathrm{rank}(T)+1$ Chow group, 
where the first equality is by the definition of \eqref{constant aux class}, and the second is by Proposition~\ref{prop:projection formula}.   
Once again, the claim follows. 
 \end{proof}

\begin{lemma}
\label{lem:C_nonzero_criterion}
Suppose $m\in \Q$ and $\alpha\in \Q^d$.  If 
\[
\mathcal{C}  \left(\begin{matrix}  T &  \frac{\alpha}{2} \\  \frac{{}^t \alpha}{2}   & m\end{matrix}\right) \neq 0
\]
then there exists a unique $w\in W$ such that $\alpha_i = [w,e_i]$  for all $i=1,\ldots,d$.
\end{lemma}

\begin{proof}
Our assumption implies that there is some $\mu' \in (L^\vee/L)^{d+1}$ for which 
\[
\mathcal{Z}\left(\begin{pmatrix}  T &  \frac{\alpha}{2} \\  \frac{{}^t \alpha}{2}   & m\end{pmatrix},
\mu ' \right) \neq \emptyset.
\]
Any non-empty scheme $S$ mapping to it determines special quasi-endomorphisms
\[
 x_1,\ldots,x_{d+1}  \in V(\mathcal{A}_S)_\Q.
\]
The first $d$-coordinates $x=(x_1,\ldots, x_d)$ satisfy $Q(x)=T$, and so determine an isometric embedding 
 \[
 W\xrightarrow{e_i\mapsto x_i}V(\mathcal{A}_S)_\Q . 
 \]
Using the relation  $[x_{d+1},x_i] = \alpha_i$  for $i=1,\ldots,d$, we see that the orthogonal projection 
 of $x_{d+1}$ to $W \subset V(\mathcal{A}_S)_\Q$ is a vector  $w\in W$ with the desired properties.
 The uniqueness is clear, as $e_1,\ldots, e_d$ span the positive definite quadratic space $W$.
\end{proof}

\begin{lemma}\label{lem:weil cycle}
For any $m\in \Q$ and column vectors   $\alpha \in \Q^d$ and $x,y \in \Z^d$ we have
\[
\omega_{L,d+1}^*(\tilde{\gamma}) \cdot 
  \mathcal{C}  \begin{pmatrix}  T&\frac{\alpha}{2}  \\ \frac{{}^t \alpha}{2} & m  \end{pmatrix}   
      = 
  e^{2\pi i \, {}^tx \alpha  } \cdot
    \mathcal{C}   \begin{pmatrix}  T& T  y + \frac{\alpha}{2}  \\   {}^ty \, {}^tT + \frac{{}^t \alpha}{2} & 
    {}^ty \,T \, y + {}^t y\,\alpha + m  \end{pmatrix} ,
\]
where 
\begin{equation}\label{uni gamma}
\tilde{\gamma} = \left( \mathrm{id}_{  \widetilde{\Gamma}}  , \begin{pmatrix} {}^tx \\ {}^ty \end{pmatrix} \right)
\in  \widetilde{\Gamma} \imes M_{2,d}(\Z) =  \widetilde{J}_{d+1} .
\end{equation}
\end{lemma}

\begin{proof}
By the explicit formulas for \eqref{weil},  if we set
\[
A= \begin{pmatrix}  I_d &  y  \\ 0 & 1 \end{pmatrix}  \in \GL_{d+1}(\Z)
\]
then  $\tilde{\gamma}$  acts on $S^*_{L,d+1}$ as 
\[ 
\omega_{L,d+1}^*(\tilde{\gamma}) \cdot \phi_{\mu' }^*  
=     e^{2\pi i x_1 [ \mu_1 , \mu_{d+1} ] }  \cdots   e^{2\pi i x_d [ \mu_d , \mu_{d+1} ] }    \cdot \phi_{  \mu' A }^* 
\]
for any $\mu' \in (L^\vee / L)^{d+1}$.  It follows that 
\[
  \mathcal{C}  \left(  \begin{pmatrix}  T&\frac{\alpha}{2}  \\ \frac{{}^t \alpha}{2} & m  \end{pmatrix}  , \mu' \right)  \otimes  \left(   \omega_{L,d+1}^*(\tilde{\gamma}) \cdot  \phi^*_{\mu'}    \right)
  =   e^{2\pi i \, {}^tx \alpha }    
   \mathcal{C} \left(   \begin{pmatrix}  T&\frac{\alpha}{2}  \\ \frac{{}^t \alpha}{2} & m  \end{pmatrix}   
   ,  \mu'  \right)  \otimes\phi^*_{\mu'  A}.   
\]
Indeed, the essential point here is that  
\[
  \mathcal{C} \left(  \begin{pmatrix}
T&\frac{\alpha}{2}\\
\frac{{}^t \alpha}{2} & m 
\end{pmatrix}  , \mu' \right)  \neq  0 \implies 
    \mathcal{Z}   \left( \begin{pmatrix}
T&\frac{\alpha}{2}\\
\frac{{}^t \alpha}{2} & m 
\end{pmatrix}  ,\mu' \right) \neq \emptyset,
\]
which implies,  using  Remark \ref{rem:special congruence}, that
$[ \mu_i ,  \mu_{d+1} ] \equiv \alpha_i \pmod{\Z} $.

The preceding paragraph allows us to compute
\begin{align*}
\omega_{L,d+1}^*(\tilde{\gamma}) \cdot 
  \mathcal{C}  \begin{pmatrix}  T&\frac{\alpha}{2}  \\ \frac{{}^t \alpha}{2} & m  \end{pmatrix}  
  & = 
     e^{2\pi i \, {}^tx \alpha }     \sum_{ \mu' \in (L^\vee/L)^{d+1}} 
   \mathcal{C} \left(   \begin{pmatrix}  T&\frac{\alpha}{2}  \\ \frac{{}^t \alpha}{2} & m  \end{pmatrix}   
   ,  \mu'  \right)  \otimes\phi^*_{\mu'  A}  \\
     & =     e^{2\pi i \, {}^tx \alpha }     \sum_{ \mu' \in (L^\vee/L)^{d+1}}   
   \mathcal{C} \left(   \begin{pmatrix}  T&\frac{\alpha}{2}  \\ \frac{{}^t \alpha}{2} & m  \end{pmatrix}   
   ,  \mu' A^{-1}  \right)  \otimes\phi^*_{\mu' }  \\
     & =     e^{2\pi i \, {}^tx \alpha }       \sum_{ \mu' \in (L^\vee/L)^{d+1}}   
   \mathcal{C} \left(  {}^tA \begin{pmatrix}  T&\frac{\alpha}{2}  \\ \frac{{}^t \alpha}{2} & m  \end{pmatrix}   A
   ,  \mu'   \right)  \otimes\phi^*_{\mu' }  \\
      & = 
   e^{2\pi i \, {}^tx \alpha }    \cdot 
   \mathcal{C} \left(  {}^tA \begin{pmatrix}  T&\frac{\alpha}{2}  \\ \frac{{}^t \alpha}{2} & m  \end{pmatrix}   A
    \right)   .
\end{align*}
In the third equality we have used the linear invariance of special cycles proved in  Proposition~\ref{prop:naive linear invariance}, which implies the same invariance  for the corrected cycle classes \eqref{eqn:def_corrected_class}.  
\end{proof}

\begin{proof}[Proof of Proposition \ref{prop:jacobi generating}]
For a fixed $\mu \in (L^\vee/L)^d$,  consider the generating series 
\[
 \sum_{   \substack{    m \in\Q  \\  w \in  W   \\  \mu_{d+1} \in L^\vee/L    }  }
f^{(T,\mu)}_* [ \mathcal{Y}( m , \mu_{d+1} , w ) ]  \otimes \phi^*_ {\mu_{d+1}} 
\cdot 
 q^{ m+ Q(w) }   \xi_1^{   [ w , e_1]  } \cdots \xi_d^{   [ w , e_d]  } .
\]
This agrees with the  pushforward via $\mathcal{Z}(T,\mu) \to \mathcal{M}$  of the  generating series of   Corollary \ref{cor:integral divisor generating}, and so converges to a holomorphic function 
\[
\mathcal{H} \times \C^d \to \mathrm{CH}^{\mathrm{rank}(T) +1}(\mathcal{M}) \otimes S_L^*
\]
transforming under  $\widetilde{\Gamma} \subset \widetilde{J}_{d+1}$ as in   Proposition  \ref{prop:generic divisor generating}.

The linear map $S_L^*  \to S^*_{L,d} \otimes S^*_{L} \iso S^*_{L,d+1}$ sending
\[
\phi^*_{\mu_{d+1}} \mapsto
\phi^*_{\mu}  \otimes\phi^*_{\mu_{d+1}} 
\]
is $\widetilde{\Gamma}$-equivariant, where the action on the source is via $\omega^*_L$, and the action on the target is via the restriction of $\omega^*_{L,d+1}$ to $\widetilde{\Gamma} \subset \widetilde{\Gamma}_{d+1}$.
Applying this map to the above generating series, summing over $\mu$, and using 
Lemma \ref{lem:first relation}, we find that 
 \begin{equation}\label{pre cycle jacobi}
 \sum_{   \substack{    m \in\Q  \\  w \in  W      }  }
  \mathcal{C} \begin{pmatrix}
T&\frac{\alpha}{2}\\
\frac{{}^t \alpha}{2} & m + Q(w)
\end{pmatrix}  
\cdot 
 q^{ m+ Q(w) }   \xi_1^{   \alpha_1  } \cdots \xi_d^{   \alpha_d } 
\end{equation}
(inside the sum, $\alpha\in \Q^d$ has components $\alpha_i = [ w,e_i]$)
defines a holomorphic function 
 \[
 \mathcal{H} \times \C^d \to 
 \mathrm{CH}^{d+1}(\mathcal{M}) \otimes S_{L,d+1}^*
 \]
   transforming under the subgroup $\widetilde{\Gamma} \subset \widetilde{\Gamma}_{d+1}$ in the same way as a  Jacobi form of  index $T$,  weight $1+\frac{n}{2}$,  and representation 
\[
 \omega^*_{L,d+1}:\widetilde{\Gamma}_{d+1}  \to \GL(S_{L,d+1}^*).
\]
Using the change of variables $m\mapsto m - Q(w)$ and  Lemma \ref{lem:C_nonzero_criterion}, 
we may rewrite \eqref{pre cycle jacobi} as
\[
  \phi_T(\tau',z) \define
 \sum_{   \substack{    m \in\Q  \\  \alpha\in \Q^d    }  }
  \mathcal{C}  \begin{pmatrix}
T&\frac{\alpha}{2}\\
\frac{{}^t \alpha}{2} & m 
\end{pmatrix} 
\cdot 
 q^{ m  }   \xi_1^{   \alpha_1  } \cdots \xi_d^{  \alpha_d  } ,
\]
which therefore transforms  under  $\widetilde{\Gamma} \subset \widetilde{\Gamma}_{d+1}$ in the same way.

To complete the proof, it only remains to check that this function  also transforms correctly  under any  $\tilde{\gamma} \in M_{2,d}(\Z) \subset \widetilde{J}_{d+1}$, which we write in the form \eqref{uni gamma}.
Using Lemma \ref{lem:weil cycle},  an elementary manipulation of the sum defining $\phi_T$  shows  that 
\[
  \omega_{L,d+1}^*(\tilde{\gamma}) \cdot   \phi_T(\tau',z) 
  =
     e^{ 2\pi i   \tau'   (     {}^t yT  y  )      }
        e^{ -4 \pi i    (      {}^t z     + {}^tx  ) T y  }
  \cdot         \phi_T( \tau' , z -  y \tau'   + x ) .
\]
  Unpacking Definition \ref{def:jacobi form} shows that this is precisely the transformation law satisfied by a Jacobi form of the desired index, weight, and representation.
  \end{proof}



 
\section{Corrected  cycle classes}
\label{s:derived cycle classes}
 

Keep the quadratic lattice $L\subset V$ and the integral model  $\mathcal{M}$ over  $\Z[\Sigma^{-1}]$ as in  \S \ref{ss:initial data} and \S \ref{ss:integral model}.  
In this subsection we construct from the naive special cycles $\mathcal{Z}(T,\mu)$, which need not be equidimensional, canonical cycle classes $\mathcal{C}(T,\mu)$ in the Chow group of $\mathcal{M}$.  


 
\subsection{Construction of the classes}
\label{ss:cycle construction}
 

We will make essential use of Theorems \ref{thm:GS_K-theory} and \ref{thm:GS}, as well as Proposition \ref{prop:inductive coniveau}.
For this reason we assume, throughout the entirety of \S \ref{s:derived cycle classes},  that $\mathcal{M}$ is regular (see Proposition \ref{prop:regularity}).

Suppose $\mathcal{Z}$ is a Deligne-Mumford stack equipped with a finite morphism  $\pi: \mathcal{Z}\to \mathcal{M}$. 
In \S \ref{ss:chow} we associate to this data a $\Q$-vector space $G_0(\mathcal{Z})_\Q$,  with the  descending  filtration 
\[
F^d G_0(  \mathcal{Z} )_\Q =
 \bigcup_{  \substack{ 
 \mathrm{closed\ substacks\ }  \mathcal{Y} \subset \mathcal{Z} \\ 
 \mathrm{codim}_\mathcal{M}( \pi(\mathcal{Y}) ) \ge d
 }  }  
 \mathrm{Image}\big( G_0(\mathcal{Y})_\Q \to G_0( \mathcal{Z} )_\Q \big)
\]
from  \eqref{G filtration}.
By Remark \ref{rem:coherent_sheaf_class},   any coherent $\co_\mathcal{Z}$-module $\mathcal{F}$ determines an \[
[\mathcal{F}] \in G_0(\mathcal{Z})_\Q.
\]

Fix a   $T\in \Sym_d(\Q)$  with $d\ge 1$,  and a tuple of cosets 
$
\mu=(\mu_1,\ldots,\mu_d) \in (L^\vee/L)^d.
$
We have defined in \eqref{special cycle} a finite and unramified morphism 
\[
\mathcal{Z}(T,\mu) \to \mathcal{M}.
\]
If we denote by $t_1,\ldots, t_d \in \Q$  the diagonal entries of $T$, there are  forgetful morphisms 
$
\mathcal{Z}(T,\mu) \to  \mathcal{Z}(t_i,\mu_i),
$
  sending an $S$-point 
\[
x=(x_1,\ldots, x_d) \in V_{\mu_1}(\mathcal{A}_S) \times \cdots \times V_{\mu_d}(\mathcal{A}_S)
\] 
to its $i^\mathrm{th}$ coordinate $x_i\in V_{\mu_i}(\mathcal{A}_S)$.  
The product of these maps  defines a morphism of $\mathcal{M}$-stacks
\begin{equation}\label{open and closed}
\mathcal{Z}(T,\mu) \to
  \mathcal{Z}(t_1,\mu_1) \times_\mathcal{M} \cdots\times_\mathcal{M} \mathcal{Z}(t_d,\mu_d), 
\end{equation}
which  presents $\mathcal{Z}(T,\mu)(S)$ as the locus of $S$-points of the codomain for which the moment matrix $Q(x)$ of $x=(x_1,\ldots,x_d)\in\prod_i V_{\mu_i}(\mathcal{A}_S)$ is equal to  $T$.
The moment matrix   $Q(x)$ is   locally constant on $S$, and hence  \eqref{open and closed} is an open and closed immersion,  by  the rigidity lemma for endomorphisms of abelian schemes: if an endomorphism of $\mathcal{A}_S$ is $0$  at some point of $S$, then it is $0$ on the entire connected component of $S$ containing that point \cite[Corollary 6.2]{MFK}.

By iterating  the  pairing  of Proposition \ref{prop:inductive coniveau}, we obtain a multilinear map
\begin{equation}\label{multilinear classes}
\xymatrix{
{  F^1  G_0 ( \mathcal{Z}(t_1,\mu_1) )_\Q \otimes \cdots \otimes  F^1 G_0 ( \mathcal{Z}(t_d,\mu_d) )_\Q  } \ar[d]^{ z_1 \otimes \cdots \otimes z_d \mapsto z_1 \cap \cdots \cap z_d }  \\
{    F^d  G_0 \big( \mathcal{Z}(t_1,\mu_1)  \times_{\mathcal{M}} \cdots \times_{\mathcal{M}} \mathcal{Z}(t_d,\mu_d) \big) _\Q } \ar[d]  \\
{    F^d  G_0 (    \mathcal{Z}(T,\mu)  )_\Q  } , 
}
\end{equation}
where the second arrow is restriction via   \eqref{open and closed}.

The multilinear map just defined has a  distinguished input $z_1\otimes \cdots \otimes z_d$.
If   $(t_i,\mu_i) \neq (0,0)$ we define
\[
z_i  =  [ \co_{\mathcal{Z}(t_i,\mu_i)} ]   \in G_0 ( \mathcal{Z}(t_i,\mu_i) )_\Q.
\]
If  $(t_i,\mu_i) = (0,0)$ then  $\mathcal{Z}(0,0)=\mathcal{M}$ by Proposition \ref{prop:divisors are divisors},  and we define
\[
z_i = [ \co_{\mathcal{M}} ] - [\omega ]  \in G_0 ( \mathcal{M}  )_\Q,
\]
where $\omega$ is the tautological bundle  \eqref{taut bundle}.

\begin{lemma}
For all $1\le i \le d$ we have $z_i \in F^1 G_0 ( \mathcal{Z}(t_i,\mu_i) )_\Q.$
\end{lemma}

\begin{proof}
If $(t_i,\mu_i) \neq (0,0)$ then Proposition \ref{prop:divisors are divisors} implies that the image of $\mathcal{Z}(t_i , \mu_i ) \to \mathcal{M}$ is either empty or of codimension $1$. Hence 
\[
F^1 G_0 ( \mathcal{Z}(t_i,\mu_i) )_\Q = G_0 ( \mathcal{Z}(t_i,\mu_i) )_\Q,
\]
 and the claim is vacuously true.
If $(t_i,\mu_i) = (0,0)$ the claim follows from the proof of Lemma \ref{lem:pic complex}, especially the relation \eqref{line filter}.
\end{proof}

\begin{remark}
The motivation for the definition of $z_i$ in the case $(t_i,\mu_i) = (0,0)$ comes from Lemma \ref{lem:pic complex}, which implies that the image of $[ \co_{\mathcal{M}} ] - [\omega ]$ under 
\[
F^1 G_0( \mathcal{M} )_\Q  \map{ \eqref{more coniveau push} } F^1 K_0(\mathcal{M})_\Q \map{ \eqref{ChowGroth} } \mathrm{CH}^1(\mathcal{M})  
\]
is  the first Chern class $c_1(\omega^{-1})$.
\end{remark}

\begin{definition}\label{def:corrected classes}
The \emph{derived fundamental class} 
\[
[ \co^{\mathrm{derived}}_{ \mathcal{Z}(T,\mu)} ]   \in  F^d G_0 (    \mathcal{Z}(T,\mu)  )_\Q 
\]
 is the image of $z_1\otimes \cdots \otimes z_d$ under  \eqref{multilinear classes}. 
The \emph{corrected cycle class}
\[
 \mathcal{C}(T,\mu)  \in \mathrm{CH}^d_{  \mathcal{Z}(T,\mu) }( \mathcal{M} )
 \]
  is the image of the derived fundamental class   under the composition
\[
F^d G_0 (    \mathcal{Z}(T,\mu)  )_\Q  \map{\eqref{more coniveau push}} 
F^d K_0^{\mathcal{Z}(T,\mu)} ( \mathcal{M} )_\Q  \map{ \eqref{ChowGroth} }
\mathrm{CH}^d_{ \mathcal{Z}(T,\mu)} ( \mathcal{M} ).
\]
\end{definition}

\begin{remark}
\label{rem:fancy construction}
A somewhat fancier way to understand this construction is to calculate the fiber product
\[
\mathcal{Z}(t_1,\mu_1)\times_{\mathcal{M}}\times\cdots \times_{\mathcal{M}}\mathcal{Z}(t_d,\mu_d)
\]
in a \emph{derived} sense, and hence as a \emph{derived stack} over $\mathcal{M}$.

When forming this derived fiber product, one should interpret any  factor of the form  $\mathcal{Z}(t_i,\mu_i)=\mathcal{Z}(0,0)$ as itself being a derived stack. 
More precisely,  inside the total space of the cotautological line bundle on $\mathcal{M}$, one can construct the derived self-intersection of the zero section.
The result is a derived stack whose  underlying classical stack is canonically identified with $\mathcal{M}$, but with virtual dimension $\dim(\mathcal{M})-1$.  Every instance of $\mathcal{Z}(0,0)$  in the above fiber product should be replaced by this derived stack.

The underlying classical stack of this derived fiber product is of course just the classical fiber product, and the open and closed substack $\mathcal{Z}(T,\mu)$ of this classical stack lifts canonically to an open and closed substack of the derived fiber product, which we denote by $\mathcal{Z}^{\mathrm{der}}(T,\mu)$. 

By construction, $\mathcal{Z}^{\mathrm{der}}(T,\mu)$ is \emph{quasi-smooth} over $\mathcal{M}$ of virtual codimension $d$: this is just the derived analogue of a local complete intersection of codimension $d$. A general result of Khan~\cite[\S 6]{Khan2022-eq}---a derived generalization of results found in~\cite{sga6} for local complete intersections---now shows that the structure sheaf of $\mathcal{Z}^{\mathrm{der}}(T,\mu)$ defines a class in $F^dK_0(\mathcal{M})_\Q$, whose image recovers the corrected class $\mathcal{C}(T,\mu)$ defined above.

Note that the derived interpretation as explained here does not (yet) shed any light on why Theorems C, D, and E from the introduction are true. 

We will not need this perspective in this paper, but these ideas are explored further in~\cite{Madapusi2022}, where it is shown that the derived stack $\mathcal{Z}^{\mathrm{der}}(T,\mu)$ admits a moduli interpretation, which leads to alternate proofs of Theorems C, D and E.
\end{remark}

The remainder of \S \ref{s:derived cycle classes} is devoted to studying the properties of $\mathcal{C}(T,\mu)$.

 
\subsection{Intersections of cycle classes}
 
 
We first explain how our corrected cycle classes behave under the intersection pairing in the Chow ring, as this is one of the few properties that follow directly from the definition.

Fix positive integers $d'$ and $d''$,  symmetric matrices 
\[
T' \in \Sym_{d'}(\Q) \quad  \mbox{and} \quad T'' \in \Sym_{d''}(\Q),
\]
and tuples $\mu' \in (L^\vee/L)^{d'}$ and  $\mu'' \in (L^\vee / L)^{d''}$.

\begin{proposition}\label{prop:intersection formula}
The  corrected  cycle classes 
\[
\mathcal{C}(T',\mu') \in \mathrm{CH}^{d'}_{ \mathcal{Z}(T',\mu')}(\mathcal{M}) 
\quad  \mbox{and}\quad 
\mathcal{C}(T'',\mu'') \in \mathrm{CH}^{d''}_{ \mathcal{Z}(T'',\mu'')}(\mathcal{M}) .
\]
of Definition \ref{def:corrected classes} satisfy the intersection formula
\[
\mathcal{C} (T',\mu') \cdot \mathcal{C} (T'',\mu'') 
= \sum_{ T = \left(\begin{smallmatrix}  T' & * \\ * & T''  \end{smallmatrix} \right)   } \mathcal{C}(T,\mu),
\]
where $\mu=(\mu',\mu'')$ is the concatenation of $\mu'$ and $\mu''$, and  the product   
\[
\mathrm{CH}^{d'} _{\mathcal{Z}(T',\mu') } ( \mathcal{M} ) \otimes \mathrm{CH}^{d''} _{\mathcal{Z}(T'',\mu'') } ( \mathcal{M} )
\to  \mathrm{CH}^{d' + d''} _{  \mathcal{Z}(T',\mu') \times_\mathcal{M} \mathcal{Z}(T'',\mu'')    }  ( \mathcal{M} )
\]
on the left is  the  intersection pairing  of \S \ref{ss:chow}.
\end{proposition}

\begin{proof}
Let $t_1',\ldots, t'_{d'}$ and $t_1'',\ldots, t''_{d''}$ be the diagonal entries of $T'$ and $T''$.  If we abbreviate
\[
\mathcal{Z}_i' = \mathcal{Z}(t_i ' , \mu_i') \quad \mbox{and}\quad  \mathcal{Z}_i'' = \mathcal{Z}(t_i '' , \mu_i''),
\]
 there is a commutative diagram
\[
\xymatrix{
{ \mathcal{Z}(T',\mu') \times \mathcal{Z}(T'',\mu'') } \ar[dr] \ar@{=}[dd]  \\
&   {    \mathcal{Z}_1'\times \cdots \times \mathcal{Z}_{d'}' \times \mathcal{Z}''_1 \times \cdots \times \mathcal{Z}''_{d''} }  \\
  {  \bigsqcup\limits_{   T    }  \mathcal{Z}(T,\mu)  }  \ar[ur]
}
\]
in which the  vertical $=$ is the canonical isomorphism of Proposition \ref{prop:naive intersection}, and both diagonal arrows are open and closed immersions.   Here and below, all fiber products are over $\mathcal{M}$.

Using the pairing \eqref{new intersection}, the constructions of \S \ref{ss:cycle construction} provide us with a class
\[
z_1'\cap \cdots \cap  z'_{d'} \cap z_1'' \cap\cdots \cap z_{d''}'' \in 
F^{d'+d''} G_0 (      \mathcal{Z}_1' \times \cdots  \times \mathcal{Z}_{d'}'  \times \mathcal{Z}''_1  \times \cdots  \times \mathcal{Z}''_{d''}  )_\Q 
\]
whose  restriction along the top diagonal arrow is 
\[
[ \co^\mathrm{derived}_{\mathcal{Z}(T',\mu')} ] \cap [  \co^\mathrm{derived}_{\mathcal{Z}(T'',\mu'') }] 
 \in G_0 (   Z(T',\mu') \times Z(T'',\mu'')  )_\Q ,
\]
and whose restriction along the bottom diagonal arrow is
\[
\sum_T [  \co^\mathrm{derived} _{\mathcal{Z}(T,\mu)}  ]  
\in  \bigoplus_T G_0( \mathcal{Z}(T,\mu) )_\Q  =  G_0 \Big( \bigsqcup_T \mathcal{Z}(T,\mu) \Big)_\Q .  
\]
In particular, the vertical $=$ in the above diagram identifies
\[
[ \co^\mathrm{derived}_{\mathcal{Z}(T',\mu')} ] \cap [  \co^\mathrm{derived}_{\mathcal{Z}(T'',\mu'') }]  
= 
\sum_T [  \co^\mathrm{derived} _{\mathcal{Z}(T,\mu)}  ]  . 
\]
The proposition follows by applying the composition  
\begin{align*}
 F^{d'+d''} G_0 \left(   Z(T',\mu') \times Z(T'',\mu'')   \right)_\Q  
& \map{ \eqref{more coniveau push} }   
F^{d'+d''} K_0^{   Z(T',\mu') \times Z(T'',\mu'')  } ( \mathcal{M} )_\Q  \\
& \map{ \eqref{ChowGroth} } \mathrm{CH}^{d'+d''}_{   Z(T',\mu') \times Z(T'',\mu'')  } ( \mathcal{M} )
\end{align*}
to both sides of this last equality.
\end{proof}


\subsection{An alternate construction}
\label{ss:alt derived}


 In this subsection we will give a different characterization of the derived fundamental classes of Definition  \ref{def:corrected classes}.
This  will be used in the proof of Proposition \ref{prop:linear invariance} below.

Consider again the situation of \S \ref{ss:cycle construction}, where we have fixed $\mu\in (L^\vee/L)^d$ and 
$t_1,\ldots, t_d \in \Q$ are the diagonal entries of $T\in \Sym_d(\Q)$.

Fix an \'etale surjection $U \to \mathcal{M}$ with $U$ scheme.  The special divisors
\[
 \mathcal{Z}(t_i,\mu_i) \to \mathcal{M}
\]
are finite and unramified, and so, as in Definition \ref{def:general cartier}, we may assume that $U$ is chosen so that the natural map
$
Z_i \to U
$
is a closed immersion of schemes for every $1\le i \le d$  and every connected component 
\[
Z_i \subset  \mathcal{Z}(t_i,\mu_i)_U .
\]

Fix a tuple $(Z_1,\ldots, Z_d)$ with each $Z_i \subset \mathcal{Z}(t_i,\mu_i)_U$ a connected component.
If $(t_i,\mu_i) \neq (0,0)$ then $Z_i \subset U$ is an effective Cartier divisor (Proposition \ref{prop:divisors are divisors}), and its ideal sheaf $I_{Z_i} \subset \co_U$  determines a chain complex of locally free $\co_U$-modules
\[
C_{Z_i}  = ( \cdots \to 0 \to I_{Z_i} \to \co_U \to 0 \to \cdots )
\]
supported in degrees $1$ and $0$.   
If $(t_i,\mu_i) = (0,0)$, so that $Z_i = U$, we instead define
\[
C_{Z_i}  = ( \cdots \to 0 \to \omega|_U \map{0} \co_U \to 0 \to \cdots ).
\]
The tensor product  
$
C_{Z_1} \otimes_{\co_U} \cdots \otimes_{\co_U} C_{Z_d}
$
 is a complex of locally free $\co_U$-modules, whose   $\ell^\mathrm{th}$  homology 
 \begin{equation}\label{homology components}
H_\ell( C_{Z_1} \otimes_{\co_U} \cdots \otimes_{\co_U} C_{Z_d} )
\end{equation}
   is a coherent sheaf on $U$ annihilated by the ideal sheaf of the closed subscheme 
$Z_1 \times_U \cdots \times_U Z_d \subset U$.
We  may therefore view this sheaf as a coherent sheaf on this closed subscheme.

By varying the tuple  $(Z_1,\ldots, Z_d)$, we obtain from \eqref{homology components} a coherent sheaf 
\begin{equation}\label{there is no complex}
H_\ell  \big( C_{\mathcal{Z}(t_1,\mu_1)_U}  \otimes_{\co_U} \cdots \otimes_{\co_U} C_{\mathcal{Z}(t_d,\mu_d)_U}  \big)
\end{equation}
(this is just notation;  no complex $C_{\mathcal{Z}(t_i,\mu_i)_U}$ of $\co_U$-modules will be defined)
on the disjoint union 
 \[
 \mathcal{Z}(t_1,\mu_1)_U \times_U  \cdots \times_U  \mathcal{Z}(t_d,\mu_d)_U
 = \bigsqcup_{ (Z_1, \ldots, Z_d) } 
 Z_1 \times_U \cdots \times_U Z_d ,
 \]
 which admits a canonical descent to a coherent sheaf
 \begin{equation}\label{fake homology complex}
 H_\ell  \big( C_{\mathcal{Z}(t_1,\mu_1)}  \otimes_{\co_\mathcal{M}} \cdots \otimes_{\co_\mathcal{M}} \co_{\mathcal{Z}(t_d,\mu_d)}  \big)
 \end{equation}
(again, no complex $C_{\mathcal{Z}(t_i,\mu_i)}$ of $\co_\mathcal{M}$-modules will be defined)  on 
 \[
  \mathcal{Z}(t_1,\mu_1) \times_\mathcal{M}  \cdots \times_\mathcal{M}  \mathcal{Z}(t_d,\mu_d).
 \]
 This last sheaf may be restricted to the open and closed substack \eqref{open and closed}.

 \begin{proposition}\label{prop:alt derived}
 The derived fundamental class of Definition \ref{def:corrected classes} is equal to 
 \[
 [ \co^{\mathrm{derived}}_{ \mathcal{Z}(T,\mu)} ]   = 
  \sum_{\ell\ge 0} (-1)^\ell  \cdot
    \big[  H_\ell  \big( C_{\mathcal{Z}(t_1,\mu_1)}  \otimes_{\co_\mathcal{M}} \cdots \otimes_{C_\mathcal{M}} \co_{\mathcal{Z}(t_d,\mu_d)}  \big)|_{  \mathcal{Z}(T,\mu)  }   \big] .
 \]
 \end{proposition}
 
 \begin{proof}
An elementary but tedious exercise in homological algebra  shows that 
 \[
 z_1\cap \cdots \cap z_d = \sum_{\ell\ge 0} (-1)^\ell  \cdot \big[   H_\ell  ( \co_{\mathcal{Z}(t_1,\mu_1)}  \otimes_{\co_\mathcal{M}} \cdots \otimes_{\co_\mathcal{M}} \co_{\mathcal{Z}(t_d,\mu_d)}  )     \big]  
 \]
 as elements of 
 \[
 G_0 ( \mathcal{Z}(t_1,\mu_1)  \times_{\mathcal{M}} \cdots \times_{\mathcal{M}} \mathcal{Z}(t_d,\mu_d) )_\Q,
 \]
 where each $z_i \in G_0 ( \mathcal{Z}(t_i,\mu_i) ) $ is as in \S \ref{ss:cycle construction}, and the intersection on the left is obtained by iterating the pairing of Lemma \ref{lem:new intersection}.  The claim follows immediately from this.

We point out that verifying the equality above does not require unpacking the use of derived algebraic geometry or the sheaf of spectra $\mathbf{G}_{Z_1 \times_M Z_2}$ in the proof of Lemma \ref{lem:new intersection}.  
One need only verify the same equality in the naive Grothendieck group $G_0^\mathrm{naive}$ of Remark \ref{rem:naive G}, with the $\cap$ on the left defined by \eqref{naive intersection}, and use the commutativity of the diagram in Lemma \ref{lem:new intersection}.
 \end{proof}
 
 \begin{remark}
 The slightly complicated  constructions above are done solely to account for the failure of the special divisors $\mathcal{Z}(t_i,\mu_i) \to \mathcal{M}$ to be closed immersions. 
  If they were closed immersions, we would simply have defined complexes of locally free $\co_\mathcal{M}$-modules 
 \[
C_{\mathcal{Z}(t_i,\mu_i) }  =
\begin{cases}
 ( \cdots \to 0 \to I_{\mathcal{Z}(t_i,\mu_i)} \to \co_\mathcal{M} \to 0 \to \cdots )  & \mbox{if } (t_i,\mu_i) \neq (0,0) \\ \\
( \cdots \to 0 \to \omega  \map{0} \co_\mathcal{M} \to 0 \to \cdots ) & \mbox{if } (t_i,\mu_i) = (0,0) .
\end{cases} 
\]
The  sheaf \eqref{fake homology complex} would then be understood in the literal sense,  as  the $\ell^\mathrm{th}$ homology of the tensor product of complexes.
 \end{remark}

 
\subsection{Linear invariance}


Suppose $X$ is any abelian group.  Given a $d$-tuple $x \in X^d$ and an $A \in \GL_d(\Z)$, we define $xA\in X^d$ using the habitual rule for multiplication of a row vector by a matrix.

Fix a matrix $T\in \Sym_d(\Q)$ and a tuple $\mu \in (L^\vee/L)^d$.
Fix also a matrix $A \in \GL_d(\Z)$, and set
\[
( T' ,\mu' )   =  ( {}^t A T A  , \mu A ) .
\]
By  Proposition \ref{prop:naive linear invariance} there is  an  isomorphism  of $\mathcal{M}$-stacks
\begin{equation}\label{naive linear equivalence}
\mathcal{Z}(T,\mu) \iso \mathcal{Z}(T',\mu'),
\end{equation}
sending an $S$-point $x\in V(\mathcal{A}_S)_\Q^d$ of the left hand side to the $S$-point $xA \in V(\mathcal{A}_S)^d_\Q$ of the right hand side.   
Hence the special cycles in \eqref{naive linear equivalence} have the same images in $\mathcal{M}$, and  there is an equality  
\begin{equation}\label{chow invariance}
\mathrm{CH}^d_{ \mathcal{Z} (T,\mu) }( \mathcal{M} ) = \mathrm{CH}^d_{ \mathcal{Z} (T',\mu') }( \mathcal{M} )
\end{equation}
of Chow groups with support.

\begin{proposition}\label{prop:linear invariance}
The isomorphism \eqref{naive linear equivalence} identifies the derived fundamental classes 
\[
[ \co^{\mathrm{derived}}_{ \mathcal{Z}(T,\mu)} ]  \in G_0(\mathcal{Z}(T,\mu) )
\quad \mbox{and}\quad  
[ \co^{\mathrm{derived}}_{ \mathcal{Z}(T',\mu')} ]  \in G_0(\mathcal{Z}(T',\mu') ).
\]
In particular, the equality $\mathcal{C}(T,\mu) = \mathcal{C}(T',\mu')$ holds in \eqref{chow invariance}.
\end{proposition}

\begin{proof}
Using the alternate construction of Proposition \ref{prop:alt derived}, we are reduced to proving the existence,  for every $\ell \ge 0$,   of an isomorphism
\begin{align}\label{higher tor equivalence}\lefteqn{ 
 H_\ell  \big( C_{\mathcal{Z}(t_1,\mu_1)}  \otimes_{\co_\mathcal{M}} \cdots \otimes_{C_\mathcal{M}} \co_{\mathcal{Z}(t_d,\mu_d)}  \big)|_{  \mathcal{Z}(T,\mu) }  } \\
&  \iso 
  H_\ell  \big( C_{\mathcal{Z}(t'_1,\mu'_1)}  \otimes_{\co_\mathcal{M}} \cdots \otimes_{C_\mathcal{M}} \co_{\mathcal{Z}(t'_d,\mu'_d)}  \big)|_{  \mathcal{Z}(T',\mu') } \nonumber
\end{align}
of coherent sheaves on \eqref{naive linear equivalence}.  Here $t_1,\ldots, t_d$ and $t_1',\ldots, t_d'$ are the diagonal entries of $T$ and $T'$.

Moreover, it  suffices to treat the case in which $d\ge 2$ and 
\begin{equation}\label{Asimple}
A  = \left(  \begin{matrix}    1 & 0    \\  1 & 1 \\ & & I_{d-2}  \end{matrix}   \right).
\end{equation}
Indeed,  the group $\GL_d(\Z)$ is generated by $A$, the permutation matrices, and the diagonal matrices, and the claim is easily proved in the latter two cases.

As in the constructions of  \S \ref{ss:alt derived}, choose   an \'etale  surjection $U\to \mathcal{M}$ fine enough 
that the morphisms
\[
\mathcal{Z}(t_i,\mu_i)_U \to U \quad \mbox{and} \quad  \mathcal{Z}(t'_i,\mu'_i)_U \to U,
\]
for all $1\le i \le d$, restrict to closed immersions on all connected components
\[
Z_i \subset \mathcal{Z}(t_i,\mu_i)_U \quad \mbox{and} \quad   Z'_i \subset \mathcal{Z}(t'_i,\mu'_i)_U .
\]

To simplify notation, we abbreviate $\mathcal{Z}$ for the $\mathcal{M}$-stack \eqref{naive linear equivalence}.
Fix a geometric point $z$ of $\mathcal{Z}$.  The finite \'etale $z$-scheme $z_U$ defined by the cartesian diagram
\[
 \xymatrix{  
 { z_U } \ar[r] \ar[d]   &  { \mathcal{Z}_U }  \ar[d]   \\
 {  z  }  \ar[r]    &  {  \mathcal{Z}  }   
 }
\]
decomposes as a finite disjoint union of points  $z_U = \bigsqcup y$.

Fix one connected component $y \subset z_U$. Its image in $\mathcal{Z}_U$ lands on some connected component $Z \subset \mathcal{Z}_U$, whose images under the  two maps
\[
\xymatrix{
{  \mathcal{Z}(T,\mu)_U  } \ar@{=}[r] \ar[d]  &  {  \mathcal{Z}_U }   \ar@{=}[r]   &  \mathcal{Z}(T',\mu')_U \ar[d] \\
 { \mathcal{Z}(t_i,\mu_i)_U }   & & { \mathcal{Z}(t'_i,\mu'_i)_U  } 
 } 
\]
are then contained in unique connected components
\[
Z_i \subset \mathcal{Z}(t_i,\mu_i)_U
 \quad \mbox{and}\quad 
 Z'_i \subset \mathcal{Z}(t'_i,\mu'_i)_U .
\]
The natural map $Z \to U$ is a closed immersion, realizing $Z$ as a connected component of both intersections
\[
Z_1\times_U \cdots \times_U Z_d \quad \mbox{and}\quad  Z'_1\times_U \cdots \times_U Z'_d .
\]
The construction \eqref{homology components} gives us coherent sheaves 
\[
H_\ell( C_{Z_1} \otimes_{\co_U} \cdots \otimes_{\co_U} C_{Z_d} )
 \quad \mbox{and}\quad 
H_\ell( C_{Z'_1} \otimes_{\co_U} \cdots \otimes_{\co_U} C_{Z'_d} )
\]
on these two intersections, and we will construct a Zariski open neighborhood $y \in V_y \subset Z$ together with a canonical isomorphism
\begin{equation}\label{key local invariance}
H_\ell( C_{Z_1} \otimes_{\co_U} \cdots \otimes_{\co_U} C_{Z_d} )|_{V_y}
 \iso 
 H_\ell( C_{Z'_1} \otimes_{\co_U} \cdots \otimes_{\co_U} C_{Z'_d} )|_{V_y} 
\end{equation}
of coherent sheaves on $V_y$.

Before we construct \eqref{key local invariance}, we explain how it implies the existence of the desired isomorphism \eqref{higher tor equivalence}, and hence completes the proof of Proposition \ref{prop:linear invariance}.
Recalling from \eqref{open and closed} that $\mathcal{Z}_U$ is an open and closed subscheme of both
\[
\mathcal{Z}(t_1,\mu_1)_U \times_U \cdots \times_U \mathcal{Z}(t_d,\mu_d)_U
\]
and
\[
\mathcal{Z}(t'_1,\mu'_1)_U \times_U \cdots \times_U \mathcal{Z}(t'_d,\mu'_d)_U , 
\]
varying the connected component $y \in \pi_0(z_U)$  in \eqref{key local invariance} defines an isomorphism
\begin{align*}\lefteqn{
H_\ell  \big( C_{\mathcal{Z}(t_1,\mu_1)_U}  \otimes_{\co_U} \cdots \otimes_{\co_U} C_{\mathcal{Z}(t_d,\mu_d)_U}  \big)|_{V_z}  } \\
& \iso 
H_\ell  \big( C_{\mathcal{Z}(t'_1,\mu'_1)_U}  \otimes_{\co_U} \cdots \otimes_{\co_U} C_{\mathcal{Z}(t'_d,\mu'_d)_U}  \big)|_{V_z}
\end{align*}
between the sheaves of \eqref{there is no complex}, after restriction to the Zariski open neighborhood 
\[
V_z \define \bigsqcup_{ y\in \pi_0(z_U)}  V_y \subset \mathcal{Z}_U
\]
 of the image of $z_U \to \mathcal{Z}_U$.  Varying the geometric point $z$ and gluing over the resulting Zariski open cover $\{ V_z\}_z$ of $\mathcal{Z}_U$ defines an isomorphism
 \begin{align*}\lefteqn{
H_\ell  \big( C_{\mathcal{Z}(t_1,\mu_1)_U}  \otimes_{\co_U} \cdots \otimes_{\co_U} C_{\mathcal{Z}(t_d,\mu_d)_U}  \big)|_{\mathcal{Z}_U}   } \\
& \iso 
H_\ell  \big( C_{\mathcal{Z}(t'_1,\mu'_1)_U}  \otimes_{\co_U} \cdots \otimes_{\co_U} C_{\mathcal{Z}(t'_d,\mu'_d)_U}  \big)|_{\mathcal{Z}_U},
\end{align*}
and finally \'etale descent via $\mathcal{Z}_U \to \mathcal{Z}$ defines the desired isomorphism \eqref{higher tor equivalence}.

We now turn to the construction of \eqref{key local invariance}.
Consider the   first order infinitesimal neighborhood 
\[
Z \subset \widetilde{Z} \subset U
\]
of the closed subscheme $Z\subset U$.
In other words, if $I_Z\subset \co_U$ is the ideal sheaf defining $Z$, then $\widetilde{Z}$ is defined by the ideal sheaf $I_Z^2$.   Similarly, denote by 
\[
Z_i \subset \widetilde{Z}_i \subset U,\qquad Z'_i \subset \widetilde{Z}'_i \subset U
\]
the first order infinitesimal neighborhoods of $Z_i$ and $Z_i'$.  
Clearly $\widetilde{Z}$ is contained in both $\widetilde{Z}_i$ and $\widetilde{Z}'_i$.

The following is the analogue of~\cite[Theorem 5.1]{How19}.

\begin{lemma}
For every $1\le i \le d$ there are canonical sections 
\[
s_i \in H^0 \big( \widetilde{Z}_i , \omega|^{-1}_{\widetilde{Z}_i  } \big) \quad \mbox{and}\quad 
s'_i \in H^0 \big( \widetilde{Z}'_i ,  \omega|^{-1}_{\widetilde{Z}'_i  } \big)
\]
with scheme-theoretic zero loci $Z_i \subset \widetilde{Z}_i$ and $Z'_i \subset \widetilde{Z}'_i$, respectively.
After restriction to $\widetilde{Z}$, these sections are related by  $s'_1 = s_1 + s_2$, and $s_i'=s_i$ when $i >1$.
\end{lemma}

\begin{proof}
By virtue of the moduli problem defining $\mathcal{Z}(t_i,\mu_i)$, there is a canonical special endomorphism $x_i \in V_{\mu_i}(\mathcal{A}_{Z_i})$.
The desired section 
\[
s_i  = \mathrm{obst}_{x_i}\in H^0 \big( \widetilde{Z}_i , \omega|^{-1}_{\widetilde{Z}_i  } \big)
\]
is the obstruction to deforming $x_i$, as  in the proof of Proposition~\ref{prop:divisors are divisors} 
(if $x_i=0$ we understand   $\mathrm{obst}_{x_i}=0$, because there is no obstruction to deforming the $0$ endomorphism).
The section $s_i'$ is defined similarly.

Because of the particular choice of matrix \eqref{Asimple}, after restriction to $Z$ the special quasi-endomorphisms 
$x_i$ and $x_i'$ are related by $x_1' = x_1+x_2$,  and $x_i'=x_i$ if $i>1$. This leads to similar relations between  $s_i$ and $s_i'$.
 \end{proof}

\begin{lemma}\label{lem:obstruction lifts}
Around every point of $Z$ one can find a Zariski  open affine  neighborhood $V \subset U$ over which $\omega|_V \iso \co_V$, and sections
\[
\sigma_1,  \sigma_2   \in H^0( V, \omega|_V^{-1}) \quad\mbox{and}\quad 
 \alpha  \in H^0(V,\co_V) 
\]
such that
\begin{enumerate}
\item[(i)]
$\sigma_1$ has zero locus $Z_1 \cap V$ and   agrees with $s_1$ on $\widetilde{Z}_1 \cap V $,
\item[(ii)]
$\sigma_2$ has zero locus $Z_2 \cap V$ and   agrees with $s_2 $ on $\widetilde{Z}_2\cap V $,
\item[(iii)]
$\alpha$ restricts to the constant function $1$ on $Z_2\cap V$,
\item[(iv)]
the section  \[\sigma'_1 \define \sigma_1 + \alpha \sigma_2\] has zero locus $Z'_1\cap V$
and agrees with $s'_1$ on the closed formal subscheme, lying between $Z'_1\cap V$ and $\widetilde{Z}'_1 \cap V$, defined by the ideal sheaf 
\[
I_{Z'_1\cap V} \cdot \big(  I_{ Z'_1\cap V } +   I_{ Z_2\cap V} \big)    \subset \co_V.
\]
\end{enumerate}
\end{lemma}
\begin{proof}
The proof is identical to that of ~\cite[Lemma 5.2]{How19}, and makes crucial use of the fact that $\mathcal{M}$ is regular and Noetherian.
\end{proof}

Choose a Zariski open $V\subset U$ as in Lemma \ref{lem:obstruction lifts} containing the image of the geoemtric point $y \to Z \subset U$.
The sections of Lemma \ref{lem:obstruction lifts} determine chain complexes of locally free $\co_V$-modules
\begin{align*}
D_{Z_1} & = (   \cdots \to 0 \to \omega|_V \map{ \sigma_1 } \co_V \to 0  )\\
D_{Z_2} & = (   \cdots \to 0 \to \omega|_V \map{ \sigma_2 } \co_V\to 0  ) \\
D_{Z'_1} & = (   \cdots \to 0 \to \omega|_V \map{ \sigma'_1 } \co_V\to 0 ),
\end{align*}
and there are canonical isomorphisms
\[
 D_{Z_1}  \iso C_{Z_1} |_V ,\quad 
  D_{Z_2}  \iso C_{Z_2} |_V ,\quad 
 D_{Z'_1}   \iso C _{Z'_1} |_V.
\]
Indeed, if  $t_1\neq 0$, so that $Z_1 \subset U$ is a Cartier divisor and $\sigma_1\neq 0$, the first isomorphism is
\begin{equation*}
\xymatrix{
{\cdots}\ar[r]  & 0 \ar[r] \ar@{=}[d] & {\omega|_V} \ar[r]^{ \sigma_1}  \ar[d]^{ \sigma_1} & { \co_V } \ar[r] \ar@{=}[d]& 0  \\
{\cdots}\ar[r]   & 0 \ar[r] & { I_{Z_1 \cap V } }  \ar[r] & { \co_V } \ar[r]& 0
}
\end{equation*}
If  $t_1=0$, so that $Z_1=U$ and $\sigma_1=0$,  then  $D_{Z_1}=C_{Z_1} |_V$ by definition. 
The other isomorphisms are entirely similar.

Using the relation $\sigma'_1 = \sigma_1 + \alpha \sigma_2$,  we see that the diagram
\[
\xymatrix{
{ \cdots }\ar[r]    &  {  0  } \ar[r] \ar@{=}[d] &    {   \omega|_V\otimes\omega|_V } \ar[rr]^{\partial_2  }  \ar@{=}[d] &  &   { \omega|_V \oplus \omega|_V }  \ar[rr]^{\partial_1}  \ar[d]^{ g_1  } &  & {\co_V }\ar[r]\ar@{=}[d] & 0 \\
{ \cdots }\ar[r]    &  {  0  } \ar[r]  &    {   \omega|_V\otimes\omega|_V } \ar[rr]^{\partial^*_2 }  &  &   { \omega|_V \oplus \omega|_V }  \ar[rr]^{\partial^*_1}  &  & {\co_V}\ar[r] & 0
}
\]
determined  by $g_1( \eta_1,\eta_2) =  (   \eta_1 ,  \eta_2 -  \alpha   \eta_1  )$  and
\begin{align*}
 \partial_1  (\eta_1,\eta_2)  & =  \sigma_1(\eta_1) + \sigma_2 (\eta_2) \\
 \partial^*_1  (\eta_1,\eta_2)  & =  \sigma'_1(\eta_1) + \sigma_2 (\eta_2) \\
 \partial_2 (    \eta_1 \otimes \eta_2   ) &=  \big(   \sigma_2 (-\eta_2) \eta_1  , \sigma_1 (\eta_1 )\eta_2 \big) \\
 \partial^*_2 (    \eta_1 \otimes \eta_2   ) &=  \big(   \sigma_2 (-\eta_2) \eta_1  , \sigma'_1 (\eta_1 )\eta_2 \big)
 \end{align*}
 commutes, and defines the middle  isomorphism in
\[
C_{Z_1} |_V  \otimes C_{Z_2} |_V  \iso
D_{Z_1}  \otimes D_{Z_2}  \iso  D_{Z'_1}  \otimes D_{Z_2}  
 \iso  C_{Z'_1} |_V  \otimes C_{Z_2}  |_V.
\]

As our choice of \eqref{Asimple} implies $Z_i=Z_i'$ and $C_{Z_i} = C_{Z_i'}$ for all $i> 1$, we obtain an isomorphism 
\begin{equation*}
( C_{Z_1}   \otimes\cdots \otimes  C_{Z_d} )  |_V  \iso  ( C_{Z'_1}   \otimes \cdots \otimes  C_{Z'_d} ) |_V ,
\end{equation*}
of complexes of locally free $\co_V$-modules. 
 This isomorphism depends on the choices of sections in Lemma \ref{lem:obstruction lifts}, which are not unique.
 However, exactly as in  the proof of \cite[Lemma 5.2]{How19},  the conditions of that lemma imply that different choices  yield homotopic isomorphisms, and so the induced isomorphism 
 \[
 H_\ell ( C_{Z_1}   \otimes\cdots \otimes  C_{Z_d})  |_V \iso  H_\ell ( C_{Z'_1}   \otimes\cdots \otimes  C_{Z_d'} )  |_V
 \]
 is independent of the choices. 
 
  In this last  isomorphism both sides are coherent sheaves on $V$ annihilated by the ideal sheaf of the closed subscheme 
 \[
 V_y \define V\cap Z,
 \]
  yielding the desired isomorphism \eqref{key local invariance}.
\end{proof}

 
\subsection{Comparison with the naive cycle}
\label{ss:naive comparison}
 

The following result shows  that the  corrected cycle class $\mathcal{C}(T,\mu)$ agrees with the class obtained by imitating the construction of \eqref{intro generic corrected}, whenever that construction makes sense.
We remark that the proof uses the linear invariance property of Proposition  \ref{prop:linear invariance} in an essential way.  

\begin{proposition}\label{prop:naive comparison}
Fix    $T\in \Sym_d(\Q)$  and $\mu \in (L^\vee/L)^d.$
If the naive special cycle $\mathcal{Z}(T,\mu)$ is equidimensional with
 \[
 \mathrm{dim}(  \mathcal{Z}(T,\mu)  ) =   \mathrm{dim}(\mathcal{M})-\mathrm{rank}(T),
 \]
so that  the naive cycle class
\[
[\mathcal{Z}(T,\mu)]\in \mathrm{CH}^{\mathrm{rank}(T)}_{\mathcal{Z}(T,\mu)}(\mathcal{M})
\]
is defined (Definition \ref{defn:naive class}),  then 
\[
\mathcal{C}(T,\mu)  =  \underbrace{c_1(\omega^{-1})\cdots c_1(\omega^{-1})}_{ d - \mathrm{rank}(T)}
\cdot  [\mathcal{Z}(T,\mu)]\in  \mathrm{CH}_{ \mathcal{Z}(T,\mu) }^d(\mathcal{M}) .
\]
\end{proposition}

\begin{proof}
We may assume that $T$ is positive semi-definite, for otherwise the Chow group with support $ \mathrm{CH}_{ \mathcal{Z}(T,\mu) }^d(\mathcal{M})$ is trivial by Remark \ref{rem:pos def T}.

First suppose  $\mathrm{rank}(T)=d$.  In particular,  $T$ is positive definite,  so has all  diagonal entries nonzero.
Recalling the open and closed immersion \eqref{open and closed}, consider a  closed geometric point
\[
s\to \mathcal{Z}(T,\mu) \subset \mathcal{Z}(t_1,\mu_1) \times_{\mathcal{M}} \cdots \times_{\mathcal{M}} \mathcal{Z}(t_d,\mu_d) .
\]
For every $1\le i \le d$,  Proposition \ref{prop:divisors are divisors} implies that the natural map 
\[
\co_{\mathcal{M},s} ^\et  \to \co^\et _{ \mathcal{Z}(t_i,\mu_i)  , s} 
\]
on \'etale local rings is surjective with kernel generated by a single element $f_i$.
As 
\[
\co_{ \mathcal{Z}(T,\mu) ,s}^\et \iso \co_{\mathcal{M},s} ^\et  / ( f_1,\ldots, f_d) ,
\]
our assumptions imply that   $f_1,\ldots, f_d \in \co_{\mathcal{M},s} ^\et $  is  a regular sequence.

For every  $1\le e\le d$  the \'etale local ring at $s$ of 
\[
\mathcal{Y}_e = \mathcal{Z}(t_1,\mu_1) \times_{\mathcal{M}}\cdots\times_{\mathcal{M}} \mathcal{Z}(t_e,\mu_e)
\]
 is therefore Cohen-Macaulay of dimension   $\mathrm{dim}(\mathcal{M}) - e$, and a result of Serre   \cite[Section V.B.6]{serre}  implies that 
 \[
\mathrm{Tor}^{\co^\et_{ \mathcal{M},s}}_\ell \big(  \co_{ \mathcal{Y}_e,s }^\et , \co_{  \mathcal{Z}(t_{e+1},\mu_{e+1})  ,s }^\et \big) =0
\]
for all $\ell >0$.    Using \eqref{naive intersection} and the commutative diagram of Lemma \ref{lem:new intersection}, one sees by induction on $e$ that the  intersection 
\begin{equation*}
[  \co_{\mathcal{Z}(t_1,\mu_1)}  ]\cap \cdots  \cap [  \co_{\mathcal{Z}(t_e,\mu_e)}  ] \in F^e  G_0 ( \mathcal{Y}_e )_\Q  ,
\end{equation*}
 has the form 
\[
[  \co_{\mathcal{Z}(t_1,\mu_1)}  ]\cap \cdots  \cap [  \co_{\mathcal{Z}(t_e,\mu_e)}  ] = [ \co_{  \mathcal{Y}_e   } ] + [ \mathcal{F}_e ] -  [\mathcal{G}_e]
\]
for coherent sheaves $\mathcal{F}_e$ and $\mathcal{G}_e$ on  $\mathcal{Y}_e$ with trivial stalks at any closed geometric point  
$
s \to \mathcal{Z}(T,\mu) \to  \mathcal{Y}_e .
$

Taking $d=e$ shows that 
\[
 [ \co^\mathrm{derived}_{ \mathcal{Z}(T,\mu) }] 
 =  [ \co_{ \mathcal{Z}(T,\mu) }]   \in F^d G_0( \mathcal{Z}(T,\mu)),
\]
as both are equal to the image of the class
 \[
 [  \co_{\mathcal{Z}(t_1,\mu_1)}  ]\cap \cdots  \cap [  \co_{\mathcal{Z}(t_d,\mu_d)}  ] 
 \in 
  F^d  G_0 \big( \mathcal{Z}(t_1,\mu_1)  \times_{\mathcal{M}} \cdots \times_{\mathcal{M}} \mathcal{Z}(t_d,\mu_d) \big) _\Q
\]
under the second arrow in \eqref{multilinear classes}.  The equality of cycle classes
\[
\mathcal{C}(T,\mu) = [ \mathcal{Z}(T,\mu)]
\]
now follows from Theorem~\ref{thm:GS}.

Now consider the other extreme, in which    $T=0_d$ has rank $0$.  In this case
\[
\mathcal{Z}(0_d ,\mu) = 
\begin{cases}
\mathcal{M} & \mbox{if }\mu=0 \\
\emptyset & \mbox{if }\mu \neq 0.
\end{cases}
\]
If $\mu \neq 0$ the the proposition is vacuously true, as the Chow group with support vanishes. 
On the other hand,  by construction  $\mathcal{C}(0_d ,0)$ is the image of 
\[
  \underbrace{ ( [\co_\mathcal{M}] - [\omega] ) \cap \cdots \cap  ( [\co_\mathcal{M}] - [\omega] ) }_{ d  }
  \in F^d G_0( \mathcal{M})_\Q 
\]
under 
\[
F^d G_0( \mathcal{M})_\Q \map{\eqref{coniveau push}} F^d K_0( \mathcal{M})_\Q \map{\eqref{ChowGroth}} \mathrm{CH}^d(\mathcal{M}) . 
\]
It follows from Lemma \ref{lem:pic complex} that this image is
\[
\mathcal{C}(0_d , 0 )  =  \underbrace{c_1(\omega^{-1})\cdots c_1(\omega^{-1})}_{ d  }
\in  \mathrm{CH} ^d(\mathcal{M}) ,
\]
as desired.

For the general case, let $r=\mathrm{rank}(T)$.
As the cycle classes $[ \mathcal{Z}(T,\mu)]$ and $\mathcal{C}(T,\mu)$ satisfy the same linear invariance property
(Proposition \ref{prop:naive linear invariance} and Proposition  \ref{prop:linear invariance}), we may reduce to the case  in which 
\[
T = \left( \begin{matrix} T' & 0 \\ 0 & 0_{d-r} \end{matrix}\right)
\]
for a positive definite $ r \times r$-matrix $T'$.
We may further assume that 
 \[
 \mu_{r+1}=\cdots=\mu_d=0.
 \]
 Indeed,  if  some $\mu_i\neq 0$ with $r <i\le d$ then $\mathcal{Z}(0,\mu_i)=\emptyset$ by Proposition \ref{prop:divisors are divisors}, and so    $\mathcal{Z}(T,\mu)=\emptyset$ by \eqref{open and closed}.

Set $\mu'=(\mu_1,\ldots, \mu_r)$.  Directly from the moduli interpretation we see
\[
\mathcal{Z}(T,\mu) \iso \mathcal{Z}(T',\mu')
\]
as $\mathcal{M}$-stacks.  Combining this with  the positive definite  and  rank $0$ cases already proved yields the first equality in
\begin{align*}
  [\mathcal{Z}(T,\mu)]  \cdot 
\underbrace{c_1(\omega^{-1})\cdots c_1(\omega^{-1})}_{ d - r} 
&=
\mathcal{C}(T',\mu') \cdot \mathcal{C}\big( 0_{ d-r } , 0 \big)  \\
& = 
\sum_S  \mathcal{C}(S,\mu).
\end{align*}
The second equality is by   the intersection formula of Proposition \ref{prop:intersection formula}, and the sum runs over all matrices of the form
\[
S = \begin{pmatrix}  T' & * \\  * &  0_{ d-r}   \end{pmatrix}  \in \Sym_d(\Q).
\]
The only nonzero terms come from positive semi-definite $S$, and the only such $S$ is $S=T$.
This completes the proof of Proposition \ref{prop:naive comparison}.
\end{proof}

\begin{corollary}
For any $T\in \Sym_d(\Q)$  and $\mu \in (L^\vee/L)^d$, restriction to the generic fiber
\[
\mathrm{CH}^d_{\mathcal{Z}(T,\mu) } ( \mathcal{M}) \to \mathrm{CH}^d_{Z(T,\mu) } (M)
\]
sends the corrected cycle class $\mathcal{C}(T,\mu)$ to the class $C(T,\mu)$ of \eqref{intro generic corrected}.
\end{corollary}

\begin{proof}
For a fixed pair $(T,\mu)$, it suffices to prove the claim after enlarging the finite set of primes $\Sigma$ that we have inverted on the base. 
By adding to $\Sigma$  all primes $p$ for which  $\mathcal{Z}(T,\mu)$ has an irreducible component supported in characteristic $p$, we may assume that no such primes exist.

As the generic fiber $Z(T,\mu)$ is equidimensional of codimension $\mathrm{rank}(T)$ in the generic fiber $M$, for example by Proposition \ref{prop:ZTmu_uniformization}, also $\mathcal{Z}(T,\mu)$ is equidimensional of codimension $\mathrm{rank}(T)$ in  $\mathcal{M}$. 
The claim now follows from Proposition \ref{prop:naive comparison}.
\end{proof}

 
\subsection{Pullbacks of cycle classes}


We will now consider the setup of \eqref{shimura embedding}, so that we have an isometric embedding $L \to L^\beef$ of quadratic lattices inducing a morphism $M\to M^\beef$ of canonical models of Shimura varieties.
Assume our finite set of primes $\Sigma$ is chosen so that both $L_p$ and $L_p^\beef$ are maximal at all $p\not\in \Sigma$, so that the above morphism of canonical models extends to a finite morphism
\[
f : \mathcal{M} \to \mathcal{M}^\beef
\]
of integral models over $\Z[\Sigma^{-1}]$ as in Remark~\eqref{rem:int shimura embedding}.
Assume further that both integral models $\mathcal{M}$ and $\mathcal{M}^\beef$ are regular, so that the corrected cycle classes of Definition \ref{def:corrected classes} are defined for both integral models.

The results of \S \ref{ss:chow} provide us with a pullback 
\[
f^* : \mathrm{CH}^d_{Z^\beef}( \mathcal{M}^\beef ) 
 \to \mathrm{CH}^d_{Z^\beef \times_{\mathcal{M}^\beef} \mathcal{M} }  (\mathcal{M}) 
\]
for any finite morphism $\mathcal{Z}^\beef \to \mathcal{M}^\beef$.
Given a pair $(T^\beef,\mu^\beef)$ with $T^\beef \in \Sym_d(\Q)$ and $\mu^\beef \in ( L^{\beef,\vee} / L^\beef)^d$, we can form the corrected cycle class
\[
\mathcal{C}^\beef(T^\beef,\mu^\beef) \in  \mathrm{CH}^d_{  \mathcal{Z}^\beef(T^\beef,\mu^\beef)   } ( \mathcal{M}^\beef ) ,
\]
 and ask how its pullback  is related to the corrected  cycle classes on $\mathcal{M}$.

The answer to this equation is exactly what one would expect given the decomposition 
\begin{equation}\label{naive decomp}
 \mathcal{Z}^\beef (T^\beef, \mu^\beef)  \times_{ \mathcal{M}^\beef} \mathcal{M}
\iso 
\bigsqcup_{   \substack{ T \in \Sym_d(\Q)  \\  \mu\in (L^\vee / L)^d } } 
 \bigsqcup_{ \substack{   \nu \in ( \Lambda^\vee/\Lambda)^d \\ \mu+\nu =  \mu^\beef   } } 
 \bigsqcup_{   \substack{ y\in  \nu + \Lambda^d \\  T+ Q(y) = T^\beef   }    }
   \mathcal{Z} (  T  , \mu ),
\end{equation}
of Proposition \ref{prop:naive pullback}.   Recall that here $\Lambda \subset L^\beef$ is the positive definite quadratic lattice of vectors orthogonal to $L \subset L^\beef$, and the relation $\mu+\nu = \mu^\beef$ means that the natural map 
\[
(L^\vee \oplus \Lambda^\vee) / (L  \oplus \Lambda) \to (L^\vee \oplus \Lambda^\vee ) / L^\beef
\]
sends 
\[
\mu + \nu \mapsto   \mu^\beef \in   L^{\beef,\vee} / L^\beef   \subset     ( L^\vee \oplus \Lambda^\vee ) / L^\beef .
\]

\begin{proposition}\label{prop:derived pullback}
The equality of cycle classes
\[
f^* \mathcal{C}^\beef(T^\beef, \mu^\beef)  =
\sum_{   \substack{ T \in \Sym_d(\Q)  \\  \mu\in (L^\vee / L)^d } } 
 \sum_{ \substack{   \nu \in ( \Lambda^\vee/\Lambda)^d \\ \mu+\nu = \mu^\beef   } } 
 \sum_{   \substack{ y\in  \nu + \Lambda^d \\  T+ Q(y) = T^\beef   }    }
   \mathcal{C} (  T  , \mu )
\]
holds in  $\mathrm{CH}^d_{ \mathcal{Z}^\beef(T^\beef, \mu^\beef)  \times_{ \mathcal{M}^\beef} \mathcal{M}}( \mathcal{M})$.
\end{proposition}

\begin{proof}
Fix one $\mathcal{Z} (  T  , \mu )$ appearing in the right hand side of \eqref{naive decomp}, in the part of the decomposition indexed by some $\nu \in (\Lambda^\vee/\Lambda)^d$ and $y\in \nu+\Lambda^d$.
 Let $t_1,\ldots, t_d$ be the diagonal entries of $T$, let $t_1^\beef, \ldots, t_d^\beef$ be the diagonal entries of $T^\beef$, and abbreviate
\[
\mathcal{Z}_i = \mathcal{Z}(t_i,\mu_i) 
\quad  \mbox{and} \quad 
\mathcal{Z}_i^\beef= \mathcal{Z}^\beef(t^\beef_i,\mu^\beef_i)
\]
for the associated special divisors.  We must have $T$ positive semi-definite (for otherwise $\mathcal{C}(T,\mu)=0$), and hence all $t_i \ge 0$.

For every $1\le i \le d$, the codimension one case of Proposition \ref{prop:naive pullback} provides us with  a commutative diagram
 \begin{equation}\label{wee pullback diagram}
 \xymatrix{
{  \mathcal{Z}_i }  \ar[r]  \ar[d]  &  {  \mathcal{Z}_i^\beef  }  \ar[d] \\
{ \mathcal{M} } \ar[r]  & { \mathcal{M}^\beef },
}
\end{equation}
which defines an open and closed immersion
 \[
 j : \mathcal{Z}_i  \hookrightarrow  \mathcal{Z}_i^\beef \times_{\mathcal{M}^\beef } \mathcal{M}
 \]
 On moduli, this sends a special quasi-endomorphism $x_i \in V_{\mu_i}(\mathcal{A}_S)$ to 
\begin{equation}\label{pullback ortho decomp}
x_i^\beef = x_i + y_i \in V_{\mu^\beef_i}(\mathcal{A}^\beef_S). 
\end{equation}
 In particular, there is a homomorphism 
 \begin{equation}\label{divisor pullback}
G_0 ( \mathcal{Z}_i^\beef )_\Q   \map{ \cap [\co_\mathcal{M} ] }      
 G_0(    \mathcal{Z}_i^\beef \times_{\mathcal{M}^\beef} \mathcal{M} )_\Q  \map{j^*}  G_0 ( \mathcal{Z}_i)_\Q
\end{equation}
obtained by  composing the intersection pairing 
\[
G_0 ( \mathcal{Z}_i^\beef )_\Q \otimes G_0 ( \mathcal{M} )_\Q 
\map{\cap}  G_0(    \mathcal{Z}_i^\beef \times_{\mathcal{M}^\beef} \mathcal{M} )_\Q 
\]
of Lemma \ref{lem:new intersection} with  restriction along $j$.

\begin{lemma}\label{lem:adjunction}
Recall from  \S \ref{ss:cycle construction} the distinguished classes 
\[
z_i^\beef\in  G_0(\mathcal{Z}_i^\beef) \quad \mbox{and} \quad z_i\in G_0(\mathcal{Z}_i).
\]
The homomorphism \eqref{divisor pullback} sends  $z_i^\beef \mapsto  z_i$.
\end{lemma}

\begin{proof}
First suppose $(t^\beef_i,\mu^\beef_i) = (0,0)$.  As 
\[
0 = t_i^\beef = t_i+ Q(y_i),
\]
both $t_i=0$ and $y_i=0$, and the latter implies $\nu_i=0$.   Thus $\mathcal{Z}_i^\beef = \mathcal{M}^\beef$ and $\mathcal{Z}_i=\mathcal{M}$, and we have
\[
z_i^\beef = [\co_{\mathcal{M}^\beef}] - [\omega^{\beef}]
\quad \mbox{and} \quad 
z_i = [\co_{\mathcal{M}}] - [\omega] .
\]
  Using \eqref{naive intersection}  and the fact that the  tautological bundle $\omega^\beef \in \mathrm{Pic}(\mathcal{M}^\beef)$  pulls back to the tautological bundle $\omega \in \mathrm{Pic}(\mathcal{M})$, we see that \eqref{divisor pullback} sends
\[
[ \co_{\mathcal{M}^\beef}  ]  \mapsto  [  \co_\mathcal{M}  ]
\quad\mbox{and}\quad 
[ \omega^\beef  ]   \mapsto  [  \omega  ] . 
\]
The lemma follows immediately from this.

Next assume that $(t^\beef_i,\mu^\beef_i) \neq (0,0)$ and $(t_i,\mu_i) \neq (0,0)$.
Fix a geometric point $y\to \mathcal{Z}_i$, which we can also view as a point on $\mathcal{M},\mathcal{M}^\beef$ and $\mathcal{Z}_i^\beef$. 
As both 
\[
\mathcal{Z}_i \to \mathcal{M} \quad \mbox{and}\quad \mathcal{Z}^\beef_i \to \mathcal{M}^\beef
\]
are generalized Cartier divisors (Proposition \ref{prop:divisors are divisors}), we can write
\[
\co^\et_{\mathcal{Z}_i^\beef,y} \iso \co^\et_{\mathcal{M}^\beef,y}/(g)
\]
for a nonzero $g\in \co^\et_{\mathcal{M}^\beef,y}$ whose image  in $\co^\et_{\mathcal{M}, y}$ satisfies
\[
\co^\et_{\mathcal{Z}_i,y} \iso  \co^\et_{\mathcal{M},y}/(g).
\]
It follows that 
\[
\mathrm{Tor}^{\co^\et_{\mathcal{M}^\beef,y}}_\ell( \co^\et_{\mathcal{Z}_i^\beef,y}   ,   \co^\et_{\mathcal{M},y}) \iso \begin{cases}
\co^\et_{\mathcal{Z}_i,y}&\mbox{if $\ell = 0$}\\
0&\mbox{if $\ell>0$}.
\end{cases} 
\]
Allowing $y$ to vary shows that 
\[
\underline{\mathrm{Tor}}^{\co_{\mathcal{M}^\beef}}_\ell(\co_{\mathcal{Z}_i^\beef},\co_{\mathcal{M}})\vert_{\mathcal{Z}_i} \iso \begin{cases}
\co_{\mathcal{Z}_i}&\mbox{if $\ell = 0$}\\
0&\mbox{if $\ell > 0$},
\end{cases}
\]
and hence \eqref{divisor pullback} sends  $z_i^\beef = [ \co_{\mathcal{Z}_i^\beef}  ]$ to $z_i= [ \co_{\mathcal{Z}_i}  ]$, as desired.

Finally, we treat the subtle case in which  $(t^\beef_i,\mu^\beef_i) \neq (0,0)$ and $(t_i,\mu_i) = (0,0)$.  
This is the case that accounts for improper intersection between the images of  $\mathcal{M} \to \mathcal{M}^\beef$ and $\mathcal{Z}^\beef(T^\beef,\mu^\beef) \to \mathcal{M}^\beef$.
The left vertical arrow in \eqref{wee pullback diagram} is an isomorphism
\[
 \mathcal{Z}_i \iso \mathcal{M} ,
\]
and the top horizontal arrow is identified with the closed immersion
\[
i : \mathcal{M} \to \mathcal{Z}_i^\beef
\]
sending a functorial point  $S \to \mathcal{M}$ to the point $S\to \mathcal{Z}_i^\beef$ determined by the special quasi-endomorphism $y_i \in V_{\mu_i^\beef}(\mathcal{A}_S^\beef)$ of \eqref{pullback ortho decomp}.
This induces the open and closed immersion
\[
j : \mathcal{M} \iso \mathcal{Z}_i \hookrightarrow
 \mathcal{Z}_i^\beef \times_{\mathcal{M}^\beef} \mathcal{M},
\]
and the composition \eqref{divisor pullback} factors as
\begin{equation}\label{improper diagram}
\xymatrix{
 {   G_0(\mathcal{Z}_i^\beef)_\Q  }  \ar[d]_{   \cap [\co_{\mathcal{Z}_i^\beef}]     }    \ar[rr]^{  \cap [\co_\mathcal{M}]  }   &   &  {   G_0 (  \mathcal{Z}_i^\beef \times_{ \mathcal{M}^\beef } \mathcal{M})_\Q  }   \ar[dd]^{j^*}   \\
{     G_0(\mathcal{Z}_i^\beef\times_{\mathcal{M}^\beef}\mathcal{Z}_i^\beef)_\Q  }  \ar[d]_{\Delta^*}  & &    \\
{  G_0( \mathcal{Z}_i^\beef )_\Q  }  \ar[rr] & &  {   G_0 (  \mathcal{M})_\Q  }   .  } 
\end{equation}
Here $\Delta : \mathcal{Z}_i^\beef \to \mathcal{Z}_i^\beef\times_{\mathcal{M}^\beef}\mathcal{Z}_i^\beef$ is the  diagonal morphism, which is both an open and closed immersion, and the bottom horizontal arrow is the derived pullback 
\begin{equation}\label{degenerating pullback}
[ \mathcal{F} ] \mapsto \sum_{\ell \ge 0} (-1)^\ell \cdot  [ \underline{\mathrm{Tor}}_\ell^{ \co_{\mathcal{Z}^\beef_i} }  ( \mathcal{F} , i_*\co_\mathcal{M} )  ] 
\end{equation}
along the closed immersion $i$.

Abbreviating
\[
(S_0,\eta_0) = 
\left( \begin{pmatrix}
t_i^\beef&0\\
0&0
\end{pmatrix},(\mu_i^\beef,0)\right )
\quad \mbox{and} \quad
(S_0',\eta_0') = 
\left( \begin{pmatrix}
t_i^\beef&t_i^\beef\\
t_i^\beef&t_i^\beef
\end{pmatrix},(\mu_i^\beef,\mu_i^\beef)\right ) ,
\]
there are canonical isomorphisms
\[
  \mathcal{Z}^\beef (S_0,\eta_0)  \iso \mathcal{Z}_i^\beef    \iso   \mathcal{Z}^\beef (S_0',\eta_0') .
\]
Under the moduli interpretations, a special quasi-endomorphism $x$ of $\mathcal{A}^\beef$ representing  a point of the stack in the middle is sent to $(x,0)$ on the left, and $(x,x)$ on the right.
Proposition \ref{prop:naive intersection} realizes
\[
\mathcal{Z}_i^\beef    \iso   \mathcal{Z}^\beef (S_0',\eta_0')  \subset  \mathcal{Z}_i^\beef \times_{\mathcal{M}}\mathcal{Z}_i^\beef
\]
as an open and closed substack, and this agrees with the diagonal embedding denoted $\Delta$ above.

The linear invariance  proved in
Proposition \ref{prop:linear invariance} implies the equality of derived fundamental classes
\[
[  \co^\mathrm{derived}_{ \mathcal{Z}^\beef (S_0,\eta_0) }  ] = [ \co^\mathrm{derived}_{ \mathcal{Z}^\beef (S,\eta) } ] \in G_0( \mathcal{Z}_i^\beef )_\Q,
\]
Unpacking their definitions shows that the composition of the left vertical arrows in \eqref{improper diagram} sends
\[
[\co_{\mathcal{Z}^\beef_i}] 
 \mapsto 
[  \co^\mathrm{derived}_{ \mathcal{Z}^\beef (S_0',\eta_0') }  ]   =
[  \co^\mathrm{derived}_{ \mathcal{Z}^\beef (S_0,\eta_0) }  ]    =   
[ \co_{\mathcal{Z}^\beef_i} ]  -     [  \omega^\beef |_{  \co_{\mathcal{Z}^\beef_i}   } ]  ,
\]
The bottom horizontal arrow  \eqref{degenerating pullback}, which simplifies to $[\mathcal{F}] \mapsto [ i^* \mathcal{F}]$ when $\mathcal{F}$ is a vector bundle on $\mathcal{Z}_i^\beef$,  then sends 
\[
[ \co_{\mathcal{Z}^\beef_i} ]  -     [  \omega^\beef |_{  \co_{\mathcal{Z}^\beef_i}   } ] \mapsto [ \co_\mathcal{M} ] - [ \omega] .
\]
Combining these calculations with the commutativity of the diagram shows that  \eqref{divisor pullback}  sends $z_i^\beef  = [\co_{\mathcal{Z}^\beef_i}]$ to 
$z_i = [ \co_\mathcal{M} ] - [ \omega]$, as desired.
\end{proof}

Now consider the commutative diagram
\[
\xymatrix{
{   \mathcal{Z}(T,\mu)    }  \ar[r] \ar[d]  & {   \mathcal{Z}(T^\beef, \mu^\beef)     } \ar[d]  \\
{   \mathcal{Z}_1 \times_\mathcal{M} \cdots \times_\mathcal{M} \mathcal{Z}_d   }  \ar[r]  \ar[d]  &  {  \mathcal{Z}_1^\beef \times_{\mathcal{M}^\beef}  \cdots \times_{\mathcal{M}^\beef} \mathcal{Z}_d^\beef } \ar[d] \\
{  \mathcal{M} } \ar[r]^{f} & { \mathcal{M}^\beef .}
}
\]
The upper vertical arrows are open and closed immersions.
The middle horizontal arrow  identifies 
\[
 \mathcal{Z}_1 \times_\mathcal{M} \cdots \times_\mathcal{M} \mathcal{Z}_d  
   \subset
    (   \mathcal{Z}_1^\beef \times_{\mathcal{M}^\beef}  \cdots \times_{\mathcal{M}^\beef} \mathcal{Z}_d^\beef) \times_{ \mathcal{M}^\beef } \mathcal{M}
\]
as an open and closed substack, and induces a morphism
\begin{align*}
G_0(   \mathcal{Z}_1^\beef \times_{\mathcal{M}^\beef}  \cdots \times_{\mathcal{M}^\beef} \mathcal{Z}_d^\beef  ) 
\to
G_0(   \mathcal{Z}_1 \times_\mathcal{M} \cdots \times_\mathcal{M} \mathcal{Z}_d    ) 
\end{align*}
exactly as in \eqref{divisor pullback}.
 It follows from  Lemma \ref{lem:adjunction} that this morphism sends 
\[
z_1^\beef \cap \cdots \cap z_d^\beef \mapsto z_1 \cap \cdots\cap z_d,
\]
and so the commutativity of the diagram
\[
\xymatrix{
{G_0(   \mathcal{Z}_1^\beef \times_{\mathcal{M}^\beef}  \cdots \times_{\mathcal{M}^\beef} \mathcal{Z}_d^\beef  )   } \ar[r]\ar[d] &  {G_0(   \mathcal{Z}_1 \times_\mathcal{M} \cdots \times_\mathcal{M} \mathcal{Z}_d    )   }  \ar[d] \\
{  G_0(  \mathcal{Z}(T^\beef,\mu^\beef)  )   }  \ar[r] &  {G_0(  \mathcal{Z}(T,\mu)  )   }  
}
\]
 implies that the bottom horizontal arrow   sends   
 \[
[  \co^\mathrm{derived}_{  \mathcal{Z}^\beef(T^\beef,\mu^\beef) } ]  \mapsto [  \co^\mathrm{derived}_{ \mathcal{Z}(T,\mu) }] .
 \]
 
 Allowing $\mathcal{Z}(T,\mu)$ to vary over the  right hand side of \eqref{naive decomp}, we see that 
 in the commutative diagram 
 \[
 \xymatrix{
 {    G_0 \big(  \mathcal{Z}^\beef(T^\beef, \mu^\beef)   \big)_\Q   } \ar[rr]^{ \cap [\co_\mathcal{M}] }  \ar[d]_{ \eqref{coniveau push}  }  &  & {   G_0 \big(  \mathcal{Z}^\beef(T^\beef, \mu^\beef)  \times_{\mathcal{M}^\beef} \mathcal{M}  \big)_\Q } \ar[d]^{ \eqref{coniveau push} }  \\
 {  K^{   \mathcal{Z}^\beef(T^\beef, \mu^\beef)  } _0(\mathcal{M}^\beef )_\Q    } \ar[rr]^{f^*}  &  & { K^{  \mathcal{Z}^\beef(T^\beef, \mu^\beef)  \times_{\mathcal{M}^\beef} \mathcal{M}    } _0( \mathcal{M} )_\Q } ,
   }
 \]
the top horizontal arrow sends  
 \[
 [  \co^\mathrm{derived}_{  \mathcal{Z}^\beef(T^\beef,\mu^\beef) } ]      
   \mapsto 
 \sum_{   \substack{ T \in \Sym_d(\Q)  \\  \mu\in (L^\vee / L)^d } } 
 \sum_{ \substack{   \nu \in ( \Lambda^\vee/\Lambda)^d \\ \mu+\nu =  \mu^\beef   } } 
 \sum_{   \substack{ y\in  \nu + \Lambda^d \\  T+ Q(y) = T^\beef   }    }
[  \co^\mathrm{derived}_{ \mathcal{Z}(T,\mu) }]  .
   \]
 Proposition \ref{prop:derived pullback} follows immediately from this and the diagram
 \[
 \xymatrix{
 { F^d   K^{   \mathcal{Z}^\beef(T^\beef, \mu^\beef)  } _0(\mathcal{M}^\beef )_\Q    } \ar[rr]^{f^*}  \ar[d]_{ \eqref{ChowGroth} }   
 &  & { F^d K^{  \mathcal{Z}^\beef(T^\beef, \mu^\beef)  \times_{\mathcal{M}^\beef} \mathcal{M}    } _0( \mathcal{M} )_\Q }  \ar[d]^{ \eqref{ChowGroth} }   \\
 {  \mathrm{CH}^d_{   \mathcal{Z}^\beef(T^\beef, \mu^\beef)  }(\mathcal{M}^\beef )    } \ar[rr]_{f^*} 
 & &  {   \mathrm{CH}^d_{  \mathcal{Z}^\beef(T^\beef, \mu^\beef)  \times_{\mathcal{M}^\beef} \mathcal{M}    } ( \mathcal{M} ) , } 
 }
 \]
 which commutes by the very definition of the bottom horizontal arrow.
 \end{proof}


\section{Modularity in all codimensions}
\label{s:main result}


In this section we prove our main result.  
We remind the reader that $V$ is a quadratic space of signature $(n,2)$ with $n\ge 1$,   $L\subset V$ is a  lattice on which the quadratic form is $\Z$-valued,   $\Sigma$ is a finite set of primes containing all primes for which $L_p$ is not maximal (an assumption that will be strengthened below), and 
\[
\mathcal{M} \to \Spec(\Z[\Sigma^{-1}])
\]
 is the integral model of dimension $n+1$ from \S \ref{ss:integral model}.


\subsection{Siegel theta series}


Let $(\Lambda,Q)$ be a positive definite quadratic space over $\Z$, satisfying $\Lambda^\vee =\Lambda$.
The self-duality condition implies that the rank of $\Lambda$ is even, say 
\[
\mathrm{rank}(\Lambda)=2k.
\]
 For any  positive integer $d$,  the  \emph{genus $d$ Siegel theta series}
\begin{equation}\label{siegel theta}
\vartheta_{\Lambda,d} (\tau) = \sum_{  y \in \Lambda^d  } q^{Q(y)}
\end{equation}
is  Siegel modular form of genus $d$ and weight $k$.  Here $\tau \in \mathcal{H}_d \subset\Sym_d(\C)$ is the variable on the  Siegel half-space of genus $d$,  $Q(y) \in \Sym_d(\Q)$ is the moment matrix as in \eqref{inner products}, and $q^{Q(y)} = e^{2\pi i \mathrm{Tr}(\tau Q(y))}$.

\begin{theorem}[B\"ocherer \cite{Bocherer}]\label{thm:thetamania}
If  $4\mid k$  and  $k > 2d$, the space of  $\C$-valued Siegel modular forms of genus $d$ and weight $k$ is spanned by the genus $d$ Siegel  theta series \eqref{siegel theta} as $\Lambda$ varies over all self-dual positive definite  $\Z$-quadratic spaces  of rank $2k$.
\end{theorem}


\subsection{The main result}


We now extend the modularity of generating series proved in  \cite{BWR} from complex Shimura varieties to  their integral models. 

 Assume throughout that  $\Sigma$ contains
\begin{itemize}
  \item all odd primes   $p$  such that  $p^2$ divides  $[ L^\vee : L]$, and
  \item $p=2$, if  $L_2$ is not hyperspecial.
\end{itemize}
In particular  $\mathcal{M}$ is regular by Proposition \ref{prop:regularity}, allowing us  to define  the derived cycle classes 
\[
 \mathcal{C}(T,\mu)  \in \mathrm{CH}^d ( \mathcal{M} )
\]
of Definition \ref{def:corrected classes}.
By   Proposition \ref{prop:naive comparison},  these are given by the elementary formula 
\begin{equation}\label{no derived}
 \mathcal{C}(T,\mu) =
  \underbrace{c_1(\omega^{-1}) \cdots c_1(\omega^{-1})}_{d-\mathrm{rank}(T) } \cdot  [\mathcal{Z}(T,\mu)] \in \mathrm{CH}^{d} (\mathcal{M}) 
 \end{equation}
whenever the naive special cycle $\mathcal{Z}(T,\mu) \to \mathcal{M}$ is equidimensional of codimension $\mathrm{rank}(T)$.
Using the notation of \S \ref{ss:jacobi},  abbreviate
 \[
 \mathcal{C}(T) = \sum_{\mu\in (L^\vee/L)^d} \mathcal{C}(T,\mu) \otimes \phi_{\mu}^* \in 
  \mathrm{CH}^d ( \mathcal{M} ) \otimes S^*_{L,d}.
 \]

\begin{theorem}\label{thm:modularity}
For every integer $1\le d \le n+1$,  the formal generating series 
\begin{equation}\label{main generating}
 \phi(\tau) = \sum_{ T\in \Sym_d(\Q) } \mathcal{C}(T) \cdot q^T 
\end{equation}
valued in $ \mathrm{CH}^d( \mathcal{M} ) \otimes S^*_{L,d}$ converges to  a Siegel modular form of weight $1+\frac{n}{2}$ and representation 
\begin{equation}\label{main siegel rep}
\omega^*_{L,d} : \widetilde{\Gamma}_d \to \GL( S^*_{L,d}).
\end{equation}
The convergence and modularity are understood in the sense of Theorem \ref{BigThm:generic_modularity}:  they hold after applying any $\Q$-linear functional $\mathrm{CH}^d(\mathcal{M}) \to \C$.
\end{theorem}

\begin{proof}
When $d=1$ the desired modularity is \cite[Theorem B]{HMP}.
  We remark that the isomorphism $S_L \iso S_L^*$ sending  $\phi_\mu \mapsto \phi_\mu^*$ identifies the representation  $\rho_{L}$ in the statement of \emph{loc.~cit.}~ with the representation $\omega_L^*$ defined in \S \ref{ss:jacobi}.
Henceforth we assume $d\ge 2$.

It is a theorem of Bruinier and Westerholt-Raum \cite{BWR} that \eqref{main generating} is modular of the stated weight and representation if and only if two conditions are satisfied:
\begin{enumerate}
\item
For every $T\in \Sym_d(\Q)$ and $A \in \GL_d(\Z)$,  the coefficients satisfy the linear invariance relation
\[
\mathcal{C} (T) = \mathcal{C} ({}^t A T A).
\]

\item
For every $T_0\in \Sym_{d-1}(\Q)$,  the   generating series 
\[
 \sum_{    \substack{  m\in \Q   \\   \alpha\in   \Q^{d-1}     }   } 
\mathcal{C}  \left(\begin{matrix}  T_0 &  \frac{\alpha}{2} \\  \frac{ {}^t \alpha }{2}  & m   \end{matrix}\right)   
 \cdot   q^m \xi_1^{\alpha_1}\cdots \xi_{d-1}^{\alpha_{d-1}} 
\]
with coefficients in $\mathrm{CH}^d(\mathcal{M})\otimes S_{L,d}^*$  is a Jacobi form  of weight $1+\frac{n}{2}$, index $T_0$,    and representation \eqref{main siegel rep}.
  \end{enumerate}

Let $r(L)$ be the integer of Definition \ref{defn:rL}.
If we assume that 
 \[
 n \ge   3d+ 2r(L) +4 
 \]
 then the special cycles $\mathcal{Z}(T,\mu) \to \mathcal{M}$ indexed by $T\in \Sym_d(\Q)$ are equidimensional of codimension $\mathrm{rank}(T)$ by Proposition \ref{prop:low key}, and so the equality \eqref{no derived} holds.
The linear invariance of condition (1)  is   Proposition \ref{prop:linear invariance} 
The Jacobi modularity of condition (2) is (up to a change of notation) Proposition \ref{prop:jacobi generating}.
Note that we are using \eqref{no derived} to compare the cycle classes of Definition \ref{def:corrected classes}  with the cycle classes \eqref{eqn:def_corrected_class} used throughout \S \ref{s:low codimension}. 
This proves the theorem when  $ n \ge   3d+ 2r(L) +4$.

 To treat the general case,  let $\Lambda$ be any positive definite self-dual quadratic lattice over $\Z$, chosen so that 
 \[
   \mathrm{rank}(\Lambda) \ge  3d+  2r(L) +4.
 \]
The  $\Z$-quadratic space $L^\beef = L\oplus \Lambda$  has signature $(n^\beef,2)$  with
\[
n^\beef \ge  \mathrm{rank}(\Lambda) \ge 3d + 2r(L^\beef)  +4  ,
\]
where we have used  $r(L^\beef) \leq r(L)$. 
The quadratic lattice $L^\beef$ determines its own integral model $\mathcal{M}^\beef$ over $\Z[\Sigma^{-1}]$,  with its own corrected special cycles
\[
\mathcal{C}^\beef(T,\mu) \in \mathrm{CH}^d(\mathcal{M}^\beef).
\]

As in Remark \ref{rem:int shimura embedding}, there is a  finite  morphism
$
f :  \mathcal{M} \to \mathcal{M}^\beef 
$
of regular stacks over $\Z[\Sigma^{-1}]$.
The self-duality of $\Lambda$ implies  that $S_{L^\beef,d} \iso S_{L,d}$, so there is a  pullback (\S \ref{ss:chow})
\[
f^* : \mathrm{CH}^d(\mathcal{M}^\beef)  \otimes  S^*_{L^\beef,d}   \to \mathrm{CH}^d(\mathcal{M}) \otimes S^*_{L,d} .
\]

By the special case proved  above, the  generating series
\[
\phi^\beef (\tau) = \sum_{ T\in \Sym_d(\Q) } \mathcal{C}^\beef(T) \cdot q^T 
\]
valued in $\mathrm{CH}^d(\mathcal{M}^\beef) \otimes S^*_{ L^\beef,d}$  is a Siegel modular form of genus $d$,  
weight $1+\frac{n^\beef}{2}$, and representation $\omega^*_{L^\beef,d}$.
On the other hand,  Proposition \ref{prop:derived pullback} (more precisely, the special case stated in the introduction as Theorem \ref{BigThm:intro pullback}) implies the factorization of generating series 
\[
f^* \phi^\beef (\tau) = \phi (\tau)  \cdot \vartheta_{\Lambda,d} (\tau),
\]
where the second factor on the right is the Siegel theta series \eqref{siegel theta}.
It follows  that $\phi(\tau)$ is a \emph{meromorphic} Siegel modular form of weight
\[
1 + \frac{n^\beef}{2} - \frac{ \mathrm{rank}(\Lambda) }{2} =1 + \frac{n}{2} 
\]
with poles supported on the zero locus of  $\vartheta_{\Lambda,d}(\tau)$.   

It remains to verify that  $\phi(\tau)$ is holomorphic, which we do by allowing $\Lambda$ to vary.
Fix a point $\tau_0\in \mathcal{H}_d$ in the genus $d$ Siegel half-space.
As in the arguments of  \cite{Baily}, it follows from the construction of Poincar\'e series found in \cite{Cartan} that there exists a Siegel modular form of genus $d$ and some weight  $k$ that does not vanish at $\tau_0$.
Moreover, we may  choose this form in such a way that  its weight satisfies  $4\mid k$  and   $2k \ge  3d + 2r(L)+4$. 
 By Theorem \ref{thm:thetamania}, there exists a self-dual  $\Lambda$ as above with  
 \[
 \mathrm{rank}(\Lambda)=2k\ge  3d + 2r(L)+4 ,
 \]
  whose associated genus $d$ Siegel theta series  $\vartheta_{\Lambda,d}$ does not vanish at $\tau_0$.
The existence of such a $\Lambda$ implies that   $\phi(\tau)$ is holomorphic near $\tau_0$.\footnote{There is a subtlety in this proof pointed out to us by Steve Kudla. In the proof we are implicitly making use of the following fact: The ring of formal $q$-series satisfying the linear invariance property enjoyed by the series here is an integral domain, since it is the complete local ring of the \emph{normal} Baily-Borel compactification of the Siegel modular variety. See the discussion in~\cite[\S 8]{kudla21}.}
\end{proof}

\appendix

 
\section{Chow groups}
 

We need a working theory of Chow groups (always with $\Q$-coefficients) for Deligne-Mumford stacks $M$,  and in greater generality than is usually found in the literature.  
For example,  \S \ref{s:good cycles} and \S \ref{s:low codimension} make systematic use of  Chow groups of stacks that are not locally integral.

Throughout this section we fix a ring $S$ that is either a field or an excellent Dedekind domain 
(for example, $S=\Z$).
The term \emph{stack} will mean an equidimensional and separated Deligne-Mumford stack  of finite type over $S$.


\subsection{Chow groups of Deligne-Mumford stacks}
\label{ss:DMchow}


Our goal in this subsection is to define Chow groups of stacks, and show that it is covariant with respect to finite morphisms.   Most of what we need can be deduced directly from the results of ~\cite{Gillet1984-tk}.

As in~\cite[Definition 3.2]{Gillet1984-tk} a stack $M$ has an underlying topological space $|M|$,  whose points are the integral closed substacks $Z\subset M$.  Each open substack $U \subset M$ determines an open set $|U| \subset |M|$, whose points are those integral closed substacks $Z\subset M$ for which $Z \cap U \neq \emptyset$.
Any morphism of stacks $M' \to M$ induces a continuous map 
\begin{equation}\label{top map}
|M'| \to |M|
\end{equation}
 by \cite[Corollary 3.4]{Gillet1984-tk}.

We write $M^{(d)}$ for the subset of $|M|$ consisting of the those integral substacks of codimension $d$.


\begin{remark}
A field-valued point $x\in M(k) $ determines a map of topological spaces $|\Spec(k)| \to |M|$ whose image is a single point. This point,  the \emph{Zariski closure} of $x$, is an integral closed substack denoted $ \overline{\{x\}} \subset M$.  
  Taking Zariski closures  of field-valued points establishes a bijection between the topological space $|M|$ as defined above, and the more common definition in terms of equivalence classes of field-valued points  \cite[Tag 04XE]{stacks-project}.  
\end{remark}

Recall from \cite[\S 3]{Gillet1984-tk} that a  stack $\xi$ is \emph{punctual} if it is reduced and $|\xi|$ is a single point.
The ring of  global functions 
\[
k(\xi) = H^0(\xi,\co_{\xi})
\]
of such a $\xi$ is a field.  Moreover, if $U\to \xi$ is an \'etale chart, then  $U = \Spec(E)$ where $E$ is a separable $k(\xi)$-algebra.  If we write  $U\times_{\xi}U = \Spec(E') $  then $E'$ is a free $E$-module (for either of the two natural maps $E\to E'$), and we define the \emph{ramification index} of $\xi$ to be
  \[
   e(\xi) = \frac{\mathrm{rank}_E(E')}{\dim_{k(\xi)}(E)}.
  \]
  This is independent of the choice of chart.

According to \cite[Tag 0H22]{stacks-project}, one can associate to any  $Z\in |M|$  a distinguished punctual stack $\xi$  together with a map $\xi \to M$ such that the image of $|\xi| \to |M|$ is $Z$.  
This $\xi$ is known as the \emph{residual gerb} at $Z$, but we will refer to it below simply as the \emph{generic point} of $Z$.
 We usually conflate $Z\in M^{(d)}$ with its generic point, for example by writing $\xi\in M^{(d)}$ and referring to $\xi$ as a codimension $d$ point of $M$.

If $\xi$ is the generic point of an integral stack $M$, then we call $k(\xi)$ the \emph{field of rational functions of} $M$, and also denote it by $k(M)$.   If we write  $\mathrm{Et}(M)$ for the category of \'etale maps $U\to M$ with $U$ a scheme,  then 
\begin{align}\label{eqn:rat fns invlim}
k(M) = \varprojlim_{(U\to M)\in \mathrm{Et}(M)}k(U).
\end{align}

For an integer $d\geq 0$, define   the \emph{vector space of $d$-cycles} $\mathscr{Z}^d(M)$ as the free $\Q$-vector space on the set $M^{(d)}$ of codimension $d$ integral closed substacks.
In particular, each point $\xi\in M^{(d)}$ gives us a basis vector $[\xi]\in \mathscr{Z}^d(M)$. 
By~\cite[Lemma 4.3]{Gillet1984-tk}, we have
\begin{align}\label{eqn:cycles invlim}
\mathscr{Z}^d(M) = \varprojlim_{(U\to M)\in \mathrm{Et}(M)}\mathscr{Z}^d(U).
\end{align}

When $M$ is an integral scheme there is a divisor map $\mathrm{div}:k(M)\to \mathscr{Z}^1(M)$ defined by
\[
\mathrm{div}(f) = \sum_{\xi\in M^{(1)}}\mathrm{ord}_{\xi}(f)[\xi] . 
\]
Here  
\[
\mathrm{ord}_\xi(f) = \mathrm{length}(R/(g)) - \mathrm{length}(R/(h))
\]
where   $R=\co_{M,\xi}$ is the local ring of $M$ at $\xi$, and  we have written  $f=g/h$ in its field of fractions. 
Combining this construction with~\eqref{eqn:rat fns invlim} and~\eqref{eqn:cycles invlim} allows us to extend the definition of  the divisor map $\mathrm{div}:k(M)\to \mathscr{Z}^1(M)$ to any integral stack $M$.

For any finite morphism $\pi:M'\to M$ of stacks with $\dim(M') = \dim(M) - r$, there is a pushforward
\begin{align}\label{eqn:cycles pushforward}
\pi_*:\mathscr{Z}^{d-r}(M')\to \mathscr{Z}^d(M).
\end{align}
Indeed, given a codimension $d-r$ point  $\xi' \in |M'|$, its image under \eqref{top map} is a codimension $d$ point $\xi \in |M|$, and there is a canonical finite morphism of punctual stacks $\xi'\to\xi$.  
This allows us to define  
\[
\pi_*[\xi'] =   \frac{e(\xi)}{e(\xi')}  \cdot \deg\big( k(\xi')/ k(\xi)\big) \cdot  [ \xi],
\]
where the degree is the usual degree of a field extension.
  
In particular, if $\xi\in M^{(d-1)}$ is the generic point of an integral closed substack $W\subset M$ there is  a divisor map
\[
\mathrm{div}_{\xi}:k(\xi)^\times  = k(W)^\times \xrightarrow{\mathrm{div}} \mathscr{Z}^1(W) \to \mathscr{Z}^{d}(M),
\]
where the final arrow is the pushforward along the inclusion $W\hookrightarrow M$.

\begin{definition}
Setting
\[
\mathscr{R}^d(M) = \bigoplus_{\xi\in M^{(d-1)}}k(\xi)^\times_{\Q},
\] 
define the ($\Q$-coefficient)  \emph{codimension $d$ Chow group} of a stack $M$ by  
\[
\mathrm{CH}^d(M) = \mathrm{coker}\left(\mathscr{R}^d(M) \xrightarrow{\sum_\xi\mathrm{div}_\xi}\mathscr{Z}^d(M)\right).
\]
As usual, cycles in the kernel of the natural map $\mathscr{Z}^d(M)\to \mathrm{CH}^d(M)$ are said to be \emph{rationally equivalent to $0$}. 
\end{definition}

\begin{proposition}
\label{prop:chow pushforward}
Suppose  $\pi:M'\to M$ is a finite morphism of stacks with $\dim(M') = \dim(M) - r$. 
The pushforward on cycles~\eqref{eqn:cycles pushforward} descends to
\[
\pi_*:\mathrm{CH}^{d-r}(M')\to \mathrm{CH}^d(M).
\]
\end{proposition}
\begin{proof}
Given a codimension $d-r+1$ point  $\xi' \in |M'| $  with image $\xi \in |M|$, and  an $f\in k(\xi' )^\times$,  we need to check that $\pi_*\mathrm{div}_{\xi'}(f)\in \mathscr{Z}^d(M)$ is rationally equivalent to $0$. 
This follows from the fact that $k(\xi')$ is a finite field extension of $k(\xi)$ satisfying
\[
\pi_*\mathrm{div}_{\xi'}(f) = \mathrm{div}_{\xi} \big(\mathrm{Nm}_{k(\xi')/k(\xi)}(f)  \big) . \qedhere
\]
\end{proof}

We need the notion of  Chow groups with support from~\cite[I.2]{soule92}.
For any finite map $\pi:Z\to M$ as in Definition \ref{defn:naive class}, set
\[
\mathrm{CH}^d_Z(M) = \mathrm{CH}^{d-r}(\pi(Z)),
\]
where $\pi(Z)\subset M$ is the stack theoretic image of $\pi$: the  reduced closed substack characterized by the property that
for every \'etale map $U\to M$ with $U$  a scheme,  the closed subscheme $U\times_M \pi(Z) \subset U$ is equal to the image of the finite morphism $U\times_MZ\to U$.

 \begin{definition}
\label{defn:naive class}
Suppose $\pi:Z \to M$ is a finite morphism of stacks with $Z$ equidimensional of dimension $\dim(Z) = \dim(M) - r$.
Define 
\[
[Z] = \sum_{i=1}^n  m_i  \cdot \pi_*[\xi_i]\in \mathrm{CH}_Z^r(M), 
\]
where  $\xi_1,\ldots,\xi_n\in Z^{(0)}$ are the  generic points of the irreducible components of $Z$, 
and  $m_i$ is the length of the \'etale local ring $\co^\et_{Z,\xi_i}$, 
\end{definition}


\subsection{Chow groups and Grothendieck groups}
\label{ss:chow}


The Chow groups  defined above are contravariant with respect to morphisms between regular stacks, and also admit a bilinear intersection pairing.  The key to these properties are the results of~\cite[\S 8]{Gillet1987-ny},~\cite[Ch. I]{soule92} and~\cite{Gillet2009-tw}, relating Chow groups to Grothendieck groups of locally free sheaves.

For a scheme $M$, let $K_0(M)$ be the quotient of the free abelian group generated by symbols $[\mathcal{Q}_\bullet]$,  where $\mathcal{Q}_\bullet$ runs over finite complexes of vector bundles on $M$,  by the relations
\begin{itemize}
\item
$[\mathcal{Q}_\bullet] = [ \mathcal{R}_\bullet]$ whenever $\mathcal{Q}_\bullet$ and $\mathcal{R}_\bullet$ are quasi-isomorphic, 
\item
$[\mathcal{Q}_\bullet] = [ \mathcal{P}_\bullet ] + [ \mathcal{R}_\bullet ]$ whenever there is a short exact sequence
\[
0 \to \mathcal{P}_\bullet \to \mathcal{Q}_\bullet \to \mathcal{R}_\bullet \to 0.
\]
\end{itemize}
If $\pi:Z\to M$ is a finite morphism, the group $K_0^Z(M)$ is defined in exactly the same way, except that we only consider complexes that become exact after restriction to the open subscheme $M \smallsetminus \pi(Z)$.

Still assuming that $M$ is a scheme, we similarly write $G_0(M)$ for the Grothendieck group of the category of coherent $\co_M$-modules. 
Thus $G_0(M)$ is the abelian group generated by symbols $[\mathcal{Q}]$ with $\mathcal{Q}$ is a coherent sheaf on $M$, subject to the relations $[\mathcal{Q}] = [ \mathcal{P} ] + [ \mathcal{R} ]$ whenever there is a short exact sequence
\[
0 \to \mathcal{P} \to \mathcal{Q}\to \mathcal{R} \to 0.
\]
Our  $G_0(M)$ is the group denoted $K_0'(M)$ in \cite{soule92}.

\begin{remark}\label{rem:naive G}
One can naively imitate these definitions when $M$ is a stack.  For example, we define  $G_0^{\mathrm{naive}}(M)$ to be the free group generated by symbols $[\mathcal{Q}]$ where $\mathcal{Q}$ is a coherent sheaf on $M$, modulo the relations $[\mathcal{Q}] = [ \mathcal{P} ] + [ \mathcal{R} ]$ whenever there is a short exact sequence as above.
While $G_0^{\mathrm{naive}}(M)$  will be of use to us, the analogous naive extensions of $K_0^Z(M)$ and $K_0(M)$ will not.
For example, Theorem \ref{thm:GS} below  is  false if one uses these naive definitions.
\end{remark}

In light of the previous remark, we associate $\Q$-vector spaces $K_0(M)_\Q$ and $G_0(M)_\Q$  to a stack $M$  following the more sophisticated constructions of~\cite[\S 2]{Gillet2009-tw}.  
This requires the machinery of $K$-theory and $G$-theory spectra as laid out in~\cite[\S 3]{MR1106918}. 
Recall that $\mathrm{Et}(M)$ is the \'etale site of $M$, whose objects are  schemes $U$ equipped with an \'etale morphism $U \to M$.

Quillen $K$-theory defines a presheaf 
\[
\mathbf{K}_M (U\to M)\define K(U)_{\Q}
\]
on  $\mathrm{Et}(M)$  valued in spectra over the Eilenberg-MacLane spectrum $H\mathbb{Q}$ (we will call this a \emph{rational spectrum} for concision), or, more prosaically, in the derived category of bounded below chain complexes of $\Q$-vector spaces; see~\cite[\S 2.1]{Takeda2004-if} for an elementary and explicit representation as a chain complex. 
By a result of Thomason, this presheaf is in fact a sheaf, and one now defines the rational $K$-theory $K(M)_{\Q}$  to be its global sections.
The vector space $K_0(M)_\Q$ is  defined as the $0^\mathrm{th}$ homology of $K(M)_{\Q}$.

A completely analogous construction, using the rational spectrum associated with the exact category of coherent sheaves,  gives us a sheaf of spectra
\[
\mathbf{G}_M(U\to M)\define G(U)_{\Q}
\]
on  $\mathrm{Et}(M)$.
Taking global sections defines the rational $G$-theory space $G(M)_{\Q}$, and the vector space $G_0(M)_{\Q}$ is defined as its $0^\mathrm{th}$ homology.

To get $K_0$-groups with support along a finite map $Z\to M$, one now repeats the construction  using the presheaf
\[
\mathbf{K}^Z_M(U\to M)\mapsto K^{Z\times_M U}(U)_{\Q}
\]
associating to $U$ the rational spectrum associated with the exact category of bounded complexes of vector bundles on $U$ with cohomology sheaves supported on the image of $Z\times_M U \to U$. 
Taking the $0$-th homology group of the global sections of this presheaf (which is, once again, actually a sheaf) defines the vector space $K^Z_0(M)_\Q$.

Still assuming that $Z\to M$ is a finite morphism of stacks,  fix a morphism 
\[
f:M'\to M 
\]
 and set $Z' = Z\times_M M' $.
Given an \'etale morphism $(U \to M) \in \mathrm{Et}(M)$,  if we set $U' = U\times_M M'$ then pullback via  $U'\to U$  takes  bounded complexes of vector bundles on $U$,  acyclic outside the image of $Z\times_M U \to U$,  
 to bounded complexes of vector bundles on $U'$,  acyclic outside the image of $Z'\times_{M'} U' \to U'$. 
 This induces a pullback map on the corresponding sheaves of spectra, and hence a map 
\begin{align}\label{eqn:k0_pullback}
f^*: K_0^Z(M)_\Q\to K_0^{Z'}(M')_\Q. 
\end{align}

Now suppose we have  finite morphisms of stacks  $Z_1\to M$ and $Z_2\to M$.
For any  $(U\to M) \in \mathrm{Et}(M)$ the tensor product of bounded complexes of vector bundles determines a map of rational spectra
\[
K^{Z_1\times_M U}(U)_\Q\otimes K^{Z_2\times_M U}(U)_\Q \to K^{Z_1\times_MZ_2\times_M U}(U)_\Q,
\]
which in turn gives rise to a canonical pairing
\begin{equation}\label{eqn:k0_pairing}
K^{Z_1}_0(M)_\Q\otimes K^{Z_2}_0(M)_\Q \to K^{Z_1\times_MZ_2}_0(M)_\Q
\end{equation}
By construction, this pairing is compatible (in the obvious sense) with the pullback  \eqref{eqn:k0_pullback}.

The above vector spaces $K_0(M)_\Q$, $K_0^Z(M)_\Q$,  and $G_0(M)_\Q$  agree with those defined  at the beginning of this subsection when $M$ is a scheme, and the operations \eqref{eqn:k0_pullback} and \eqref{eqn:k0_pairing} are the obvious ones defined by pullbacks and tensor products of complexes of sheaves.
For general stacks $M$ these  vector spaces  do not admit  obvious descriptions in terms of coherent sheaves on $M$.

\begin{remark}
\label{rem:k_theory_punctual}
When $\xi$ is a punctual stack, the rank map $K_0(\xi)_\Q\to \Q$ is an isomorphism. 
To see this one  reduces to the case where $\xi$ admits a finite \'etale cover $\Spec (L) \to \xi$ by a field $L$, and then uses  \cite[Corollary 2.7]{Gillet2009-tw}, which shows that  $K_0(\xi)_\Q \to K(\Spec (L) )_\Q\iso \Q$ is an isomorphism. 
\end{remark}

\begin{remark}
\label{rem:coherent_sheaf_class}
Every coherent sheaf $\mathcal{F}$ on $M$ gives rise to\footnote{This can be seen for instance from the explicit chain complex from~\cite[\S 2.1]{Takeda2004-if}, where $G(X)_\Q$ is represented by a complex of $\Q$-vector spaces whose degree $0$ component is a quotient of the free abelian group on the set of coherent sheaves on $X$.} a map of sheaves of spectra $\underline{H\Q}\to \mathbf{G}_M$, which evaluates on global sections to a canonical class 
\[
[\mathcal{F}]\in G_0(M)_{\Q} .
\]
 Here $\underline{H\Q}$ is the locally constant sheaf of spectra assigning to every connected $(U\to M) \in \mathrm{Et}(M)$ the constant spectrum $H\Q$ (equivalently, the object $\Q[0]$ in the derived category of bounded below complexes of $\Q$-vector spaces).
 Recalling the notation of Remark \ref{rem:naive G},  this construction defines a homomorphism of vector spaces
 \begin{equation}\label{naive push}
G_0^{\mathrm{naive}}(M)_{\Q} \to G_0(M)_{\Q},
\end{equation}
which is  surjective  by  \cite[Lemma 2.5]{Gillet2009-tw}.
\end{remark}

\begin{remark}
The homomorphism \eqref{naive push} need  not be an isomorphism. Consider  the  punctual stack
\[
M = [\Spec(\C)/H]
\]
determined by a finite group  $H$   acting trivially on $\C$.
In this case $G_0^{\mathrm{naive}}(M)_{\Q}$ is the free $\Q$-module on the finite set of isomorphism classes of irreducible representations of $H$,  
while  $G_0(M)_{\Q} \iso \Q$ by Remark \ref{rem:k_theory_punctual}. 
The map \eqref{naive push}  sends an irreducible representation to its dimension.
\end{remark}

\begin{proposition}\label{prop:GtoK}
Assume $M$ is regular.
Any  finite morphism of stacks $\pi : Z \to M$ induces  a pushforward homomorphism
\begin{equation}\label{coniveau push}
\pi_* : G_0( Z )_\Q \to K_0^Z( M )_\Q. 
\end{equation}
It is an isomorphism if $M$ is a scheme and $\pi:Z\to M$ is a closed immersion.
\end{proposition}
\begin{proof}
In the case where $M$ is a scheme, this is \cite[I.3.1 Lemma 4]{soule92}. The pushforward homomorphism sends the class of a coherent sheaf $[\mathcal{F}] \in G_0( Z )$ to any finite resolution of $\pi_*\mathcal{F}$ by vector bundles on $M$.

In general, for every  $(M'\to M) \in \mathrm{Et}(M)$ we have a map
\[
\pi'_*:G(Z')_{\Q} \to K^{Z'}(M')_\Q
\]
of spectra arising from a functor of exact categories. The right hand side can be identified with the rational spectrum associated with the category of perfect complexes on $M'$ that are acyclic outside of $Z'$ (see the argument in~\cite[Theorem 3.21]{MR1106918}, and the left hand side is associated with the exact category of bounded complexes of coherent sheaves on $Z' = Z\times_MM'$. The functor is now induced by pushforward along $\pi':Z'\to M'$.
\end{proof}

Any finite morphisms of stacks  $Y\to Z\to M$ induce maps 
\[
K_0^Y(M)_\Q \to K_0^Z(M)_\Q. 
\]
Define the  \emph{coniveau filtration} on $K_0^{Z}(M)_\Q$  by 
\begin{equation}\label{Kconiveau}
F^d K_0^Z(M)_\Q = 
\bigcup_{  \substack{ \mathrm{closed\ substacks\ }  Y \subset Z  \\  \mathrm{codim}_M(Y) \ge d } } \mathrm{Image} \big(  K_0^Y(M)_\Q \to K_0^Z(M)_\Q\big) ,
\end{equation}
and denote by 
\[
\mathrm{Gr}^d_{\gamma}K_0^Z(M)_\Q \define F^dK_0^Z(M)_\Q/F^{d+1}K_0^Z(M)_\Q
\]
the graded pieces of the filtration.
For schemes, the following theorems of Gillet-Soul\'e are proved in~\cite{Gillet1987-ny} and~\cite{soule92}. 
The extensions to stacks are addressed in~\cite[\S 2.4]{Gillet2009-tw}.

\begin{theorem}[Gillet-Soul\'e] \label{thm:GS_K-theory}
Suppose $M$ is a regular stack.
\begin{enumerate}
\item
Given finite morphisms $Z_1\to M$ and $Z_2\to M$, the pairing~\eqref{eqn:k0_pairing} restricts to a bilinear pairing
\[
F^{d_1} K_0^{Z_1}(M)_\Q \otimes  F^{d_2} K_0^{Z_2}(M)_\Q 
\to F^{d_1+d_2} K_0^{Z_1 \times_M Z_2}(M)_\Q.
\]
\item
Given a morphism $f : M' \to M$ with  $M'$ another regular stack, and a finite morphism
$Z \to M$, the pullback ~\eqref{eqn:k0_pullback} restricts to 
\[
f^* : F^d  K_0^{Z }(M)_\Q  \to F^d  K_0^{Z  \times_{M} M'   }(M')_\Q .
\]
\end{enumerate}
\end{theorem}




\begin{theorem}[Gillet-Soul\'e]
\label{thm:GS}
Let $M$ be a regular stack. For any  finite morphism 
\[
\pi:Z \to M
\]
 with $\dim(Z) = \dim(M) - r$, there is a canonical isomorphism
\begin{equation}\label{ChowGroth}
\mathrm{CH}^d_{Z}(M)  \iso \mathrm{Gr}^d_{\gamma} K_0^{Z}(M)_\Q
\end{equation}
carrying the  class $[Z]\in \mathrm{CH}^r_Z(M)$ of Definition \ref{defn:naive class} to the image of  $[\co_Z]$ under 
\[
G_0(Z)_\Q \map{\eqref{coniveau push}}  K_0^{Z}(M)_\Q .
\]
\end{theorem}

\begin{proof}
In the case of schemes, the existence of this isomorphism is~\cite[Theorem 8.2]{Gillet1987-ny}, and this argument is generalized to stacks in~\cite[Theorem 2.8]{Gillet2009-tw}. For the convenience of the reader, we recall some key inputs into these proofs. This will also help us justify the last assertion about the relationship between the cycle class $[Z]$ and the $G$-theory class $[\co_Z]$, since this is not made completely explicit in the references cited.

The starting point is the Brown-Gersten-Quillen spectral sequence with first page 
\[
E_1^{p,q} = \bigoplus_{\xi\in M^{(p)}\cap \pi(Z)}K_{-p-q}(\xi)_\Q, 
\]
converging to the (higher) $K$-groups with support $K_{-p-q}^Z(M)_\Q$.
This converges to the coniveau filtration on $K_0^Z(M)_\Q$.  See~\cite[Theorem 6]{soule92} or~\cite[Theorem 2.8]{Gillet2009-tw} .

Next, we have Bloch's formula (due to Quillen), which shows that, on the second page, we have $E_2^{p,p} \iso \mathrm{CH}^p_Z(M)$. More precisely, we obtain the composition
\begin{align*}
\bigoplus_{\xi\in M^{(p-1)}\cap \pi(Z)} k(\xi)^\times_{\Q} \iso \bigoplus_{\xi\in M^{(p-1)}\cap \pi(Z)}K_1(\xi)_\Q &= E_1^{p-1,p}\to E_1^{p,p} \\
&=\bigoplus_{\eta\in M^{(p)}\cap \pi(Z)}K_0(\eta)_\Q \iso \mathscr{Z}^p_Z(M).
\end{align*}
The arrow in the middle is the differential in the spectral sequence, and Quillen shows that this is exactly the divisor map whose cokernel is $\mathrm{CH}^p_Z(M)$ (for the case of stacks, we also need the observation from Remark~\ref{rem:k_theory_punctual}).

Finally, the interaction between this spectral sequence and Adams operations is used to show that the spectral sequence stabilizes on the second page (see~~\cite[\S 6.4]{soule92},~\cite[Theorem 8.2]{Gillet1987-ny} and~\cite[\S 2.4]{Gillet2009-tw}), which establishes the isomorphisms
\[
\mathrm{CH}^p_Z(M) \iso E_2^{p,p}\iso \mathrm{Gr}^p_{\gamma}K_0^Z(M)_\Q.
\] 
Here, for $\xi\in M^{(p)}\cap \pi(Z)$, the associated map
\begin{equation}\label{eqn:points coniveau}
K_0(\xi)_\Q \to \mathrm{Gr}^p_{\gamma}K_0^Z(M)_\Q
\end{equation}
can be described as follows.
 Let $Y\subset M$ be the integral substack with generic point $\xi$. Then we have $F^pK_0^Y(M)_\Q = K_0^Y(M)_\Q$, as well as an exact sequence
\[
F^{p+1}K_0^Y(M)_\Q \to K_0^Y(M)_\Q\to K_0(\xi)_\Q\to 0,
\]
where the second map is just the map $K_0^Y(M) \to K_0(\xi)$ obtained by restriction to the generic point. This arises from the \emph{localization sequence} for $K$-theory spectra \cite[Theorems 6.8, 7.4, 7.6]{MR1106918}, which shows that, for every closed substack $Z\subset Y$, we have a fiber sequence of sheaves of spectra
\[
\mathbf{K}_M^Z\to \mathbf{K}_M^Y\to \mathbf{K}_{M\smallsetminus Z}^{Y\smallsetminus Z}.
\]
Taking global sections and then looking at $H_0$ gives us an exact sequence
\[
K_0^Z(M)_\Q \to K_0^Y(M)_\Q \to K_0^{Y\smallsetminus Z}(M\smallsetminus Z)_\Q\to 0
\]
To finish, we need to observe that
\[
\mathrm{colim}_{\underset{Z\neq Y}{Z\subset Y}}K_0^{Y\smallsetminus Z}(M\smallsetminus Z)_\Q \iso K_0(\xi)_\Q,
\]
which can be checked on the level of the corresponding sheaves of spectra.

This gives us an isomorphism
\[
K_0(\xi)_\Q \iso \mathrm{Gr}^p_{\gamma}K_0^Y(M)_\Q,
\]
and composing it with the natural map
\[
 \mathrm{Gr}^p_{\gamma}K_0^Y(M)_\Q\to  \mathrm{Gr}^p_{\gamma}K_0^Z(M)_\Q
\]
now yields~\eqref{eqn:points coniveau}.

It still remains to verify the assertion about the class $[Z]$.
That $[\pi_*\co_Z]\in F^rK^Z_0(M)_\Q$ is immediate from the definitions.
That its image in 
\[
\mathrm{Gr}^r_{\gamma}K^Z_0(M)_\Q \iso \mathrm{CH}^r_Z(M)
\]
is $[Z]$ comes down to the fact that  for any generic point $\zeta\in Z^{(0)}$ with image $\xi\in M^{(r)}\cap \pi(Z)$, the image of $\co_\zeta$ in $K_0(\xi)_\Q \iso \Q$ is
\[
\frac{e(\xi)}{e(\zeta)}  \cdot \deg\big( k(\zeta)/ k(\xi)\big). \qedhere
\]
\end{proof}

Combining Theorems~\ref{thm:GS_K-theory} and ~\ref{thm:GS} yields intersection pairings and pullbacks on Chow groups of regular stacks. 
If $M$ is a regular stack and $Z_1,Z_2 \to M$ are finite morphisms,  there is a canonical bilinear intersection pairing
\[
\mathrm{CH}^{d_1}_{Z_1}(M) \otimes \mathrm{CH}^{d_2}_{Z_2}(M) 
\to \mathrm{CH}^{d_1+d_2}_{Z_1 \times_M Z_2 }(M).
\]
For any morphism  $M' \to M$   between regular stacks and any finite morphism $Z \to M$, there is a pullback 
\[
\mathrm{CH}^d_{Z}(M) \to \mathrm{CH}^d_{Z'}(M') ,
\]
where   $Z' = Z \times_{M}M'$.



\subsection{Line bundles and divisor classes}


Suppose $\mathcal{L}$ is a line bundle on an integral scheme $M$. 
A \emph{rational trivialization} $s$ of $\mathcal{L}$ is an equivalence class of pairs $(U,\xi)$, where $U\subset M$ is a dense open subscheme, and $\xi:\co_U\xrightarrow{\simeq}\mathcal{L}\vert_U$ is a trivialization;   two such pairs $(U_1,\xi_1)$ and $(U_2,\xi_2)$ are equivalent if the trivializations $\xi_1,\xi_2$ agree on the intersection $U_1\cap U_2$. Write $k(\mathcal{L})^\times$ for the set of such rational trivializations. 

The \emph{divisor} of $s\in k(\mathcal{L})^\times$ is the cycle
\[
\mathrm{div}(s) \define \sum_{\xi\in M^{(1)}}\mathrm{ord}_\xi(s)[\xi]\in \mathscr{Z}^1(M),
\]
where the integer $\mathrm{ord}_\xi(s)$ is defined as follows.
Let $R = \co_{M,\xi}$, and choose an isomorphism $\mathcal{L}\vert_{\Spec(R)} \iso \co_{\Spec (R)}$. Via this isomorphism $s$ corresponds to a rational function $f/g$ in the fraction field of $R$, and 
\[
\mathrm{ord}_\xi(s) = \mathrm{length}(R/(f)) - \mathrm{length}(R/(g)).
\]

More generally, for  a line bundle $\mathcal{L}$ on an integral stack $M$ define
\[
k(\mathcal{L})^\times \define \varprojlim_{(U\to M)\in \mathrm{Et}(M)} k(\mathcal{L}\vert_U)^\times.
\]
From the case of schemes discussed above, we obtain a map
\[
\mathrm{div}:k(\mathcal{L})^\times \iso \varprojlim_{(U\to M)\in \mathrm{Et}(M)} k(\mathcal{L}\vert_U)^\times \to \varprojlim_{(U\to M)\in \mathrm{Et}(M)}\mathscr{Z}^1(U)\iso \mathscr{Z}^1(M).
\]
Note that  $k(\co_M )^\times = k(M)^\times$ is the set of nonzero elements in \eqref{eqn:rat fns invlim}.

If $M$ is any (not necessarily integral) stack, let $Z_1,\ldots,Z_r$ be  its irreducible components.
Viewing these as integral stacks, we define
\[
k(\mathcal{L})^\times = \prod_{i=1}^r k(\mathcal{L}\vert_{Z_i})^\times
\]
and
\[
\mathrm{div}(s) = \sum_i\mathrm{div}(s_i)\in \mathscr{Z}^1(M)
\]
for any $s = (s_1,\ldots,s_r) \in k(\mathcal{L})^\times$.
It is easy to see that the class of $\mathrm{div}(s)$ in $\mathrm{CH}^1(M)$ depends only on the isomorphism class of $\mathcal{L}$, and not on the particular choice of $s$.  This allows us to make the following definition.

\begin{definition}
\label{def:first chern class}
The \emph{first Chern class map}
\[
c_1:\mathrm{Pic}(M) \to \mathrm{CH}^1(M)
\]
sends   a line bundle $\mathcal{L}$ to the cycle class  $[\mathrm{div}(s)]$ for any $s\in k(\mathcal{L})^\times$.  
\end{definition}

Suppose that $D\subset M$ is an effective Cartier divisor.
In other words, $D$ a closed substack whose ideal sheaf $\mathcal{I}_D\subset \co_M$ is a line bundle.  
We have two ways of associating to $D$  a class in $\mathrm{CH}^1(M)$.
First, we can take the class $[D]$ as in Definition~\ref{defn:naive class}.
 Second, we can take the first Chern class of the line bundle
 \[
 \mathcal{L}(D) \define \mathcal{I}_D^{-1} .
 \] 
 These two constructions agree, as the canonical section $\co_M\to \mathcal{L}(D)$ determines an $s\in k(\mathcal{L})^\times$ with 
 \[
c_1(\mathcal{L}(D)) =  [\mathrm{div}(s)] = [D].
\]


\begin{lemma}\label{lem:pic complex}
If   $M$ is regular, the composition 
\[
 \mathrm{Pic}(M) \map{c_1} \mathrm{CH}^1(M) \map{ \eqref{ChowGroth}  }   \mathrm{Gr}^1_\gamma K_0(M)_\Q
\]
sends $\mathcal{L} \mapsto [\co_M]-[\mathcal{L}^{-1}]$.
By slight abuse of notation,  we are here identifying $[\co_M]$ and  $[\mathcal{L}^{-1}]$ with their images under 
\[
 G_0(M)_\Q \map{\eqref{coniveau push}} K_0(M)_\Q .
\]
 \end{lemma}
 
\begin{proof}
If we write   $\mathcal{L}^{-1} \iso \mathcal{I}_D \otimes \mathcal{I}_E^{-1}$  for  effective Cartier divisors $D, E  \subset M$ 
with $D\cap E \subset M$ of codimension $\ge 2$, then $c_1(\mathcal{L}) \in \mathrm{CH}^1(M)$ is represented by the class of the associated Weil divisor $D-E$.

Tensoring the short  exact sequence
\[
0 \to \co_M \to \mathcal{I}_E^{-1} \to   \mathcal{I}_E^{-1} / \co_M  \to 0
\]
with $I_D$ shows that 
\[
[ \mathcal{L}^{-1} ]  = [ \mathcal{I}_D ]  + [ \mathcal{I}_D \otimes ( \mathcal{I}_E^{-1} / \co_M )  ] 
\] 
holds in $G_0^\mathrm{naive}(M)$, which we rewrite as
\[
 [ \co_M]  - [ \mathcal{L}^{-1} ]  =   [  \co_M/ \mathcal{I}_D ]  -  [ \mathcal{I}_D \otimes ( \mathcal{I}_E^{-1} / \co_M )  ] .
\]
Using the assumption that  $D\cap E$ has codimension $\ge 2$, and the fact that $\mathcal{I}_E$ is locally principal, one can check that the equalities 
\[
[ \mathcal{I}_D \otimes ( \mathcal{I}_E^{-1} / \co_M )  ]  = [  \mathcal{I}_E^{-1} / \co_M ] = [ \co_M / \mathcal{I}_E] 
\]
hold  in $G_0^\mathrm{naive}(M)$, up to a linear combination of classes $[\mathcal{F}]$  with $\mathcal{F}$ the pushforward to $M$ of a coherent sheaf  on a codimension two closed substack  (contained in $E$).
Hence
\begin{equation}\label{line filter}
 [ \co_M]  - [ \mathcal{L}^{-1} ]  =   [  \co_M/ \mathcal{I}_D ]  -  [  \co_M/ \mathcal{I}_E ] = [\co_D] - [\co_E] 
\end{equation}
holds in $G_0^\mathrm{naive}(M)$ up to the same ambiguity.
Using the final claim of Theorem \ref{thm:GS},  we find that that the image of $c_1(\mathcal{L})= D-E$ under \eqref{ChowGroth} is equal to $ [ \co_M]  - [ \mathcal{L}^{-1} ]$.
\end{proof}

\begin{proposition}
\label{prop:projection formula}
As in Proposition~\ref{prop:chow pushforward},  let $\pi:M'\to M$ be a finite morphism of stacks with image of codimension $r$.   Suppose in addition that $M$ is a regular stack over $\Z$, and that for every prime $p$, there exists a quasi-projective scheme $X$ over $\Z[1/p]$ equipped with the action of a finite group $G$ such that 
\[
[X/G]\iso  M_{\Z[1/p]}.
\]
For any line bundle $\mathcal{L} \in \mathrm{Pic}(M)$ we have
\[
\pi_*c_1( \pi^*\mathcal{L} ) = c_1(\mathcal{L})\cdot [M']\in \mathrm{CH}^{r+1}(M).
\]
\end{proposition}
\begin{proof}
We claim first that for any finite subset $T\subset |M|$, there exists an $r \in \Z^+$ and a section $s\in H^0(M,\mathcal{L}^{\otimes r})$ whose vanishing locus is disjoint from $T$.  
For this,  choose a prime $p$ that does not divide the characteristics of $k(\xi)$ for any $\xi\in T$, and fix  $[X/G]\iso M_{\Z[1/p]}$ as in the statement of the proposition. 
By~\cite[Prop. 9.1.11]{Liu2002-tj} there is a section $ \sigma \in H^0(X,\mathcal{L}\vert_X)$ whose vanishing locus is disjoint from the pre-image of $T$ in $X$.   If we write $G = \{g_1,\ldots,g_r\}$, then
\[
g_1 \sigma \otimes  \cdots \otimes g_r \sigma \in H^0(X,\mathcal{L}^{\otimes r} \vert_X)^G = H^0(M_{\Z[1/p]} ,\mathcal{L}^{\otimes r})
\]
is a section  whose vanishing locus in $M_{\Z[1/p]}$ is disjoint from $T$. 
Multiplying this section by a sufficiently large power of $p$  provides us with the  desired section  $s\in H^0(M,\mathcal{L}^{\otimes r})$.

We apply the paragraph above with $T$ equal to the image under $\pi$ of the set of associated points\footnote{An \emph{associated point} of a stack $Z$ is one that is the image of an associated prime of $R$ for some \'etale map $\Spec (R) \to Z$.} of $M'$. The Cartier divisor $D$ of the resulting section $s \in H^0( M , \mathcal{L}^{\otimes r})$ then has the property that 
\[
D' \define D\times_M M'
\]
is  an effective Cartier divisor on $M'$, and $\mathcal{L}^{\otimes r} \iso \mathcal{I}_D^{-1}$.
Recalling that the Chow group $ \mathrm{CH}^{r+1}(M)$ has rational coefficients, it suffices to prove the stated equality after replacing $\mathcal{L}$ by $\mathcal{L}^{\otimes r}$.  Thus we may ease notation by assuming $r=1$.

The left hand side of the desired equality is now just the cycle class $[D']$ associated to the finite map
$D'  \to M$ by Definition \ref{defn:naive class}, 
so is represented in $\mathrm{Gr}^{r+1}_{\gamma}K_0(M)_\Q$ by the class $[\pi_*\co_{D'}]$. 

On the other hand, the right hand side is represented by 
\[
[  \co_D \otimes^{\mathbb{L}}_{\co_M} \pi_*\co_{M'}   ]  
= \sum_{i\ge 0} (-1)^i \cdot [\underline{\mathrm{Tor}}_i^{\co_M}(\co_D, \pi_*\co_{M'})].
\]
Using the  resolution
\[
0\to \mathcal{I}_D \to \co_M \to \co_D\to 0 
\]
of $\co_D$ by vector bundles on $M$, 
  the $\mathrm{Tor}$ sheaves in the sum  can be computed by taking the  homology of the complex
\[
\cdots \to 0 \to \mathcal{I}_{D} \otimes_{\co_M} \pi_*\co_{M'} \map{f} \pi_*\co_{M'} \to 0,
\]
where $f(a\otimes b) = ab$ is the multiplication map.
Our assumption that $D'$ is an effective Cartier divisor on $M'$ guarantees that $f$ is injective with image $\pi_* \mathcal{I}_{D'} \subset \pi_*\co_{M'}$.  It follows that the  $i=0$ term in the sum is $[\pi_*\co_{D'}]$, while all  terms with $i>0$ vanish.
\end{proof}


\subsection{A generalized intersection pairing}


Throughout this subsection we assume that $M$ is a regular stack. Our goal is to construct a  refinement of the intersection pairing of Theorem \ref{thm:GS_K-theory}.

Analogously to the coniveau filtration \eqref{Kconiveau} on $K^Z( M )_\Q$, for a finite morphism $Z \to M$ we define 
 \begin{equation}\label{G filtration}
F^d G_0( Z)_\Q = \bigcup_{ \substack{ Y \subset Z \\ \mathrm{codim}_M(Y) \ge d}  }  \mathrm{Image}\big( G_0(Y)_\Q \to G_0( Z)_\Q \big),
\end{equation}
where the union is over all closed substacks $Y\subset Z$  whose image  $\pi(Y) \subset M$ has codimension $\ge d$.  
This defines the  \emph{coniveau-in-$M$ filtration}  on  $G_0( Z)_\Q$.
Of course the filtration  depends on the morphism $\pi: Z \to M$, but we suppress this from the notation as it will always be clear from context.  It is clear that \eqref{coniveau push} restricts to a morphism 
\begin{equation}\label{more coniveau push}
F^d G_0( Z)_\Q  \to F^d K^Z_0( M)_\Q.
\end{equation}

Suppose  we are given finite  morphisms
\[
\xymatrix{
{ Z_1 } \ar[dr]_{\pi_1}&  & {Z_2} \ar[dl]^{\pi_2}\\
& {M}.
}
\]
The natural map   $\pi :  Z_1\times_M Z_2 \to M$ is also finite, hence affine, and so
\begin{equation}\label{specrel}
Z_1 \times_M Z_2 \iso \underline{\Spec}_{\co_M} ( \pi_*  \co_{ Z_1\times_M  Z_2 }   ).
\end{equation}
 Given coherent sheaves $\mathcal{F}_1$ and $\mathcal{F}_2$ on $Z_1$ and $Z_2$, respectively,  
 we can  form, for every $\ell\ge 0$, the coherent sheaf 
\begin{equation}\label{tor sheaf}{}
\underline{\mathrm{Tor}}_\ell^{\co_M} (\pi_{1*}\mathcal{F}_1 , \pi_{2*}\mathcal{F}_2)
\end{equation}
on $M$. 
 As the formation of Tor  is functorial in both variables, \eqref{tor sheaf} carries an action of the  $\co_M$-algebra 
\[
\pi_* \co_{Z_1\times_M Z_2} \iso \pi_{1*}\co_{Z_1} \otimes_{\co_M} \pi_{2*}\co_{Z_2} ,
\]
which  determines a lift of \eqref{tor sheaf}  to a coherent sheaf on \eqref{specrel}. 
This lift then determines a  class
\[
[  \underline{\mathrm{Tor}}_\ell^{\co_M} (\pi_{1*}\mathcal{F}_1 , \pi_{2*}\mathcal{F}_2) ] \in G^\mathrm{naive}_0( Z_1 \times_M Z_2 )_\Q
\]
in the naive $G$-theory group of Remark \ref{rem:naive G}.   In this way we obtain   a bilinear pairing
\[
G^\mathrm{naive}_0( Z_1 )_\Q  \otimes G^\mathrm{naive}_0( Z_2 )_\Q  \map{\cap} G^\mathrm{naive}_0( Z_1 \times_M Z_2 )_\Q
\]
defined by 
\begin{equation}\label{naive intersection}
[ \mathcal{F}_1 ] \cap [\mathcal{F}_2 ] 
= \sum_{\ell \ge 0}(-1)^\ell \cdot [ \underline{\mathrm{Tor}}_\ell^{\co_M} (\pi_{1*}\mathcal{F}_1 , \pi_{2*}\mathcal{F}_2) ].
\end{equation}
Note that the sum on the right hand side is  finite as $M$, being assumed regular,  has finite Tor dimension.

\begin{lemma}\label{lem:new intersection}
There is a unique bilinear pairing
 \begin{equation}\label{new intersection}
 G_0( Z_1 )_\Q  \otimes G_0( Z_2 )_\Q  \map{\cap} G_0( Z_1 \times_M Z_2 )_\Q
 \end{equation}
 making the diagram 
 \begin{equation*}
\xymatrix{
{ G^\mathrm{naive}_0( Z_1 )_\Q  \otimes G^\mathrm{naive}_0( Z_2 )_\Q } \ar[rr]^\cap \ar[d]  & &{  G^\mathrm{naive}_0( Z_1\times_M Z_2)_\Q  } \ar[d]  \\
{ G_0( Z_1 )_\Q  \otimes G_0( Z_2 )_\Q } \ar[rr]^\cap \ar[d]_{ \pi_{1*}\otimes\pi_{2*}}  & &{  G_0( Z_1\times_M Z_2)_\Q  } \ar[d]^{\pi_*}  \\
{ K_0^{Z_1}(M )_\Q \otimes K_0^{Z_2}(M) } \ar[rr]^{\otimes} & & {  K_0^{Z_1\times_M Z_2}(M)_\Q }
}
\end{equation*}
 commute, where the top vertical arrows are the surjections of Remark \ref{rem:coherent_sheaf_class}, and the bottom vertical arrows are those of Proposition \ref{prop:GtoK}.
\end{lemma}

\begin{proof}
In the case of schemes, so that $G_0=G_0^\mathrm{naive}$,  this is clear from the definitions.

For the  stack case, recall that  \eqref{naive push}  is surjective.  In particular $G_0( Z_1 )_\Q  \otimes G_0( Z_2 )_\Q$ is generated by elements of the form $[\mathcal{F}_1]\otimes [\mathcal{F}_2]$ for coherent sheaves $\mathcal{F}_i$ on $Z_i$, so  there can be at most one pairing \eqref{new intersection} making the top square of the diagram commute.

The cleanest way to the prove existence of \eqref{new intersection}  involves a little bit of derived algebraic geometry. 
Namely, the derived tensor product $\pi_{1*}\mathcal{F}_1\otimes^{\mathbb{L}}_{\co_M}\pi_{2*}\mathcal{F}_2$ gives a coherent sheaf on the derived affine scheme over $M$ with underlying structure sheaf $\pi_{1*}\co_{Z_1}\otimes^{\mathbb{L}}_{\co_M}\pi_{2*}\co_{Z_2}$. The underlying classical scheme here is just $Z_1\times_MZ_2$. Therefore, using~\cite[Corollary 3.4]{Khan2022-eq}, this actually gives a global section of the sheaf $\mathbf{G}_{Z_1\times_MZ_2}$, which can be identified explicitly with the right hand side of \eqref{naive intersection}. 
\end{proof}

\begin{remark}
It is natural to expect that \eqref{new intersection} restricts to 
\begin{equation}\label{G-filtration-pairing}
F^{d_1} G_0( Z_1 )_\Q  \otimes F^{d_2}  G_0( Z_2 )_\Q \map{?} F^{d_1+d_2} G_0( Z_1\times_M Z_2)_\Q.
\end{equation}
If $\pi_1$ and $\pi_2$ are closed immersions of schemes this is clear from   Theorem \ref{thm:GS_K-theory} and the final claim of Proposition \ref{prop:GtoK}.  
In general, even if one assume that $\pi_1$ and $\pi_2$ are finite morphisms of schemes, we are unable to provide a proof. 
If one attempts to imitate the proof of the analogous claim in Theorem \ref{thm:GS_K-theory}, one is immediately obstructed by the lack of Adams operators in this context.

To give a concrete sense of why  finite maps are more difficult to deal with than closed immersions,  let $C_1,\ldots, C_r$ be the connected components of 
$ Z_1\times_M Z_2 $.   Given a  class
\[
[\mathcal{F}_1] \otimes [\mathcal{F}_2] \in F^{d_1} G_0( Z_1 )_\Q  \otimes F^{d_2}  G_0( Z_2 )_\Q,
\]
we may  decompose
\[
[\mathcal{F}_1] \cap [\mathcal{F}_2]  =  c_1+ \cdots + c_r  \in \bigoplus_{j=1}^r G_0(C_j)_\Q = G_0(   Z_1\times_M Z_2 )_\Q.
\]
The image  of the sum $c_1+ \cdots + c_r$ in $K_0^{ Z_1\times_M Z_2 }(M)_\Q$  lies in the $d_1+d_2$ part of the coniveau filtration  by Theorem \ref{thm:GS_K-theory} and the commutativity of the diagram in Lemma \ref{lem:new intersection}, but if \eqref{G-filtration-pairing} holds  then   the image of each \emph{individual} $c_j$ in $K_0^{ Z_1\times_M Z_2 }(M)_\Q$ must also  lie in the $d_1+d_2$ part of the coniveau filtration.  
Even this weaker property seems quite subtle. (Note that the images of  $C_1,\ldots, C_r$ in $M$ may no longer be disjoint, leading to cancellation among the terms in  $c_1+ \cdots + c_r$ after pushforward to $M$.)
\end{remark}

\begin{remark}
One can define a coniveau-in-$M$ filtration on $G_0^\mathrm{naive}(Z)_\Q$ in exactly the same way as \eqref{G filtration}, but it is dubious that one should expect the analogue of \eqref{G-filtration-pairing} to hold with this naive definition. 
\end{remark}

The following weaker version of \eqref{G-filtration-pairing} is enough for our applications.

\begin{proposition}\label{prop:inductive coniveau}
Suppose   $Z_1\to M$ and $Z_2\to M$ are finite and unramified.
For any $d\ge 0$,  the pairing \eqref{new intersection} restricts to 
\[
F^d G_0(Z_1)_\Q \otimes F^1 G_0(Z_2)_\Q \map{\cap} F^{d+1}G_0(Z_1\times_M Z_2)_\Q .
\]
\end{proposition}

\begin{proof}
Assume first that 
\begin{equation}\label{simple dimension}
\mathrm{codim}_M(Z_1) \ge  d,\quad \mathrm{codim}_M(Z_2) \ge  1.
\end{equation}
If    $\mathrm{codim}_M(Z_1\times_M Z_2) \ge d+1$ then  
\[
F^{d+1} G_0(Z_1\times_M Z_2) = G_0(Z_1\times_M Z_2),
\]
 and there is nothing to prove. 
 Thus we assume further that 
 \[
 d = \mathrm{codim}_M(Z_1) =  \mathrm{codim}_M(Z_1\times_M Z_2)  .
 \]

\begin{lemma}\label{lem:coniveau key}
Suppose $C\subset Z_1\times_M Z_2$ is an irreducible component with $\mathrm{codim}_M(C)=d$, and with generic point $\eta$. For any pair of classes $(z_1,z_2)\in G_0(Z_1)_\Q\times G_0(Z_2)_\Q$, there is a Zariski   open substack  $U\subset Z_1\times_M Z_2$ containing $\eta$  for which 
   \[
   z_1\cap z_2 \in \mathrm{ker} \big(   G_0 ( Z_1 \times_M Z_2)  \to G_0 (U) \big).
 \]
\end{lemma}

\begin{proof}
Let $\bar{\eta} \to Z_1\times_M Z_2$ be a geometric point above $\eta$, and consider the commutative diagram of \'etale local rings 
\begin{equation}\label{etale local cap}
\xymatrix{
 &  {  \co^\et_{Z_1\times_M Z_2,\bar{\eta}}  }  \\
 {  \co^\et_{Z_1,\bar{\eta}} }  \ar[ur]& &  {   \co^\et_{Z_2,\bar{\eta}} }  \ar[ul]\\
 &   { \co_{M,\bar{\eta}}^\et   } \ar[ul]\ar[ur] 
}
\end{equation}
at $\bar{\eta}$.
As both $Z_1\to M$ and $Z_2\to M$ are finite and unramified, all of the morphisms in \eqref{etale local cap}  are surjective.  
 For any one of these local rings $R$, we abbreviate $G_0(R) = G_0(\Spec(R))$ for the Grothendieck group of 
finitely generated $R$-modules.  If $R\to S$ is any one of the four arrows in the above diagram, 
we similarly abbreviate
\[
K_0^S(R) = K_0^{\mathrm{Spec}(S) } ( \Spec(R) ).
\]
When $R = \co^{\et}_{M,\bar{\eta}}$,  Proposition \ref{prop:GtoK} provides a canonical isomorphism
\[
G_0(S)_\Q \iso K_0^S(R)_\Q.
\]

Consider the commutative diagram
\[
\xymatrix{
{  G_0 (Z_1)_\Q  \otimes   G_0 (Z_2)_\Q  }   \ar[rr]^{ \cap }     \ar[d]  & &  {  G_0 ( Z_1 \times_M Z_2)_\Q    }  \ar[d]  \\
{    G_0(   \co^\et_{Z_1,\bar{\eta}} )_\Q  \otimes  G_0(   \co^\et_{Z_2,\bar{\eta}} )_\Q    }   \ar[rr]  \ar[d]_{\iso}  &  & {  G_0(   \co^\et_{Z_1\times_M Z_2,\bar{\eta}}    )_\Q  }  \ar[d]_{\iso} \\
{    K_0^{    \co^\et_{Z_1,\bar{\eta}}   }  ( \co^\et_{M,\bar{\eta} })_\Q  \otimes  K_0^{    \co^\et_{Z_2,\bar{\eta}}   }  ( \co^\et_{M,\bar{\eta} })_\Q   }   \ar[rr] &  & {  K_0^{    \co^\et_{Z_1\times_M Z_2,\bar{\eta}}   } (    \co^\et_{M,\bar{\eta} })_\Q   }  
}
\]
in which  the middle  arrow is defined in exactly the same way as the top  pairing, and the bottom pairing is that of 
Theorem \ref{thm:GS_K-theory}.

The bottom pairing is multiplicative with respect to the coniveau filtration, but  
\begin{align*}
F^d K_0^{   \co^\et_{Z_1 ,\bar{\eta}}  } ( \co^\et_{M,\bar{\eta}} )_\Q  & = K_0^{   \co^\et_{Z_1 ,\bar{\eta}}  } ( \co^\et_{M,\bar{\eta}} )_\Q  \\
F^1 K_0^{   \co^\et_{Z_2 ,\bar{\eta}}  } ( \co^\et_{M,\bar{\eta}} )_\Q  & = K_0^{   \co^\et_{Z_2 ,\bar{\eta}}  } ( \co^\et_{M,\bar{\eta}} )_\Q  
\end{align*}
and, as $\mathrm{dim}( \co^\et_{M,\bar{\eta}} ) =d$ by hypothesis,
\[
F^{d+1} K_0^{   \co^\et_{ Z_1\times_M Z_2 ,\bar{\eta}}  } ( \co^\et_{M,\bar{\eta}} )_\Q   = 0.
\]
Thus the bottom horizontal arrow is trivial, and hence so is  the composition 
\[
G_0 (Z_1)_\Q  \otimes   G_0 (Z_2)_\Q \map{\cap} G_0 ( Z_1 \times_M Z_2)_\Q    \to 
K_0(\co^\et_{Z_1\times_MZ_2,\bar\eta})_\Q.
\]

As $\co^\et_{Z_1\times_MZ_2,\bar\eta}$ is an Artinian local ring, by d\'evissage (see~\cite[Lemma 7.3]{Gillet1984-tk})  and  Remark~\ref{rem:k_theory_punctual}, we have
\[
K_0(\co^\et_{Z_1\times_MZ_2,\bar\eta})_\Q \iso K_0(k(\bar\eta))_\Q \iso K_0(\eta)_\Q .
\]
Therefore the composition 
\[
G_0 (Z_1)_\Q  \otimes   G_0 (Z_2)_\Q \map{\cap} 
G_0 ( Z_1 \times_M Z_2)_\Q \to   G_0( \eta )_\Q \iso K_0(  \eta  )_\Q.
\]
is also trivial.

To finish, we only need to observe that
\[
\mathrm{colim}_{\eta\in U} G_0 ( U )_\Q \iso G_0( \eta )_\Q,
\]
where on the left hand side the colimit is over pullbacks of inclusions of open neighborhoods of $\eta$ in $Z_1\times_MZ_2$. 
This once again be checked on the level of sheaves of rational spectra, where it comes down to the fact that the exact category of coherent sheaves over a point of a scheme is equivalent to the colimit of the exact categories of coherent sheaves over a system of affine neighborhoods of the point; see for instance~\cite[\S 8.5]{MR217086}.
\end{proof}

We can now complete the proof of Proposition \ref{prop:inductive coniveau} under the assumption \eqref{simple dimension}.
By Lemma \ref{lem:coniveau key} there exists a Zariski open substack $U \subset Z_1\times_M Z_2$ such that 
\[
\mathrm{codim}_M ( (Z_1\times_M Z_2) \smallsetminus U)\ge d+1,
\]
and such that  $z_1\cap z_2$ lies in the kernel of  the second arrow in 
\[
G_0( ( Z_1 \times_M Z_2)  \smallsetminus U )_\Q \to G_0 ( Z_1 \times_M Z_2)_\Q   \to G_0 ( U )_\Q  .
\]
This sequence is exact~\cite[Lemma 7.4]{Gillet1984-tk},  and so   
\begin{align*}
z_1\cap z_2  &
\in \mathrm{Image} \big(   G_0( ( Z_1 \times_M Z_2) \smallsetminus U )_\Q \to G_0 ( Z_1 \times_M Z_2)_\Q  \big) \\
& \subset  F^{d+1} G_0 ( Z_1 \times_M Z_2 )_\Q.
\end{align*}

We now reduce  the general case to the case just proved.  Suppose we are given classes 
\[
z_1  \in F^d G_0(Z_1)_\Q ,\qquad z_2 \in F^1 G_0(Z_2)_\Q.
\]
By definition of the coniveau-in-$M$ filtration, there are closed substacks
$Y_1 \subset Z_1$ and $Y_2\subset Z_2$ such that 
\[
\mathrm{codim}_M(Y_1) \ge d ,\qquad \mathrm{codim}_M(Y_2) \ge 1,
\]
and $z_1 \otimes z_2$ lies in the image of the left vertical arrow in the commutative diagram
\[
\xymatrix{
{ F^d G_0(Y_1)_\Q \otimes    F^1 G_0(Y_2)_\Q } \ar@{=}[d] \\
{   G_0(Y_1)_\Q \otimes    G_0(Y_2)_\Q }  \ar[rr]^{\cap}\ar[d] & & { G_0(Y_1\times_M Y_2)_\Q } \ar[d] \\
{ F^d G_0(Z_1)_\Q \otimes F^1 G_0(Z_2)_\Q }  \ar[rr]^{\cap}  &  & { G_0(Z_1\times_M Z_2)_\Q }. 
}
\]
The  special case of the proposition proved above   shows that the top horizontal arrow  takes values in 
$F^{d+1}G_0(Y_1\times_M Y_2)_\Q$, and 
Proposition \ref{prop:inductive coniveau} follows immediately.
\end{proof}








\section{Quadratic lattices}


This appendix contains some technical results on the existence of isometric embeddings of quadratic lattices.


\subsection{Embeddings of hyperbolic planes}


Let $L$ be a quadratic lattice over $\Z$. That is to say, a free $\Z$-module of finite rank endowed with a $\Z$-valued quadratic form such that $L\otimes \Q$ is nondegenerate.

\begin{lemma}\label{lem:local-global-hyperbolic}
Suppose  $L^\beef$ is an indefinite quadratic lattice such that
\begin{enumerate}
\item
for every prime $p$ there exists an isometric embedding
\[
\alpha_p : L\otimes\Z_p \to L^\beef \otimes \Z_p,
\]
\item
$ \mathrm{rank}_\Z(L^\beef) \ge \mathrm{rank}_\Z(L)+4$.
\end{enumerate}
If there exists an isometric embedding
$
a : L \otimes \Q \to L^\beef \otimes \Q 
$
such that
\begin{equation}\label{loc embedding match}
a(L\otimes \Z_p) = \alpha_p(L\otimes \Z_p)
\end{equation}
 for all but finitely many primes $p$, then  $a$ can be chosen so that \eqref{loc embedding match} holds for every prime $p$.
 \end{lemma}

\begin{proof}
As all embeddings 
$
L \otimes \Q_p  \to L^\beef \otimes \Q_p
$
 lie in a single $\mathrm{SO}( L^\beef \otimes \Q_p)$-orbit, there exists a $g  \in \mathrm{SO}( L^\beef \otimes \A_f)$ such that 
\begin{equation}\label{spinor adjust}
g_p  \cdot a (L\otimes \Z_p ) =\alpha_p (L^\beef \otimes\Z_p )  .
\end{equation}

By assumption, the orthogonal complement  
\[
W \define a(L\otimes \Q)^\perp \subset L^\beef \otimes \Q
\]
has dimension $\ge 4$.  As a quadratic space over $\Q_p$ of dimension $\ge 4$ represents every element of $\Q_p^\times$, the spinor norm
\[
\mathrm{SO}(W\otimes \Q_p) \to \Q_p^\times / (\Q_p^\times)^2 
\]
is surjective.  Multiplying  $g$ by a suitable element of 
$\mathrm{SO}(W\otimes \A_f ) \subset \mathrm{SO}( L^\beef \otimes \A_f)$, which does not change the relation \eqref{spinor adjust},  we may  assume that  $g$ has trivial  spinor norm and fix a lift to 
$
g \in \mathrm{Spin}( L^\beef \otimes \A_f).
$

Using strong approximation for the (simply connected) spin group we may replace this lift by a
$g \in \mathrm{Spin}(  L^\beef \otimes \Q)$ in such a way that   \eqref{spinor adjust} still holds, and the resulting embedding
$
g  a : L\otimes\Q \to L^\beef \otimes \Q
$
has the desired properties.
\end{proof}

Let $H$ be the hyperbolic plane over $\Z$.
In other words,   $H=\Z \ell \oplus \Z \ell_*$  where $\ell$ and $\ell_*$ are isotropic vectors with $[\ell ,\ell_*]=1$. 
The following result was used in the proof of Proposition \ref{prop:ZTmu geom conn}.

 \begin{proposition}\label{prop:hyperbolic embeddings}
 Let $\gamma \ge 0$ be the minimal number of elements needed to generate the finite abelian group $L^\vee/L$.
If $L$ is indefinite with
\[
\mathrm{rank}_\Z(L) \ge 2\gamma +6,
\]
then there exists an isometric embedding $H \to L$. 
\end{proposition}

\begin{proof}
Using Lemma \ref{lem:local-global-hyperbolic} and the Hasse-Minkowski theorem,
 we are reduced to proving the existence of 
an isometric embedding $H\otimes\Z_p \to L \otimes\Z_p$  for every prime $p$.

Using the classification of quadratic lattices over $\Z_p$, one can find an orthogonal decomposition
\[
L \otimes\Z_p \iso J_1\oplus \cdots \oplus J_t
\]
in such a way that  each $J_i$  has $\Z_p$-rank either $1$ or $2$.
Each summand satisfies $J_i \subset J_i^\vee$, and we collect together into one self-dual $\Z_p$-quadratic space $K$ those summands for which equality holds.  This gives a decomposition
\[
L \otimes\Z_p \iso J_1\oplus \cdots \oplus J_s \oplus K
\]
in such a way that  $J_i \subsetneq J_i^\vee$  and $K=K^\vee$.

Equating the $\Z_p$-ranks of both sides shows that 
\[
\mathrm{rank}_{\Z}(L)  = 
\mathrm{rank}_{\Z_p}(J_1\oplus \cdots\oplus J_s)+\mathrm{rank}_{\Z_p}(K) \le 2 s+\mathrm{rank}_{\Z_p}(K) .
\]
On the other hand,  the definition of $\gamma$ implies the existence of  a surjective $\Z_p$-module map
\[
\Z_p^\gamma \to      ( L^\vee / L ) \otimes\Z_p  \iso  \bigoplus_{i=1}^s J_i^\vee / J_i, 
\]
which in turn implies $s\le \gamma$. Combining these gives the second inequality in
\[
2\gamma+6\le \mathrm{rank}_\Z(L)  \le 2\gamma +\mathrm{rank}_{\Z_p}(K),
\]
and so $\mathrm{rank}_{\Z_p}(K) \ge 6$.

As every quadratic space over $\Q_p$ of dimension at least $5$ contains an isotropic vector,  there exists an isometric embedding
\[
H  \otimes \Q_p \to K\otimes\Q_p.
\]
Certainly the image of $H\otimes\Z_p$ is contained in some maximal lattice (in the sense of Definition \ref{def:max lattice}) in  $K\otimes \Q_p$, and it is a theorem of Eichler  that all maximal lattices in $K\otimes\Q_p$ are isometric.    
Thus   $H \otimes \Z_p$ can be embedded isometrically into \emph{any} maximal lattice in $K\otimes \Q_p$, including  $K$ itself  (which is self-dual, hence maximal).
In particular, $H \otimes \Z_p$ embeds isometrically  into $L\otimes \Z_p$.
\end{proof}


\subsection{Embeddings into self-dual lattices}


As above, let $H$ be the hyperbolic plane over $\Z$.

\begin{lemma}
\label{lem:split_everywhere}
If   $r,s \in \Z_{\ge 0}$ satisfy  $r\equiv s\pmod{8}$, then  there exists a quadratic space $V$ over $\Q$ of signature $(r,s)$ such that  
 \[
 V\otimes\Q_p\iso (H\otimes\Q_p)^{ \frac{r+s}{2}} 
 \]
  for every prime $p$.
\end{lemma}

\begin{proof}
This is an application of the classification of quadratic forms over $\Q$, as found in  ~\cite[Theorem 28.9]{ShimuraQuadratic}. 
\end{proof}

The following result is needed to make sense of Definition \ref{defn:rL}.  

\begin{proposition}\label{prop:r genus appendix}
Let $L$ be a quadratic lattice over $\Z$ of signature $(n,m)$, with $m>0$.
There exist an integer $r\geq 1$, a self-dual quadratic lattice $L^\beef$ of signature $(n+r,m)$, and an isometric embedding $L \to L^\beef$ identifying $L$ with a $\Z$-module direct summand of $L^\beef$. 
\end{proposition}

\begin{proof}
First, we claim that for every prime $p$ and every quadratic lattice $J$ over $\Z_p$ there is an isometric embedding
\[
J \to (H \otimes \Z_p)^{\mathrm{rank}_{\Z_p}(J)}
\]
realizing the source as a $\Z_p$-module direct summand of the target.
As in the proof of Proposition~\ref{prop:hyperbolic embeddings}, one can write $J$ as an orthogonal direct sum of quadratic lattices of rank $\le 2$, so we may assume that  $\mathrm{rank}_{\Z_p}(J)\le 2$.
For rank $1$ lattices, this just amounts to the fact that, for every $m\in\Z_p$, there exists a basis  $v,w \in H\otimes\Z_p$ with $Q(v) = m$.  If  $\mathrm{rank}_{\Z+p}(J)=2$  and $J$ is diagonalizable (which is always the case if $p>2$), we are immediately  reduced to the rank one case.  This leaves us with the case where $p=2$ and $J$ is  non-diagonalizable of rank $2$.  
In this case there is a basis $v,w \in J$ such that 
\[
Q(v) = 2^ka, \quad Q(w) = 2^kb,  \quad [v,w] = 2^kc
\]
 for some $a,b,c\in\Z_2^\times$ and $k\geq 0$. Suppose that $e_1,f_1,e_2,f_2$ is a standard hyperbolic basis for $(H\otimes\Z_2)^2$, and set
\[
v' = e_1 + 2^kaf_1 \quad \mbox{and} \quad w' = a^{-1}ce_1 + e_2+2^kbf_2.
\]
One can easily check that $v\mapsto v'$ and  $w\mapsto w'$ defines an isometric embedding 
$
J \to (H\otimes \Z_2)^2
$
 onto a direct summand.

By the paragraph above, for every prime $p$  and every $r\ge 0$, there exists an isometric embedding
\begin{equation}\label{local primitive summand}
 L\otimes \Z_p \to (H\otimes \Z_p)^{m+n+r}
\end{equation}
realizing the source as a $\Z_p$-module direct summand of the target.  
Choosing $r\ge 4$ so that $n+r \equiv m \pmod{8}$, Lemma \ref{lem:split_everywhere} allows us to choose a quadratic space $V^\beef$ over $\Q$ of signature $(n+r,m)$ such that 
\[
V^\beef \otimes \Q_p \iso (H\otimes \Q_p)^{m+n+r}
\]
for every prime $p$.  
By the Hasse-Minkowski theorem, there exists an isometric embedding
$
a : L \otimes \Q \to V^\beef.
$

Let $L^\beef \subset V^\beef$ be any maximal lattice containing $a(L)$. 
  By Eichler's theorem that all maximal lattices in a $\Q_p$-quadratic space are isometric, 
  \[
   L^\beef \otimes \Z_p \iso  (H\otimes \Z_p)^{m+n+r}
   \] 
   for every prime $p$.  In particular, $L^\beef$ is self-dual.
  For all but finitely many primes $p$, the embedding 
  \[
  \alpha_p : L \otimes \Z_p \to L^\beef \otimes \Z_p
  \]
  induced by $a$ realizes the source as a $\Z_p$-module direct summand of the target. 
   For the remaining primes, we take $\alpha_p$ to be the composition
  \[
   L\otimes \Z_p \map{\eqref{local primitive summand}}  (H\otimes \Z_p)^{m+n+r} \iso L^\beef \otimes \Z_p .
  \]
  By Lemma \ref{lem:local-global-hyperbolic}, there exists an isometric embedding $b:L \to L^\beef$ such that 
  $b(L\otimes \Z_p) = \alpha_p(  L \otimes \Z_p)$ for all $p$, from which it follows that $b(L)$ is a $\Z$-module direct summand of $L^\beef$.
\end{proof}

\bibliographystyle{amsalpha}

\providecommand{\bysame}{\leavevmode\hbox to3em{\hrulefill}\thinspace}
\providecommand{\MR}{\relax\ifhmode\unskip\space\fi MR }
\providecommand{\MRhref}[2]{%
  \href{http://www.ams.org/mathscinet-getitem?mr=#1}{#2}
}
\providecommand{\href}[2]{#2}

\end{document}